\documentclass[12pt, a4paper, parskip=half, abstracton]{scrartcl}
\pdfoutput=1

\usepackage{array}
\usepackage{marginnote}
\usepackage{xcolor}
\usepackage{amscd,amssymb,amsfonts,amsmath,latexsym,amsthm}
\usepackage{hyperref}
\usepackage[all,cmtip]{xy}
\usepackage{mathrsfs}
\usepackage{graphicx}
\usepackage{bm} 
\usepackage{enumitem} 

\usepackage{etoolbox}

\def\hB{\hspace*{\fill}$\qed$}

\usepackage[nottoc]{tocbibind} 

\usepackage{defs_pp1}
\usepackage{slashed}
\usepackage[utf8]{inputenc}
\usepackage{microtype}
\usepackage[english]{babel}
\usepackage{mathtools}
\usepackage{esvect}

\title{$\mathclap{\text{Topological equivariant coarse $\boldsymbol{K}$-homology}}$
}
\author{
Ulrich Bunke\thanks{Fakult{\"a}t f{\"u}r Mathematik,
Universit{\"a}t Regensburg,
93040 Regensburg,
GERMANY\newline
\href{mailto:ulrich.bunke@mathematik.uni-regensburg.de}{ulrich.bunke@mathematik.uni-regensburg.de}}  
\and
Alexander Engel\thanks{Universit\"at Greifswald,
Walther-Rathenau-Strasse 47,
17489 Greifswald,
GERMANY\newline
\href{mailto:alexander.engel@uni-greifswald.de}{alexander.engel@uni-greifswald.de}}}

\numberwithin{equation}{section}
\newtheorem{theorem}{Theorem}[section] 
\newtheorem{prop}[theorem]{Proposition}
\newtheorem{lem}[theorem]{Lemma}

\newtheorem{ddd}[theorem]{Definition}
\newtheorem{kor}[theorem]{Corollary}
\newtheorem{ass}[theorem]{Assumption}

\theoremstyle{remark}
\theoremstyle{definition}

\newtheorem{ex}[theorem]{Example}
\newtheorem{rem}[theorem]{Remark}

 \newcommand{\ndaCcat}{C^{*}\mathbf{Cat}^{\mathrm{nu}}_{\mathrm{\infty add,ndeg}}}
\newcommand{\ndCcat}{C^{*}\mathbf{Cat}^{\mathrm{nu}}_{\mathrm{ndeg}}}
\newcommand{\ndeCcat}{C^{*}\mathbf{Cat}^{\mathrm{nu}}_{\mathrm{eadd,ndeg}}}
\newcommand{\ndasCcat}{C^{*}\mathbf{Cat}^{\mathrm{nu}}_{\mathrm{sadd,ndeg}}}

\newcommand{\Kast}{\mathrm{K}^{\mathrm{C}^{*}}}
\newcommand{\bCg}{\mathbf{ C}^{G}}
\newcommand{\bCgc}{\mathbf{ C}_{\mathrm{pt}}^{G}}
\newcommand{\bCgsm}{\mathbf{ C}^{G}_{\mathrm{sm}}}
\newcommand{\bCgsmc}{\mathbf{ C}_{\mathrm{lf}}^{G }}

\newcommand{\pt}{\mathrm{pt}}

\newcommand{\LF}{\mathrm{LF}}

\newcommand{\bCgtsm}{ \mathbf{ \bar C}_{\mathrm{sm}}^{G,\mathrm{ctr}}}
   \newcommand{\Cgtsm}{ \mathbf{ C}_{\mathrm{sm}}^{G,\mathrm{ctr}}}

   \newcommand{\bCgtsmc}{ \mathbf{\bar C}_{\mathrm{lf}}^{G,\mathrm{ctr}}}
     \newcommand{\bCgtsmctr}{ \mathbf{\bar C}_{\mathrm{lf},\mathrm{tr}}^{G,\mathrm{ctr}}}
     \newcommand{\bCGtsmc}{ \mathbf{\bar C}_{\mathrm{lf},G}^{ \mathrm{ctr}}}
  
      \newcommand{\Cgtsmc}{ \mathbf{ C}_{\mathrm{lf}}^{G,\mathrm{ctr}}}
       \newcommand{\bCtsmc}{ \mathbf{\bar C}_{\mathrm{lf}}^{\mathrm{ctr}}}
          
          \newcommand{\bCtsm}{\mathbf{\bar C}^{\mathrm{ctr}}_{\mathrm{sm}}}
             \newcommand{\btCtsm}{\mathbf{\tilde{  \bar{C}}}^{\mathrm{ctr}}_{\mathrm{sm}}}
                \newcommand{\btCtsmc}{\mathbf{\tilde { \bar{C}}}^{\mathrm{ctr}}_{\mathrm{lf}}}

  \newcommand{\bCGtsm}{ \mathbf{\bar C}_{\mathrm{sm},G}^{ \mathrm{ctr}}}

\newcommand{\bChtsmc}{\mathbf{\bar C}_{\mathrm{lf}}^{H,\mathrm{ctr}}}

\newcommand{\Homol}{\mathrm{Hg}}

\newcommand{\bJ}{\mathbf{J}}

\newcommand{\All}{\mathbf{All}}

\newcommand{\Yo}{\mathrm{Yo}}

\newcommand{\Res}{\mathtt{Res}}

\newcommand{\Orb}{\mathbf{Orb}}

\newcommand{\Hilb}{\mathbf{Hilb}}

\newcommand{\BC}{\mathbf{BornCoarse}}

\newcommand{\Fin}{\mathbf{Fin}}

\newcommand{\Ob}{\mathrm{Ob}}

\newcommand{\bB}{{\mathbf{B}}}

\newcommand{\bL}{\mathbf{L}}
\newcommand{\bM}{\mathbf{M}}

\newcommand{\cP}{\mathcal{P}}

\renewcommand{\Proj}{\mathrm{Proj}}

\newcommand{\bA}{{\mathbf{A}}}

\newcommand{\cO}{{\mathcal{O}}}
\newcommand{\cU}{{\mathcal{U}}}
\newcommand{\cY}{{\mathcal{Y}}}

 \newcommand{\Cat}{{\mathbf{Cat}}}

\newcommand{\KX}{K\!\mathcal{X}}

\newcommand{\Idem}{\mathrm{Idem}}

\newcommand{\Born}{\mathbf{Born}}
\newcommand{\Coarse}{\mathbf{Coarse}}

\newcommand{\IC}{\mathbb{C}}

\newcommand{\Ccat}{{\mathbf{C}^{\ast}\mathbf{Cat}}}
\newcommand{\Calg}{{\mathbf{C}^{\ast}\mathbf{Alg}}}

\newcommand{\bP}{\mathbf{P}}

\newcommand{\op}{\mathrm{op}}

\newcommand{\nCcat}{C^{*}\mathbf{Cat}^{\mathrm{nu}}}
\renewcommand{\Ccat}{C^{*}\mathbf{Cat}}
\newcommand{\alg}{\mathrm{alg}}
\newcommand{\nClincat}{\mathbf{{}^{*}\Cat^{\mathrm{nu}}_{\C}}}

\newcommand{\Clincat}{\mathbf{{}^{*}\Cat_{\C}}}
\renewcommand{\Calg}{C^{*}\mathbf{Alg}}
\newcommand{\nCalg}{C^{*}\mathbf{Alg}^{\mathbf{nu}}}
\renewcommand{\nCalg}{C^{*}\mathbf{Alg}^{\mathrm{nu}}}

 \newcommand{\Kcat}{\mathrm{K^{C^{*}Cat}}}
\newcommand{\Ass}{\mathrm{Asmbl}}

\begin{document} 
\maketitle\begin{abstract}
For a $C^{*}$-category with a strict $G$-action we construct examples of equivariant coarse homology theories.  
To this end we first introduce versions of Roe categories of objects  in $C^{*}$-categories which are  controlled  over bornological coarse spaces,
and then apply a homological functor. 
These  equivariant coarse homology theories are then employed to verify that certain functors on the orbit category are CP-functors. This fact has consequences for the injectivity of assembly maps.
\end{abstract}
 \tableofcontents
\setcounter{tocdepth}{5}

 
 \newcommand{\sep}{\mathrm{sep}}
 \newcommand{\free}{\mathrm{free}}
  \newcommand{\fg}{\mathrm{fg}}
\newcommand{\fin}{\mathrm{fin}}

\section{Introduction}

Let $G$ be a group.
In this paper we construct and study  several  versions    of $G$-equivariant coarse  homology theories listed in  \eqref{qefefewfefqff} below which are equivariant analogs of the coarse topological $K$-homology $K\!\cX$ constructed in \cite{buen}. They depend on the choices   of a coefficient    $C^{*}$-category  $  \bC$ with $G$-action and 
 a homological functor $\Homol \colon \nCcat\to \bM$. The main example of a homological functor  is the $K$-theory functor for  possibly non-unital $C^{*}$-categories $\Kcat \colon \nCcat\to \Sp$.

The present paper focusses on those details of the construction of the homology theories which are related {to} coarse geometry.  
 For the   results about  $C^{*}$-categories   going into the construction and  the notion of a homological functor we refer to  the companion papers \cite{cank}, \cite{crosscat} which were partly written to provide the background for the present paper. 
 In addition,  the main definitions and facts will be recalled in Section \ref{qergijfioqrfqwedwedqwdq}.

{One of the main applications} of the  coarse homology theory   $\Kcat \bC\cX^{G}_{c}$ introduced  in the present paper is the verification of the CP-condition \cite[Def.~1.8]{desc} for the functors  on the orbit category of $G$ 
 defined in  \cite[Def.\ 19.12]{cank}. As shown in  \cite{desc}
   this CP-condition implies split-injectivity results for assembly maps. This application will be discussed  in  some detail in this introduction further below.
 For  further applications to index theory we refer to \cite{indexclass,bl}. 
 The results of the present paper are the crucial technical input for the discussion of equivariant Paschke duality  in \cite{bel-paschke}.

In the present paper we work in the setup of equivariant coarse homotopy theory  as developed in  \cite{equicoarse} which generalizes the non-equivariant case from  \cite{buen}.  Its basic category is  the category 
  $G\BC$   of $G$-bornological coarse spaces introduced in   Definition \ref{gjwerogijwoergwergwergweg}. 
  By   Definition \ref{wefuihqfwefwffqwefefwq} an equivariant coarse homology theory   is a functor
$$E\colon G\BC\to \bM$$ with a stable target
 $\infty$-category $\bM$ (e.g.\ the category $\Sp$ of spectra),
 which sends coarse equivalences to equivalences  and annililates flasque spaces. It must furthermore satisfy
an excisiveness condition and the  condition of $u$-continuity. 
In Section  \ref{qriogfhjqowerfweqfqewfewfqewf} we give a complete and selfcontained  introduction to  the notions from coarse geometry and homotopy theory of bornological coarse spaces  which are relevant for the present paper.
 For different   setups and axioms for coarse homology theories see \cite{MR1834777,Heiss:2019aa,wulff_axioms}.

The output of our construction are four   functors
\begin{equation}\label{qefefewfefqff}
\Homol \bC \cX^{G}, \Homol \bC \cX_{c}^{G},\Homol\bC\cX_{G},  \Homol\bC\cX_{c,G} \colon G\BC\to \bM
\end{equation}
which are all equivariant coarse homology theories, see Sections \ref{qergiuhquifewdqewdewdqwedqewd} and \ref{ewtiguhiufhriuewdwedqwdewd}.  
Our construction is basically  a generalization of the construction of the non-equivariant coarse $K$-homology $\KX \colon \BC\to \Sp$ in \cite{buen} to the equivariant case and for general coefficients.  Thereby we follow the structure of
the construction of coarse homology theories via controlled objects  in left exact $\infty$-categories in  \cite{unik}.

%
%
One of the main achievements of \cite{buen} is the construction of a refinement of Roe's coarse $K$-homology to a spectrum-valued coarse homology theory defined the category $\BC$ of bornological coarse spaces.
In  \cite{buen,ass} we observed that an interpretion of  the coarse assembly map as a natural transformation between such spectrum-valued coarse homology theories allows for a considerable  simplification of the formal structure of  the proof of the coarse Baum-Connes conjecture  given in \cite{nw1}. 
%
%
The equivariant version of the coarse assembly map (also known as forget-control map) turned out to be a tool to study a different sort of assembly maps. Given a functor $$M \colon G\Orb\to \bM$$ with a cocomplete target $\bM$ and two families of subgroups $\cF$ and $\cF'$ of $G$ such that $\cF\subseteq \cF'$ we  have the assembly map 
 \begin{equation}\label{ewrviejv98vffdvsfdvfsdv}
\Ass_{\cF,M}^{\cF'} \colon \colim_{G_{\cF}\Orb} M\to \colim_{G_{\cF'}\Orb} M\, ,
\end{equation}
where $G_{\cF}\Orb$ is the full subcategory of $G\Orb$ of orbits with stabilizers in $\cF$.
One of the driving force   for  developing an equivariant coarse homotopy  theory   was the study of injectivity and isomorphism results for assembly maps  \eqref{ewrviejv98vffdvsfdvfsdv}  for various choices of the functor $M$ and families of subgroups.
  Initially this development was a part of 
 controlled topology
\cite{calped}, see also  \cite{blr,MR3598160,RTY}. The connection with 
coarse geometry    is  explained in \cite{higson_pedersen_roe}.

We say that 
  a coarse homology theory $E \colon G\BC\to \bM$ extends the functor $M$ if  for every $S$ in $G\Orb$  we have a natural equivalence  \begin{equation}\label{wefweqfqwfqwefqff}
M(S)\simeq E(S_{min,max}\otimes G_{can,min})\, ,
\end{equation}  where
 $S_{min,max}$ is the  $G$-bornological coarse space  given by the $G$-orbit $S$ with the minimal coarse structure and the maximal bornology, and $G_{can,min}$ is the group $G$ with the canonical coarse structure and the minimal bornology.   
Such an extension of $M$ to a coarse homology theory is one of the basic 
 ingredients of the approach of  \cite{blr} (isomorphism result, see also \cite{fj}) or \cite{RTY}
 (injectivity result). In \cite[Def.\ 1.8]{desc} we axiomatized the additional conditions on $E$ needed to show that $\Ass_{\Fin,M}^{\All}$ is split-injective by the notion of a CP-functor.
 
  Hence more examples of coarse homology theories $E$  leading to CP-functors $M$ via \eqref{wefweqfqwfqwefqff}  lead to more examples of functors on the orbit category for which one can show  split  injectivity of the assembly map by the method of  \cite{desc}. Examples  are    coarse ordinary homology and coarse algebraic $K$-theory with coefficients in an additive category with strict $G$-action  \cite{equicoarse} (the CP-conditions are verified in  \cite{desc}), 
 the  equivariant  coarse version of   Waldhausen algebraic $K$-theory of spaces 
\cite{Bunke:aa}, and the  equivariant coarse $K$-homology with coefficients in a left-exact $\infty$-category with $G$-action constructed in 
  \cite{unik}.  The present paper contributes  $\Kcat \bC\cX^{G}_{c}$ as a further example. We refer to Theorem \ref{wergiojerfqerwfqwefwefewfqef} for a precise statement.

 For a $C^{*}$-algebra $A$ with $G$-action
 we can consider the $C^{*}$-category $  \bC \coloneqq  \Hilb_{c}( A)$ of   Hilbert-$A$-modules and compact operators with the $G$-action described in 
 Example \ref{dqedqwdqwd1234frgfrfwerger}. We then consider the functor 
 $$\Kcat \bC_{G,r} \colon G\Orb\to \Sp$$ introduced in Definition \ref{qergijoqergwrefweqfwqefwedewdqewd}. By   Proposition \ref{qergioqefweqwecqcasdc} this functor is extended (in the sense explained above) by   the
  coarse homology theory $\Kcat\bC\cX^{G}_{c}$.  
   By Theorem \ref{wergiojerfqerwfqwefwefewfqef} it is a CP-functor. 
 In particular the split injectivity results from 
  \cite{desc} (see also \cite[Sec.\ 6.5]{unik} for a quick review) apply to $ \Kcat \bC_{G,r}$. 
 But note the following remark which explains that the above  can not extend the class of groups $G$ for which  split-injectivity results are known.
     
     \begin{rem}
     If $A$ is unital and has the trivial  $G $-action, then by 
   \cite[Prop. 19.18]{cank}    the functor $\Kcat \bC_{G,r}$ can be identified with 
 the functor $K^{\mathrm{DL},G}_{A} \colon G\Orb\to \Sp$ introduced by Davis--L\"uck  \cite{davis_lueck}.\footnote{To be precise, in \cite{davis_lueck} only the case $A=\C$ is considered but the generalization to general unital $C^{*}$-algebras $A$ is straightforward.} We can thus consider  $\Kcat \bC_{G,r}$ as a 
 generalization of the  Davis--L\"uck construction to $C^{*}$-algebras with  non-trivial $G$-action. An independent construction of such a functor 
   $K^{\mathrm{DL},G}_{A}$  for  $C^{*}$-algebras with $G$-action $  A$ has been given recently by Kranz \cite{kranz}, see also \cite{bel-paschke}.
   Note also that for every $C^{*}$-category $\bC$ with $G$-action there exists a suitable $C^{*}$-algebra with $G$-action $A$  and an equivalence
   $\Kcat \bC_{G,r}\simeq K^{\mathrm{DL},G}_{A}$. Hence taking $C^{*}$-categories instead of $C^{*}$-algebras  with $G$-action as coefficients does not provide additional equivalence  classes of functors on the orbit category.
 

%

By  passing to homotopy groups  the assembly map $\Ass_{\Fin,K^{\mathrm{DL},G}_{ A} }^{\All}$   becomes isomorphic   to the usual   Baum--Connes assembly map
\begin{equation}\label{rfkqfoewdwdqdewdqdewwqdedqewd}
\mu^{BC}_{A,*} \colon  RK_{*}^{G}(E_{\Fin}G,A)\to K_{*}^{C^{*}}(A\rtimes_{r}G) 
\end{equation}
defined via Kasparov's $K\!K$-theory, 
where the domain denotes the equivariant,  locally finite $K$-homology of $E_{\Fin}G$ with coefficients in $  A$. This is shown (independently and with completely different methods) in \cite{kranz} 
and   in
  \cite{bel-paschke}. 

In \cite{skandalis_tu_yu} it is shown that the   Baum--Connes assembly map  $\mu^{BC}_{A,*}$   is  {split-}injective provided the countable group $G$ with its canonical coarse structure $G_{can}$ admits a coarse embedding into a Hilbert space.  In the following we explain why this result is more general than what  could possibly be shown by specializing the results from \cite{desc} to $M$. To this end we note that 
the main  assumption on $G$ going into the split-injectivity results in \cite{desc} is 
 finite decomposition complexity (FDC) for  $G_{can}$. 
 In   \cite{gty} it is shown that this   FDC  condition implies that  $G_{can}$ 
admits  a coarse embedding into a Hilbert space. Hence
all  split-injectivity results  obtained by specializing \cite{desc} to equivariant topolological $K$-homology are already 
  covered by  \cite{skandalis_tu_yu}.

  But note that the method of the proof in   \cite{skandalis_tu_yu}   {is} specific to equivariant topological $K$-theory with coefficients in a $C^{*}$-algebra, while  the results from   \cite{desc} apply to a variety of other functors on the orbit category which are not related with topological $K$-theory at all.
\hB
\end{rem}

We now describe the structure of the present paper in some detail. 
The input data for our construction is  a $C^{*}$-category with $G$-action $\bC$.  In Section~\ref{rgijotgweregwrgrgwrg}
we associate to every $G$-set $X$ a $C^{*}$-category  ${(\bM\bC)}^{G}(X)$ of $X$-controlled $G$-objects in the multiplier category  $\bM\bC$ of $\bC$. If $X$ has the structure of a $G$-bornological coarse space, then in Section \ref{fiughiufvfdcasdcscdscca} we use the additional structures on $X$ in order to define smallness conditions on the objects of ${(\bM\bC)}^{G}(X)$ and control conditions on the morphisms. As a result we obtain
versions of Roe categories $\bar \bC^{G,\mathrm{ctr}}_{?}(X)$. In Section
\ref{eriugheriuvfvfvfvfvsvfvfdvsfs} we show the basic properties
of the functors $X\mapsto \bar \bC^{G,\mathrm{ctr}}_{?}(X)$ which go into the verification of the axioms of a coarse homology theory later.
In Sections \ref{qergiuhquifewdqewdewdqwedqewd} and \ref{ewtiguhiufhriuewdwedqwdewd} we  define  our four versions of coarse homology theories listed in \eqref{qefefewfefqff} and show that they have the required properties.
In Section \ref{qwrgioeqrgefefvsfvfdvfds} we calculate the values of these functors on certain $G$-bornological spaces. These calculations are relevant for identifying the corresponding functors   $M$ on $G\Orb$ defined via \eqref{wefweqfqwfqwefqff}.
The Sections \ref{qeriughqifqwcqeqc} and \ref{3qrgiuhiurlferfvrefvfs} prepare the proof of Theorem \ref{wergiojerfqerwfqwefwefewfqef} in Section \ref{sec_CPfunctors} by showing that $\Kcat\bC\cX^{G}_{c}$ admit transfers (an additional structure) and is strongly additive  (an additional property).
The Section \ref{eoigjortwgwegergwegreg} is of independent interest since it relates equivariant coarse homology theories for different groups via transfer and induction. The results of {that} section will be used in \cite{bl}.

{\em Acknowledgement: U.B.\ was supported by the SFB 1085 (Higher Invariants) funded by the Deutsche Forschungsgemeinschaft (DFG).}

\section{Bornological coarse spaces and equivariant coarse homology theories} \label{qriogfhjqowerfweqfqewfewfqewf}

In this section we recall the category of $G$-bornological  {coarse} spaces $G\BC$ and the notion of an equivariant coarse homology 
theory. The reference  for the equivariant case is  \cite{equicoarse}, but see also \cite{buen} for the non-equivariant case.

In order to fix size issues we fix four Grothendieck universes. Their elements are called very small sets, small sets, large sets, and very large sets,  respectively.

Let $G$ be a very small group. We let $\Set $ denote the small category of very small sets. We further let
$$G\Set \coloneqq \Fun(BG,\Set)$$ be the  small category of very small sets with $G$-action and equivariant maps.

If $X$ is a set, then $\cP_{X}$ denotes the power set of $X$.  The elements of $\cP_{X\times X}$ can be composed or inverted in the sense of correspondences (multi-valued maps). For $B$ in $\cP_{X}$ and $U$ in $\cP_{X\times X}$ we let $U[B]$ in $\cP_{X}$ be the result of applying the correspondence $U$ to $B$.

If 
  $X$ is in $G\Set$, then  $G$ acts on $\cP_{X}$.    
\begin{ddd}\label{trbertheheht}
A $G$-coarse structure on $X$ is a subset $\cC_{X}$ of $\cP_{X\times X}$ satisfying the following conditions:
\begin{enumerate}
\item $\cC_{X}$ is $G$-invariant.
\item $\diag(X)\in \cC_{X}$.
\item \label{qerighioergergwgergwergwergwerg}$\cC$ is closed under forming subsets, finite unions, inverses   and compositions.
\item \label{igwoegwergergwrgrg}The sub-poset (w.r.t.\ the inclusion relation) of $G$-invariants $\cC_{X}^{G}$ is cofinal in $\cC_{X}$. \qedhere
\end{enumerate}
\end{ddd}
The elements of $\cC_X$ are called coarse entourages.

  A subset $A$ of $\cP_{X\times X}$ generates a smallest $G$-coarse structure $\cC\langle A\rangle$ on $X$ containing $A$.

We consider  two $G$-sets with $G$-coarse structures  $(X,\cC_{X})$ and $(X^{\prime},\cC_{X^{\prime}})$, and a morphism   $f\colon X\to X^{\prime}$  in $ G\Set$. 

\begin{ddd}
	The map $f$ is  controlled if $(f\times f)(\cC_{X})\subseteq \cC_{X^{\prime}}$.
\end{ddd}

\begin{ddd}\label{thiowhwfgwrgwergwreg}
A $G$-coarse space is a pair $(X,\cC_{X})$  of  a $G$-set $X$ with a $G$-coarse structure $\cC_{X}$. A morphism between $G$-coarse spaces is an equivariant controlled map.
\end{ddd}

We obtain the category $G\Coarse$  of $G$-coarse spaces and  equivariant, controlled maps. 
Usually we will shorten the notation and denote $G$-coarse spaces just by  the symbol $X$.
\begin{ex}\label{rwtekoierfrefegd}
Let $(X,d)$ be a metric space equipped with an isometric $G$-action.
The metric $G$-coarse structure $\cC_{d}$ on $X$ is generated by the family of entourages $(U_{r})_{r\in [0,\infty)}$, where $U_{r} \coloneqq \{(x,x')\in X\times X \:|\: d(x,x')\le r \}$.
\hB
\end{ex}
  
  \begin{rem}\label{egiojweogergregwefwerfwrefw}
 If one drops Condition~\ref{igwoegwergergwrgrg} in Definition \ref{trbertheheht}, then one gets the characterization of a coarse space with a $G$-action, i.e., of an object of $\Fun(BG,\Coarse)$.

For $\lambda$ in $\R\setminus \{0\}$ we can consider the group $\Z$ acting  by $(n,x)\mapsto \lambda^{n}x$  on  the coarse space $\R$ with the coarse structure induced from the standard metric. This is an example of a coarse space with a $G$-action by automorphisms. But if $|\lambda|\not=1$, then it is not a $G$-coarse space.
   
 In contrast, if we let $\Z$ act on $\R$ by translations $(n,x)\mapsto n+x$, then this action is isometric and we would get a $G$-coarse space.
\hB
\end{rem}

Let $X$ be in $G\Set$. 
\begin{ddd}\label{wthoiwhthwgreggwregwgr}
A $G$-bornology on $X$ is a subset $\cB_{X}$ of $\cP_{X}$ satisfying the following conditions: 
\begin{enumerate}
\item $\cB_{X}$ is $G$-invariant.
\item $\cB_{X}$ is closed under forming subsets and finite unions.
\item\label{qergoijiogrgqgreg} $\bigcup_{B\in \cB_{X}}B=X$.
\end{enumerate}
\end{ddd}
The elements of $\cB_{X}$ will be called the bounded subsets of $X$.

A subset $A$ of $\cP_{X}$ generates a smallest $G$-bornology $\cB\langle A\rangle$ containing $A$.

 We consider  two $G$-sets with $G$-bornologies   $(X,\cB_{X})$ and $(X^{\prime},\cB_{X^{\prime}} )$, and a map   $f\colon X\to X^{\prime}$  in  $G\Set$. 

\begin{ddd}\label{etgiohjroifgjqrofiqfewfqew}\mbox{}
\begin{enumerate}\item
  $f$ is  proper if $f^{-1}(\cB_{X'})\subseteq \cB_{X }$.
\item \label{rfoijoiffqefefeq} $f$ is bornological if $f(\cB_{X})\subseteq \cB_{X'}$.
 \end{enumerate}
\end{ddd}

\begin{ddd}\label{ebjoifdbsdbsfdbsbsdb}A $G$-bornological space is a pair $(X,\cB_{X})$ of a $G$-set $X$  and a $G$-bornological structure $\cB_{X}$. 
A morphism between  $G$-bornological spaces is an equivariant and proper map.
\end{ddd}
  
We obtain the category $G\Born$ of $G$-bornological spaces and  equivariant and proper maps.

Let $X$ be a $G$-set with a $G$-coarse structure $\cC_{X}$ and a $G$-bornology $\cB_{X}$. 
\begin{ddd}\label{rgejqieogjrgoij1o4trqq}
$\cC_{X}$ and $\cB_{X}$ are compatible if for every $B$ in $\cB_{X}$ and every $V$ in $\cC_{X}$ we have  $V[B]\in \cB_{X}$.
\end{ddd}

\begin{ddd}\label{gjwerogijwoergwergwergweg}
A $G$-bornological coarse space is a triple  of a $G$-set $X$ with a $G$-coarse structure $\cC_{X}$ and a $G$-bornology $\cB_{X}$ such that $\cC_{X}$ and $\cB_{X}$ are compatible.
A morphism between $G$-bornological coarse spaces is a morphism  of $G$-sets which is proper and controlled.
\end{ddd}
In this way we get a category $G\BC$   of $G$-bornological coarse spaces and morphisms.
We usually denote a $G$-bornological coarse space 
$(X,\cC_{X},\cB_{X})$ just by  the symbol~$X$.

\begin{ex}\label{ergiowergergregfwrefwref}
If $(X,d)$ is a metric space with an isometric action of $G$, then we get a $G$-bornological coarse space $X_{d}$. Its $G$-coarse structure $\cC_{d}$ is  described in Example \ref{rwtekoierfrefegd}. Its bornology $\cB_{d}$ is the smallest one which is compatible
with $\cC_{X}$. Explicitly, $\cB_{d}$ is the set of all metrically bounded subsets of $X$.  
\hB
\end{ex}

\begin{ex}\label{qrgioqjrgoqrqfewfeqfqewfe} 
 For   $S$ in $G\Set$ we can consider the objects $S_{max,max}$, $S_{min,max}$ and $S_{min,min}$ in $G\BC$, where
the first index $min$ or $max$ refers to the minimal or maximal $G$-coarse structure (subsets of $\diag(S)$ or all of $\cP_{S\times S}$),
and the second $min$ or $max$ refers to the minimal  or maximal bornology (all finite subets or all of $\cP_{S}$) on $S$. 

If $S$ is infinite, then $S_{max,min}$ does not exist since in this case the coarse structure and the bornology are not compatible. \hB
 \end{ex}

\begin{ex}\label{etwgokergpoergegregegwergrg}
The group $G$ has a canonical $G$-coarse structure $\cC_{can}$  generated by the entourages $\{(g,g')\}$ for  all pairs of elements $g,g'$ in $G$. We 
let $G_{can,min}$  be the object of $G\BC$ given by the $G$-set $G$ (with the left action) with the coarse structure $\cC_{can}$ and the minimal bornology $\cB_{min}$ which consists of the finite subsets.
\hB
\end{ex}

\begin{ex}\label{wetgwetgrefrefwfrefw}
If $f\colon Z\to X$ is a map in $G\Set$ and $X$ is in $G\BC$, then $Z$ has an induced $G$-bornological  coarse structure. 
The coarse structure $\cC_{Z}$ is the maximal one such that $f$ is controlled, and the bornology $\cB_{Z}$ is the minimal one such that $f$ is proper. These structures are in fact  compatible. If $Z$ is equipped with these structures, then  $f$ becomes a morphism in $G\BC$.
\hB
\end{ex}

The category $G\BC$ has a symmetric monoidal structure $\otimes$:
\begin{ddd}[{\cite[Ex.\ 2.17]{equicoarse}}]\label{defn_symmetric_monoidal_GBC}\mbox{}
\begin{enumerate}\item 
For $X$,$X^{\prime}$ in $G\BC$ the underlying $G$-set of $X\otimes X^{\prime}$ is $X\times X^{\prime}$ with the diagonal action. Its $G$-coarse structure is generated by the entourages $U\times U^{\prime}$ for  all $U$ in $\cC_{X}$ and $U^{\prime}$ in $\cC_{X^{\prime}}$, and its $G$-bornology is generated by the sets $B\times B^{\prime}$ for all $B$ in $\cB_{X}$ and $B^{\prime}$ in $\cB^{\prime}_{X}$.  
\item The tensor product of two morphisms $f,f'$ is given by the map $f\times f$ between the cartesian products of the underlying $G$-sets.
\item The tensor unit is the one-point space $*$ with the unique bonological coarse structure.
\item The associativity and symmetry constraints are imported from the cartesian structure on {the} category $G\Set$.
\end{enumerate}
\end{ddd}


Let $f_{0},f_{1}\colon X\to X^{\prime}$ be a pair of morphisms in $G\BC$.

\begin{ddd}\label{wegwerfreweggerg}
$f_{0}$ and $f_{1}$ are close to each other if $(f_{0}\times f_{1})(\diag(X))\in \cC_{X'}$.
\end{ddd}

\begin{ex}
We consider $\R$ with the metric coarse structure.

The morphisms $\id_{\R}$ {and} $x\mapsto -x$ are not close to each other.

The morphisms $\id_{\R}$  {and} $x\mapsto x+1$ are close to each other.
\hB
\end{ex}

Let  $f\colon X\to X'$  be a morphism  in $G\BC$.
 
\begin{ddd}\label{wetgiojwergergergw} \mbox{}
\begin{enumerate}
\item $f$ is a coarse equivalence if there exists a morphism $g\colon X'\to X$ such that $f\circ g$ is close to $\id_{X'}$ and $g\circ f$ is close to $\id_{X}$.
\item\label{wetgiojwergergergw1}
 $f$ is a weak coarse equivalence
it it becomes a coarse equivalence after forgetting the $G$-action.
\end{enumerate}
\end{ddd}
 
Note that a coarse equivalence is a weak coarse equivalence, but the converse is {in general} not true:
 
 \begin{ex}
 We consider the group $\Z/2\Z$. We further consider   the metric space $\R$ as a $\Z/2\Z$-bornological coarse space such that the    non-trivial element of $\Z/2\Z$ acts as multiplication by $-1$. It preserves the subspace $\R\setminus \{0\}$.
 The inclusion $\R\setminus \{0\}\to\R$ is a morphism in $(\Z/2\Z)\BC$ which is a weak coarse equivalence, but not a coarse equivalence. Indeed, there does not exist any equivariant map $\R\to \R\setminus \{0\}$, because the image of $0$ under such a map had to be a $\Z/2\Z$-fixed point in the target which does not exist.
\hB
\end{ex}

Let $X$ be in $G\BC$.

\begin{ddd}\label{qriugoqergreqwfqewfqef}
$X$ is called flasque if it admits an endomorphism $f$ such that:
\begin{enumerate}
\item\label{twgiojweorgrgrgwergwegergwerg} $\id_{X}$ and $f$ are close to each other.
\item \label{qrgioqrgfregqfqefqwefqef1} For every $U$ in $\cC_{X}$ we have $\bigcup_{n\in \nat} (f^{n}\times f^{n})(U)\in \cC_{X}$.
\item\label{qrgioqrgfregqfqefqwefqef} For every $B$ in $\cB_{X}$ there exists $n$ in $\nat$ such that $f^{n}(X)\cap B=\emptyset$.
\end{enumerate}
\end{ddd}

We say that $f$ implements the flasqueness of $X$.

\begin{ex}
We consider $\nat$ with the bornological coarse structure induced by the usual metric. 
Then $\nat \otimes X$ is flasque for 
  every $G$-bornological coarse space $X$. The selfmap $(n,x)\mapsto (n+1,x)$ implements flasqueness. \hB
\end{ex}

If $X$ is in $G\BC$ and $U$ is in $\cC_{X}^{G}$, then we let $X_{U}$ denote the object $(X,\cC\langle \{U\}\rangle ,\cB_{X})$ of $G\BC$, where $\cC\langle\{U\}\rangle$ is the $G$-coarse structure generated by $U$. The identity of the underlying set of $X$ induces a  morphism $X_{U}\to  X$ in $G\BC$.

\begin{ex}
If $(X,d)$ is a path metric space and we consider
\[U_{r}\coloneqq \{(x,x')\in X\times X\:|\: d(x,x')\le r\}\]
for some $r$ in $(0,\infty)$, then $X_{U_{r}}\to X_{d}$ is an isomorphism (see Example \ref{ergiowergergregfwrefwref}).

If we consider the subset $X \coloneqq \{n^{2}\:|\: n\in \nat\}$ of $\R$ with the induced bornological coarse structure (Example \ref{wetgwetgrefrefwfrefw}), then the map $X_{U}\to X$ for $U$ in $\cC_{X}$ is never  an isomorphism.
\hB
\end{ex}

Let $\cY\coloneqq (Y_{i})_{i\in I}$ be a family of subsets of $X$ indexed by a very small filtered poset.

\begin{ddd}\label{gqrfqwerfqfwefwefqwf}
$\cY$ is an invariant big family if:
\begin{enumerate} 
\item  For all $i,i'$ in $I$ with $i\le i'$ we have $Y_{i}\subseteq Y_{i'}$.
\item $Y_{i}$ is $G$-invariant for every $i$ in $I$.
\item\label{weoigjwoegergergergwerg} For every $i$ in $I$ and $U$ in $\cC_{X}$ there exists  $i'$ in $I$ such that  $U[Y_{i}]\subseteq Y_{i'}$.  \end{enumerate}
\end{ddd}

\begin{ex} In the case of the trivial group
the bornology $\cB_{X}$  considered as a self-indexing family is an example of a big family on $X$ by the compatibility of $\cB$ with  the coarse structure.

In general, if $Y$ is a $G$-invariant subset of $X$, then the family $\{Y\}:=(U[Y])_{U\in \cC_{X}^{G}}$ is an invariant big family. It is called the invariant big family generated by $Y$.
\hB
\end{ex}

\begin{ddd}\label{eqrgklefqwewefqeff}
An  invariant complementary pair $(\cY,Z)$ on $X$ is a pair of an invariant big family $\cY$ and an invariant subset $Z$ of $X$ such that there exists $i$ in $I$ with $Y_{i}\cup Z=X$.
\end{ddd}

If not said differently we consider subsets of $X$ with the induced bornological coarse structures (Example \ref{wetgwetgrefrefwfrefw}).
We can form the big family  $\cY\cap Z\coloneqq (Y_{i}\cap Z)_{i\in I}$ on $Z$.

In the present paper we use the language of $\infty$-categories\footnote{more precisely $(\infty,1)$-categories}. References are \cite{htt,Cisinski:2017}.  Ordinary categories will be considered as $\infty$-categories using the nerve functor. 
A typical target $\infty$-category for the homological functors introduced below is the stable $\infty$-category $\Sp$ of spectra.  We refer to \cite{HA} for an introduction to stable $\infty$-categories in general, and for $\Sp$ in particular.  
The $\infty$-categories considered in the present paper belong to the large universe. 

Let $$E\colon G\BC\to \bM$$ be a functor with target some $\infty$-category.   For an invariant big family $\cY=(Y_{i})_{i\in I}$ on some $G$ bornological coarse space $X$ we set 
$$E(\cY)\coloneqq \colim_{i\in I} E(Y_{i})$$
provided this colimit exists.
 
\begin{ddd}\label{ergiweorgregergwerg}
$E$ is coarsely invariant if  for any pair of close morphism $f_{0},f_{1}\colon X\to X'$ the induced morphisms $E(f_{0})$ and $E(f_{1})$ are equivalent.
\end{ddd}

One can show that $E$ is coarsely invariant if and only if it sends coarse equivalences to equivalences.

  \begin{ddd} \label{wetgiwoergregwrgregwg}
 $E$ is excisive, if the following conditions are satisfied:
 \begin{enumerate}
 \item 
  $E(\emptyset)$ is an initial object in $\bM$.
  \item For  every $X$ in $G\BC$ and invariant complementary pair $(\cY,Z)$  on $X$:
  \begin{enumerate}
  \item $E(\cY)$ and $E(\cY\cap Z)$ exist.
  \item The square
$$\xymatrix{E(\cY\cap Z)\ar[r]\ar[d]&E(Z)\ar[d]\\E(\cY)\ar[r]&E(X)}$$
is a push-out square in $\bM$. 
 \end{enumerate} \end{enumerate}
  \end{ddd}

\begin{ddd}\label{wetklhgwoelgregweg9}
$E$ is $u$-continuous if for every $X$ in $G\BC$ the following conditions are satisfied:
\begin{enumerate}
\item   The  colimit $\colim_{U\in \cC_{X}^{G}} E(X_{U})$  exists.
\item  The  canonical  {morphism}
$\colim_{U\in \cC_{X}^{G}} E(X_{U})\to E(X)$ is an equivalence. \end{enumerate}
\end{ddd}

\begin{ddd}
$E$ vanishes on flasques if for every   flasque object  $X$  in $G\BC$ the object  $E(X) $ is initial in $\bM$.
\end{ddd}

Let $X$ be in $G\BC$ and $L$ be a subset of $X$.
\begin{ddd}\label{rgoiqrgrfqwfewfq} We say that $L$ is locally finite if $|L\cap B|<\infty$ for all $B$ in $\cB_{X}$.
\end{ddd}
We let $\LF(X)$ denote the poset of locally finite, $G$-invariant subsets of $X$. If $f\colon X\to Y$ is a morphism in $G\BC$, then it induces a morphism of posets $\LF(X)\to \LF(Y)$ given by $L\mapsto f(L)$.

 \begin{ddd}\label{weifhqewiefjeefqefwefqwef}
 $E$ is continuous if for every $X$ in $G\BC$  the following conditions are satisfied:
 \begin{enumerate}
 \item The colimit  $\colim_{L\in \LF(X)} E(L)$ exists.
 \item The canonical morphism $\colim_{L\in \LF(X)} E(L)\to E(X)$ is an equivalence.
 \end{enumerate}
  \end{ddd}

Let $X$ be in $G\BC$, and let $Y,Z$ be  two  subsets. 
\begin{ddd}
$Y$ and $Z$ are coarsely disjoint if $U[Y]\cap U[Z]=\emptyset$ for every $U$ in $\cC_{X}$.
\end{ddd}

We assume that $\bM$  is pointed and let $0_{\bM}$ denote the zero-object.

\begin{ddd}\label{qerogfjqeropfewfqwffe}
$E$ is called $\pi_{0}$-excisive if for every $X$ in $G\BC$ and invariant coarsely disjoint partition $(Y,Z)$ of $X$ into $G$-invariant subsets the square
\begin{equation}\label{eq_pushout_pizero_excisive}
\xymatrix{0_{\bM}\ar[r]\ar[d]&E(Z)\ar[d]\\E(Y)\ar[r]&E(X)}
\end{equation}
is a push-out square.
\end{ddd}
 
\begin{ex} An excisive functor is $\pi_{0}$-excisive. In fact, if $(Y,Z)$ is any invariant coarsely disjoint partition of $X$ into $G$-invariant subsets, then $(Z,(Y))$ is an equivariant complementary pair, where $(Y)$ is the single member family with member $Y$.
Moreover, $Z\cap (Y)=(\emptyset)$.  In view of Definition \ref{wetgiwoergregwrgregwg} the assertion is now obvious.
\hB
 \end{ex}

We consider  an invariant coarsely disjoint partition $(Y,Z)$ of $X$ into $G$-invariant subsets.
The pair of maps $\id_{E(Y)}\colon E(Y)\to E(Y)$ and $0\colon E(Z)\to 0_{\bM}\to E(Y)$ induce, using the presentation of $E(X)$ as a push-out  \eqref{eq_pushout_pizero_excisive}, a projection $p_{Y}\colon E(X)\to E(Y)$.

Let $(X_{i})_{i\in I}$ be a family in $G\BC$ indexed by a very small set $I$.
\begin{ddd}[{\cite[Ex.~2.16]{equicoarse}}]\label{ergoiejoqfqefweqfqwef}
The free union $\bigsqcup^{\free}_{i\in I}X_{i}$ is the following object of $G\BC$:
\begin{enumerate}
\item The underlying $G$-set is the disjoint union $\bigsqcup_{i\in I} X_{i}$ of $G$-sets.
\item The $G$-coarse structure is generated by the entourages $\bigcup_{i\in I} U_{i}$ for all families $(U_{i})_{i\in I}$ in $\prod_{i\in I} \cC_{X_{i}}$.
\item The $G$-bornology is generated by $\bigcup_{i\in I} \cB_{X_{i}}$.
\end{enumerate}
\end{ddd}

For every $j$ in $I$ we have an invariant  and coarsely disjoint decomposition 
$(X_{j},\bigsqcup^{\free}_{i\in I\setminus \{j\}}X_{i})$ of $\bigsqcup^{\free}_{i\in I}X_{i}$.
If $E$ is $\pi_{0}$-excisive, then we have a projection 
$$p_{j}\colon E\big(\bigsqcup^{\free}_{i\in I}X_{i}\big) \to E(X_{j})\, .$$

\begin{ddd}[{\cite[Def. 3.12]{equicoarse}}]\label{ergoijerwogergerfwfref9}
$E$ is called strongly additive, if:
\begin{enumerate}
\item $\bM$ is pointed.
\item $E$ is $\pi_{0}$-excisive.
\item For every  family $(X_{i})_{i\in I}$  in $G\BC$  indexed by a very small set $I$ we have:
\begin{enumerate}
\item The product  $\prod_{j\in I} E(X_{j})$ exists.
\item  The canonical morphism
 $(p_{j})_{j\in I}\colon E(\bigsqcup^{\free}_{i\in I}X_{i})\to \prod_{j\in I} E(X_{j})$  is an equivalence.
\end{enumerate}
\end{enumerate}
%
\end{ddd}
 
We consider a functor $E\colon G\BC\to \bM$ where $\bM$ is a stable $\infty$-category.

\begin{ddd}\label{wefuihqfwefwffqwefefwq}
$E$ is an equivariant coarse homology theory if it has the following properties:
\begin{enumerate}
\item $E$ coarsely invariant 
\item $E$ is excisive. 
\item $E$ is $u$-continuous. 
\item $E$ vanishes on flasques.
\end{enumerate}
\end{ddd}

Note that the target category of an equivariant coarse homology theory is always assumed to be stable. 
Usually it will also be cocomplete. In this case   the conditions of existence of various filtered colimits   in the definitions above 
{are} automatic.

The properties of continuity and strong additivity are considered as additional properties which
a coarse homology may or may not have.  

 
\begin{ex}\label{regiuehrifwefqfewffqewedqe}
If $Y$ is in $G\BC$ and $E\colon G\BC\to \bM$ is an equivariant coarse homology theory, then we can define a new equivariant coarse homology theory
$$E_{Y}\colon G\BC\to \bM\, , \quad X\mapsto E(Y\otimes X)$$
called the twist of $E$ by $Y$, see \cite[Sec.~4.3]{equicoarse}. Here $\otimes$ is the symmetric monoidal structure on $G\BC$ introduced in Definition~\ref{defn_symmetric_monoidal_GBC}.
\hB
\end{ex}

  Recall the Definition  \ref{qriugoqergreqwfqewfqef} flasquenes. In order to obtain the notion of weak flasqueness
 we weaken the Condition~\ref{qriugoqergreqwfqewfqef}.\ref{twgiojweorgrgrgwergwegergwerg}. Let $X$ be in $G\BC$. 
\begin{ddd}[{\cite[Def.~4.18]{equicoarse}}]\label{qriugoqergreqwfqewfqefr324r34r3r34r34r34r}
$X$ is called weakly flasque if it admits an endomorphism $f$ such that:
\begin{enumerate}
\item\label{toijgworegewrgwer} For every {equivariant} coarse homology theory $E$ we have  $E(\id_{X})\simeq E(f)$.
\item \label{qrgioqrgfregqfqefqwefqef1w} For every $U$ in $\cC_{X}$ we have $\bigcup_{n\in \nat} (f^{n}\times f^{n})(U)\in \cC_{X}$.
\item\label{qrgioqrgfregqfqefqwefqefw} For every $B$ in $\cB_{X}$ there exists $n$ in $\nat$ such that $f^{n}(X)\cap B=\emptyset$.
\end{enumerate}
\end{ddd}
 
 We say that $f$ implements weak flasqueness of $X$.
Note that in \cite[Def.~4.18]{equicoarse} we used the universal equivariant coarse homology theory $\Yo^{s}$ to give an equivalent formulation of Condition~\ref{qriugoqergreqwfqewfqefr324r34r3r34r34r34r}.\ref{toijgworegewrgwer}.  
 
Since an {equivariant} coarse homology theory is coarsely invariant it is clear that a flasque $G$-bornological coarse space is weakly flasque.

Let $E\colon G\BC\to \bM$ be an equivariant coarse homology theory. 
 
\begin{ddd}\label{wqroijwoidfewdewqdqwdqew}
$E$ is called strong if it annihilates  weakly flasque $G$-bornological coarse spaces.
\end{ddd}

Strongness of a coarse homology theory plays an important role in the construction and study of the coarse assembly map \cite{ass}
and its equivariant generalization \cite[Sec.~11.3]{equicoarse}. 
All examples of coarse homology theories constructed in the present paper are strong.

   \section{\texorpdfstring{$C^{*}$}{Cstar}-categories and homological functors} \label{qergijfioqrfqwedwedqwdq}

  In this section we recall the basic definitions and facts on $C^{*}$-categories which will be used in the present paper.
  For details and further references we refer to  \cite{cank} and \cite{crosscat}.
   We start with a description of the category $\nCcat$ of possibly non-unital $C^{*}$-categories.
 A $\C$-linear $*$-category  is a possibly non-unital  (we do not require the existence of identity endomorphisms)  category $\bC$ which is enriched in complex vector spaces and has an antilinear involution $*\colon \bC^{\op}\to \bC$  fixing objects. Note that a $C^{*}$-algebra $A$ is an example of a $\C^{*}$-linear $C^{*}$-category with a single object. A morphism between $\C$-linear $*$-categories is a not necessarily unit-preserving functor which is compatible with the enrichment and the involutions.  We let $\nClincat$ denote the large category of small $\C$-linear $*$-categories and morphisms. By $\Clincat$ we denote its subcategory of unital $\C$-linear $*$-categories  (where all objects admit identity morphism) and unit-preserving functors.
  
 Given a morphism $f$ in a $\C$-linear $*$-category $\bC$ we can define its maximal seminorm
 $\|f\|_{\max}$ in $[0,\infty]$ as the supremum of the norms $\|\rho(f)\|$ for all morphisms $\rho \colon \bC\to A$ from $\bC$ to   $C^{*}$-algebras $A$.   A $\C$-linear $*$-category is called a pre-$C^{*}$-category if all  its  morphisms have finite maximal seminorms. 
 
 \begin{ddd}A $C^{*}$-category is a  pre-$C^{*}$-category with the property that its $\Hom$-spaces are Banach spaces
 with respect to the maximal norm. A morphism between $C^{*}$-categories is a morphism of $\C$-linear $*$-categories.
 \end{ddd}
 By definition, $\nCcat$ is a full subcategory of $\nClincat$.
 In order to show that a $\C$-linear $*$-category $\bC$ is a $C^{*}$-category it suffices
 to provide a submultiplicative Banach space norm on the morphism spaces which in addition satisfies the $C^{*}$-identity and has the property that
 $f^{*}\circ f$ is a non-negative element in the corresponding endomorphism $C^{*}$-algebra for any morphism $f$ in $\bC$.
 We say that this norm exhibits $\bC$  as a $C^{*}$-category. 
 When we talk about morphisms in $\nCcat$ we will use the words morphism or functor synonymously.
  In this case it coincides with the maximal seminorm.
A functor between    $C^{*}$-categories automatically induces contractions on the morphism spaces. 
 
 \begin{ddd} A functor between $C^{*}$-categories is said to be   faithful (fully faithful) if it induces injections (bijections)
 on morphism spaces.
 \end{ddd}

Note that a  faithful functor induces isometries on the morphism spaces. 
 
We have inclusions $\nCalg\subseteq \nCcat$ and $\Calg\subseteq \Ccat$ 
    of the categories of $C^{*}$-algebras and unital $C^{*}$-algebras into the corresponding categories of $C^{*}$-categories  which interpret $C^{*}$-algebras as single object $C^{*}$-categories.
      The following generalizes the notion of an exact sequence of $C^{*}$-algebras and of an ideal  to $C^{*}$-categories.
 \begin{ddd}A sequence of functors $\bC\to \bD\to \bE$ between $C^{*}$-categories  will be called exact 
 if the functors are bijections on the level of objects and induce short exact sequences on the level of morphism spaces.
  \end{ddd} 
  
  We will usually indicate exactness by writing $0\to \bC\to \bD\to \bE\to 0$. 
 
 \begin{ddd}\label{rwgiojworgergrwegrefr}
 A functor $\bC\to \bD$ between $C^{*}$-categories is an ideal inclusion if it faithful,    a bijection on the level of objects,   and if the composition of a morphism from $\bD$ with a morphism from $\bC$ always belongs to $\bC$.
 \end{ddd}
  
 In an exact sequence as above the inclusion $\bC\to \bD$ is an ideal inclusion, and every ideal inclusion
 can be extended to an exact sequence by adding the quotient $C^{*}$-category as the third term.

For a group $G$ we let $\Fun(BG,\nCcat)$ denote the category of $C^{*}$-categories with a strict $G$-action. 
   There are a maximal and  a reduced crossed product functor
$$-\times G, -\rtimes_{r}G \colon \Fun(BG,\nCcat)\to \nCcat$$ 
described in   \cite[Sec.~5]{crosscat} and \cite[Sec.~12]{cank}, respectively.
   
   \begin{ex}\label{dqedqwdqwd1234frgfrfwerger} We consider a very  small $C^{*}$-algebra with $G$-action
  $A$ in $\Fun(BG,\nCalg)$. We write the action of $g$ in $G$ on $A$ by $a\mapsto {}^{g}a$.
  
  Following  \cite[Ex. 2.10]{cank} we  have   the $C^{*}$-category with $G$-action
 $\Hilb(A)$ of very small $A$-Hilbert $C^{*}$-modules in $\Fun(BG,\nCcat)$.
 The objects of $\Hilb(A)$ are the very small $A$-Hilbert $C^{*}$-modules, and the morphisms are the adjointable bounded operators between Hilbert $C^{*}$-modules. If $g$ is an element of $G$, then the functor
$g \colon \Hilb(A)\to \Hilb(A)$ is defined as follows:
\begin{enumerate}
\item objects: We consider  an $A$-Hilbert $C^{*}$-module  $M$ with right multiplication $\cdot$ and $A$-valued scalar product $\langle-,-\rangle_{M}$.  The  underlying complex vector space of $gM$ is that of $M$. 
The right $A$-action $\cdot_{g}$ on $gM$ is given by $m\cdot_{g} a \coloneqq m\cdot {}^{g^{-1}}a$, and the $A$-valued scalar product
is $\langle m,m'\rangle_{gM} \coloneqq {}^{g}\langle m,m' \rangle_{M}$.
 \item For a morphism $f \colon M\to M'$ the morphism $gf \colon gM\to gM'$ is given by the same linear map $f$.
\end{enumerate}
The  $C^{*}$-category $\Hilb(A)$ is unital, but it has  an invariant   non-unital subcategory $\Hilb_{c}(A)$
with the same objects but  the  compact operators (in the sense of $A$-Hilbert $C^{*}$-modules) as morphisms.
The inclusion $\Hilb_{c}(A)\to \Hilb(A)$ is an ideal inclusion.
\hB
\end{ex}

We can take the $C^{*}$-category $\Hilb_{c}(A)$ described in Example \ref{dqedqwdqwd1234frgfrfwerger} as the coefficient category for the constructions in the subsequent sections which then produce a variety of new examples
of $C^{*}$-categories with and without $G$-action.

For the following material we refer to \cite[Sec.~2 \& Sec.~3]{cank}.
The notion of a $W^{*}$-category generalizes that of a von Neumann algebra.
\begin{ddd} A $W^{*}$-category is a unital $C^{*}$-category with the property that all its morphism Banach spaces admit pre-duals.
\end{ddd}

By \cite[Def.~2.35]{cank}, to any $C^{*}$-category $\bC$ one can functorially associate its $W^{*}$-envelope $\bW^{\mathrm{nu}}\bC$. The superscript indicates that it is the non-unital version of  the functor $\bW$ from unital $C^{*}$-categories to   $W^{*}$-categories. The latter  is characterized as    the left-adjoint to the inclusion from unital $W^{*}$-categories (and functors which are in addition $\sigma$-weakly continuous) to unital $C^{*}$-categories.
For a non-unital $C^{*}$-category $\bC$ we then get the $W^{*}$-catergory $\bW^{\mathrm{nu}}\bC$ (it is unital) as the $\sigma$-weak closure of $\bC$ under
$\bC\to \bC^{u}\to \bW\bC^{u}$, where $\bC\to \bC^{u}$ denotes the unitalization.
 We  have an isometric and $\sigma$-weakly dense inclusion
$\bC\to \bW^{\mathrm{nu}}\bC$, and if $f \colon \bC\to \bD$ is a morphism of $C^{*}$-categories, then the functor $\bW^{\mathrm{nu}}(f) \colon \bW^{\mathrm{nu}}\bC\to \bW^{\mathrm{nu}}\bD$ is in addition 
$\sigma$-weakly continuous. 

In order to give a quick description of the multiplier category of a $C^{*}$-category we turn   \cite[Thm.~3.15]{cank} into a definition.

\begin{ddd} The multiplier category $\bM\bC$ is defined as the wide subcategory of $\bW^{\mathrm{nu}}\bC$ which idealizes $\bC$.
\end{ddd}
 The multiplier category $\bM\bC$  thus consists of those morphisms of $\bW^{\mathrm{nu}}\bC$ which have the property that all compositions with morphisms from $\bC$ again belong to $\bC$. Note that $\bM\bC$ is unital and contains $\bC$ as an ideal.
The universal  property of  the inclusion $\bC\to\bM\bC$  says   that it is a  final object in the category of  ideal inclusions under $\bC$.

\begin{ddd} A morphism $f \colon \bC\to \bD$ is non-degenerate if for any two objects $C,C'$ of $\bC$
the linear subspaces  $\phi(\End_{\bC}(C'))\Hom_{\bD}(\phi(C),\phi(C'))$ and
$ \Hom_{\bD}(\phi(C),\phi(C'))\phi(\End_{\bC}(C))$ are dense in $  \Hom_{\bD}(\phi(C),\phi(C'))$.\end{ddd}
If $f \colon \bC\to \bD$ is non-degenerate, then by \cite[Prop.~3.16]{cank}
the morphism $\bW^{\mathrm{nu}}(f) \colon \bW^{\mathrm{nu}}\bC\to \bW^{\mathrm{nu}}\bD$ restricts to a morphism $\bM(f) \colon \bM\bC\to \bM\bD$.

\begin{ex}
In the situation of the Example  \ref{dqedqwdqwd1234frgfrfwerger} (forget the $G$-action) we have an  isomorphism
$\bM\Hilb_{c}(A)\cong \Hilb(A)$. 
If $A\to B$ is a homomorphism, then the functor $-\otimes_{A}B \colon \Hilb_{c}(A)\to \Hilb_{c}(B)$ is a non-degenerate morphism.
Its extension to the multiplier categories corresponds  under the isomorphism  above to the same functor $-\otimes_{A}B \colon \Hilb(A)\to \Hilb(B)$. \hB
\end{ex}

\begin{ddd}An endomorphism $p $ in a $C^{*}$-category is called a projection if $p=p^{*}=p^{2}$. \end{ddd}
\begin{ddd}A morphism $u \colon C\to C'$   is called a unitary (an isometry, or a partial isometry), if
$u^{*}\circ u=\id_{C}$ and $u\circ u^{*}=\id_{C'}$ ($u^{*}\circ u=\id_{C}$, or $u\circ u^{*}$ and $u\circ u^{*}$ are projections, respectively).  \end{ddd}Note that objects supporting a unitary or an isometry are necessarily unital.

\begin{ddd}\label{wetijgowegreferfwrefwrefrf}A projection $p$ on $C$  is effective if there exists an isometry $u \colon C'\to C$ with $u\circ u^{*}=p$.
We then say that $(C',u)$ represents an image of $p$.  
\end{ddd}

\begin{ddd}
We say that a unital $C^{*}$-category is idempotent complete if every projection  is effective.
\end{ddd}

Consider a  finite collection  $(C_{i})_{i\in I}$  of objects in a $C^{*}$-category. 

\begin{ddd}An orthogonal sum of this family is a pair $(C,(e_{i})_{i\in I})$ of an object $C$ and a family of isometries
$e_{i} \colon C_{i}\to C$ such that 
\begin{equation}\label{ewfoiwjfoaewfewfeadewd} \sum_{i\in I} e_{i} \circ e_{i}^{*}=\id_{C}\, . 
\end{equation}
\end{ddd}  Again note that in this case the  objects
$C$ and  $C_{i}$ for all $i$ in $I$ are necessarily unital.

\begin{ddd}\label{ewrgiojwgerwfweferfw} We say that $C^{*}$-category is additive if it admits orthogonal sums for all finite families of objects. 
\end{ddd}
An additive $C^{*}$-category is unital.

 The category $\Ccat$ of unital $C^{*}$-categories and unit-preserving functors has a natural extension to  a $(2,1)$-category $\Ccat_{2,1}$ whose   $2$-isomorphisms are natural transformations  (called unitary isomorphisms) implemented by unitaries.

To any unital $C^{*}$-category $\bC$ we can functorially associate its
additive and its additive and idempotent completions
$$\bC\to \bC_{\oplus}\to\Idem( \bC_{\oplus})\, ,$$ see \cite[Sec.~16]{cank}. 
These constructions are left-adjoints of two-categorical adjunctions whose right-adjoints are the respective inclusion functors.

\begin{ddd}\label{weroigjowregrefwerferwfw}
A unital morphism $\bC\to \bD$ between unital $C^{*}$-categories is called a Morita equivalence if it is fully faithful and 
the induced morphism $\Idem(\bC_{\oplus})\to \Idem(\bD_{\oplus})$ is an equivalence.
\end{ddd}

In order to check whether a fully faithful morphism $\bC\to \bD$ is a Morita equivalence it suffices to
show that every object of $\bD$ is a subobject of a finite sum of objects of $\bC$.

\begin{ex}The $C^{*}$-categories $\Hilb(A)$ and the full subcategories $\Hilb^{\fg, \free}(A)$ and  $\Hilb(A)^{\fg,\mathrm{proj}}$
of finitely generated free or finitely generated projective modules, respectively, are additive. In general,
$\Hilb_{c}(A)$ is not additive.

The categories  $\Hilb(A)$ and the full subcategory $\Hilb(A)^{\fg,\mathrm{proj}}$ are 
idempotent complete, but $\Hilb^{\fg, \free}(A)$ in general is not. The inclusion
$\Hilb^{\fg, \free}(A)\to \Hilb^{\fg, \mathrm{proj}}(A)$ presents the target as an idempotent completion. \hB
\end{ex}

In order to extend the definition of an orthogonal sum to infinite families 
we must specify in which sense the sum in \eqref{ewfoiwjfoaewfewfeadewd} converges.
Requiring convergence in norm would exclude all non-trivial examples of infinite sums. 
In the present paper we will work with  the notion of AV-sums (named after \cite{antoun_voigt} where this definition has been introduced).  We will use the  strict topology on the morphism spaces of $\bM\bC$  which is the topology of pointwise convergence
of compositions with morphisms in $\bC$.

Let $\bC$ be a $C^{*}$-category and
consider a small family  
$(C_{i})_{i\in I}$  of objects in $\bC$. 

\begin{ddd}\label{qerghqiferwfqefedeq}An orthogonal AV-sum 
of this family is a pair 
 $(C,(e_{i})_{i\in I})$ of an object $C$ and a family of isometries
$e_{i} \colon C_{i}\to C$ in $\bM\bC$ such that the sum in \eqref{ewfoiwjfoaewfewfeadewd} converges
in the strict topology. 
\end{ddd}
We refer to \cite[Sec.~7]{cank} for a description of the space of  morphisms from and into AV-sums.
If $\phi \colon \bC\to \bD$ is a non-degenerate morphism, then it preserves AV-sums.

\begin{ex} AV-sums in $\Hilb_{c}(A)$ are represented by orthogonal sums of $A$-Hilbert $C^{*}$-modules in the usual sense, see \cite[Sec.~8]{cank}. \hB
\end{ex}

Let  $\bC$ be  a $C^{*}$-category.
 \begin{ddd}\label{wiotgwtgwregewf} We  call $\bC$  effectively additive if for every object $C$ of $\bC$ and very small mutually orthogonal family of effective projections $(p_{i})_{i\in I}$ on $C$  in $\bM\bC$ such that $\sum_{i\in I}p_{i}$ converges strictly  to a projection $p$ in $\bM\bC$ also   $p$   is   effective  in $\bM\bC$. 
\end{ddd}

Thus $\bC$ is effectively additive if it admits  orthogonal AV-sums of orthogonal families of subobjects of objects. 
If $\bM\bC$ is idempotent complete, then it is clearly effectively additive.
If $\bC$ in $\nCcat$ admits all very small orthogonal AV-sums, then it is effectively additive
by \cite[Lem.~7.9]{cank}.

 We introduce the following notation:
  \begin{ddd}\label{weiotjgoegergerfwrefer}
  \mbox{}
  \begin{enumerate}
  \item
 $\ndCcat$ is the wide  subcategory of $\nCcat$ of  non-degenerate morphisms. \item   $\ndeCcat$ is the full subcategory of $\ndCcat$ of effectively additive objects.
\item   $\ndaCcat$  is  the full subcategory of $\ndeCcat$ of  objects 
admitting  countable  orthogonal AV-sums. 
\item   $\ndasCcat$  is  the full subcategory of $\ndaCcat$ of  objects 
admitting   all very small orthogonal AV-sums. 
\end{enumerate}
\end{ddd} 


 Let $\bC$ be an additive $C^{*}$-category.
 If $\phi,\psi \colon \bC'\to \bC$ are two functors, then we can define an orthogonal sum $\phi\oplus \psi$ by
 choosing sums of $(\phi(C'),\psi(C'))$ for any object $C'$ in $\bC'$. The sum of functors is well-defined up to unitary isomorphism.
 \begin{ddd}[{\cite[Defn.~11.3]{cank}}] \label{erjigoqgfedewdq} $\bC$ is called flasque if it admits an endofunctor $S \colon \bC\to \bC$
such that there is a unitary isomorphism $\id_{\bC}\oplus S\cong S$.
\end{ddd}
If $\bC$ admits countable AV-sums, then it is flasque by \cite[Ex.~11.5]{cank}.

We consider a stable $\infty$-category $\bM$ and a functor $\Homol \colon \nCcat\to \bM$. For the following notions see also \cite[Sec.~13]{cank}.
\begin{ddd}\label{wetgjwieoferferfefwre}\mbox{}
\begin{enumerate} \item 
The functor $\Homol$ is called a homological functor if it sends unitary equivalences to equivalences and 
exact sequences to fibre sequences. \item The functor $\Homol$ is called finitary, if  $\bM$ is cocomplete and $\Homol$  in addition preserves very small filtered colimits.  \item The functor $\Homol$ is called Morita invariant if it sends Morita equivalences to equivalences.
\end{enumerate}
\end{ddd}
A zero object in a $C^{*}$-category is an object whose endomorphism algebra is the zero algebra.
A zero category is a $C^{*}$-category with only zero objects. A homological functor sends zero categories to 
zero.  

A homological functor sends products of $C^{*}$-categories to sums and preserves sums of morphisms to an additive $C^{*}$-category. It in particular annihilates flasque $C^{*}$-categories.  
 
\begin{ddd}[{\cite[Def.~13.5]{cank}}]
\label{wetogjowgrefwerfwrfw}A square of $C^{*}$-categories $$\xymatrix{\bI\ar[r]\ar[d] &\bA \ar[d] \\ \bJ\ar[r] & \bB} $$
 is called excisive if the horizontal functors are inclusions of   ideals
 and the induced map $\bA/\bI\to \bB/\bJ$ of quotients is a unitary equivalence between unital $C^{*}$-categories.
 \end{ddd}
A homological functor sends excisive squares of $C^{*}$-categories to cocartesian squares.

Our basic example of a homological functor is the $K$-theory functor for $C^{*}$-categories
\begin{equation}\label{fqwerfoihjqiowefjqwedewdqwd} \Kcat \colon \nCcat\to \Sp\, ,
\end{equation} 
see \cite[Sec.~14]{cank}.
This functor is finitary, Morita invariant,  and in addition it preserves arbitrary products of additive $C^{*}$-categories by \cite[Thm.~15.7]{cank}.

  \section{Controlled objects}\label{rgijotgweregwrgrgwrg}

\renewcommand{\small}{\mathrm{small}}
\renewcommand{\pr}{\mathrm{pr}}

%
%

In this section we introduce the $C^{*}$-category of objects in a coefficient $C^{*}$-category with $G$-action  which are controlled over a $G$-set.

Let $G$ be a very small group. For any category $\cC$ we let  $\Fun(BG,\cC)$ denote the  category of $G$-objects in $\cC$ and equivariant morphisms.
To any    $  \bD$   in $\Fun(BG,\Ccat)$
 we will associate
  the functor
\begin{equation}\label{ergwnkj23ngkwergwerg}
\bD^{G}(-)\colon G\Set\to \Ccat
\end{equation}
which associates to $X$ in $G\Set$ the unital $C^{*}$-category  of equivariant $X$-controlled objects in $\bD$ and equivariant multiplier morphisms.
We will apply this construction to 
  $\bD=\bM\bC$ for $\bC$ in $\Fun(BG,\nCcat)$. 
Provided that $\bC$ is effectively additive  we will furthermore   define a  full subfunctor 
\begin{equation}\label{ergwnkj23ngkwergwerg1}
\bC_{\pt}^{G}(-)\colon G\Set\to \Ccat
\end{equation}
 of ${(\bM\bC)}^{G}(-)$ such that $\bC^{G}_{\mathrm{pt}}(X)$ is the full subcategory of those objects which are determined on points.

%

We first introduce the $C^{*}$-category $\bD^{G}$  of $G$-objects in 
  $  \bD$.

 \begin{ddd}\label{etgiowergergwegwerg}\mbox{}\begin{enumerate} \item \label{ewrgiuhefdcsdcsdc}
A $G$-object in $  \bD$ is a pair $(D,\rho)$ of an object $D$ in $\bD$ and  a family   of unitary  morphisms   $\rho \coloneqq (\rho_{g})_{g\in G}$, $\rho_{g}\colon D\to gD$  such that
for all $ g,h$ in $G$ we have $g\rho_{h}\circ \rho_{g}=\rho_{gh}$.
\item \label{ewbiojhiobewrvcdfvsdv} A   morphism of $G$-objects $(D,\rho)\to (D^{\prime},\rho')$ is a 
  morphism $A \colon D\to D^{\prime}$ in $\bD$ such that for every $g$ in $G$ we have
$gA\circ \rho_{g}=\rho'_{g}\circ A$. The $\C$-vector space structure on the morphism sets  is inherited from $\bD$.
 \item We let $\bD^{G}$ in $\Ccat$ denote the $C^{*}$-category of  $G$-objects in $  \bD$ and {corresponding} morphisms. The composition and the  involution on $\bD^{G}$ are inherited from $\bD$.
\end{enumerate}
\end{ddd}
 
\begin{rem}
In  this remark we explain why $\bD^{G}$ is well-defined. Note that Condition  \ref{etgiowergergwegwerg}.\ref{ewbiojhiobewrvcdfvsdv} on $A$  is linear and compatible with the composition and involution. This implies    that $\bD^{G}$ is a well-defined object of $\Clincat$. 
We have a forgetful functor
$\bD^{G}\to \bD$ which sends $(D,\rho)$ to $D$ and acts as the natural inclusion on morphism spaces.
Since the Condition  \ref{etgiowergergwegwerg}.\ref{ewbiojhiobewrvcdfvsdv} is also continuous in $A$ we can conclude that
the restriction of the norm of $\bD$ to $\bD^{G}$ exhibits the latter as a $C^{*}$-category.
\hB
 \end{rem}

The construction of the category of $G$-objects is functorial. Let $ \phi\colon  \bD_{0}\to  \bD_{1}$ be a functor in $\Fun(BG,\Ccat)$.
Then we get an induced functor $\phi^{G}\colon \bD_{0}^{G}\to \bD_{1}^{ G}$ as follows:
\begin{enumerate}
\item objects: The functor $\phi^{G}$ sends the object $(D,\rho) $ in $\bD_{0}^{G}$ to the object $(\phi(D),\phi(\rho))$ in $\bD_{1}^{G}$, where $\phi(\rho) \coloneqq (\phi(\rho_{g}))_{g\in G}$.
\item morphisms: The functor $\phi^{G}$ sends a morphism $A \colon (D ,\rho )\to (D',\rho')$ to the morphism $\phi(A)\colon (\phi(D),\phi(\rho))\to (\phi(D'),\phi(\rho'))$.
\end{enumerate}
One checks that $\phi^{G}$ is well-defined, and that $  \phi\mapsto \phi^{G}$ is compatible with compositions.
We therefore obtain a functor
\begin{equation}\label{wgegwregwergwergwegergw}
(-)^{G}\colon \Fun(BG,\Ccat)\to \Ccat\, .
\end{equation}

Let  $ \bD$ be in $\Fun(BG,\Ccat)$. If $D$ is an object of $\bD$, then $\Proj(D)$ denotes the (small) set of selfadjoint projections on $D$.  
Let $ X$ be in $G\Set$.
\begin{ddd}\label{gwergoiegujowergwergwergwerg}
An equivariant $X$-controlled object in $\bD$ is
a triple $(D,\rho,\mu)$  of the following objects:
\begin{enumerate}
\item $(D,\rho)$ is  an object of  $\bD^{G}$. 
 \item $\mu$ is a function $\mu\colon \cP_{X} \to \Proj(D)$. 
\end{enumerate}
The triple must have the following properties:  
\begin{enumerate}[resume]
\item\label{etgijwogerfrewrr} $\mu$ is $G$-equivariant, i.e., we have
\begin{equation}\label{eq_mu_equivariant}
\mu(gY)=\rho_{g}^{-1}\circ g\mu(Y)\circ \rho_{g}
\end{equation}
for all $Y$ in $\cP_{X}$ and $g$ in $G$.
\item \label{riuqhfuiefqwefqef} $\mu$ is finitely additive, i.e., $\mu(\emptyset)=0$ and 
{for all subsets $Y,Z$ of $X$
with $Y\subseteq Z$ we have
$\mu(Z)=\mu(Y)+\mu(Y\setminus Z)$.}

\item \label{roqigjqeriogerqgrqfqwefqwfqwef}$\mu$ is full, i.e., 
$\mu(X)=\id_{D}$.
\end{enumerate}
\end{ddd}

 
 \begin{rem}
 Note that Condition~\ref{gwergoiegujowergwergwergwerg}.\ref{riuqhfuiefqwefqef} implies that 
  for {all} subsets $Y,Y^{\prime}$ of $X$ we have 
 $[\mu(Y),\mu(Y')]=0$. 
  \hB
 \end{rem}

Let $X$ be in $ G\Set$, and let ${ {\bD}}$ be in $\Fun(BG,\Ccat)$. 

\begin{ddd} \label{voiejvosdfvdfvsfvsfd} We define 
  the  $C^{*}$-category $\bD^{G}(X )$ as follows: 
  \begin{enumerate}
  \item objects: The objects of  $\bD^{G}(X )$ are equivariant  $X$-controlled  objects $(D,\rho,\mu)$ in $\bD$.
  \item morphisms: A morphism  $A\colon (D,\rho,\mu)\to  (D',\rho',\mu')$ in   $\bD^{G}(X )$ is a morphism $A\colon (D,\rho)\to (D^{\prime},\rho')$ in $\bD^{G}$.  The $\C$-vector space structure is inherited from $\bD^{G}$.
  \item The composition and the involution are inherited from  $\bD^{G}$. 
  \end{enumerate}
\end{ddd}
  
 It is obvious that $\bD^{G}(X)$ belongs to $\Clincat$. We have a forgetful functor $\bD^{G}(X)\to \bD^{G}$ which sends
 $(D,\rho,\mu)$ to $(D,\rho)$ and is the identity on morphisms. The restriction of the norm on $\bD^{G}$ along this functor exhibits
 $\bD^{G}(X)$ as a $C^{*}$-category.
  
 Note that $\bD^{G}(X)$ is actually a small category.

\begin{ex}\label{wiortgjowergwergregewfe}
The category $\bD^{G}(X)$ may contain pathological objects. 
For simplicity we consider the case of the trivial group.
Fix an ultrafilter $\cU$ on $X$.
   If $D$ is an object of $\bD$, then we could consider the object
 $(D,\rho,\mu)$ in $\bD(X)$ where
\[
\mu(Y)\coloneqq \begin{cases} 0& {\text{if }Y\not\in\cU}\,,\\
\id_{D}& {\text{if }Y\in \cU}\,.\end{cases}
\]
In Definition~\ref{rgoigsgsgfgg} below we will introduce a condition which excludes such objects for {free} ultrafilters.
\hB
\end{ex}
  
  \begin{ex}\label{rehiuqwefqweewdqwdew}
  The category $\bD^{G}(\emptyset)$ consists of zero objects. 
  Indeed, let $(D,\rho,\mu)$ be in $\bD^{G}(\emptyset)$. Then $\id_{D}=\mu(\emptyset)=0$ by
  additivity and fullness.
  \hB
 \end{ex}

 Assume that $f\colon X\to X^{\prime}$ is a morphism in $G\Set$.
 \begin{ddd}\label{erigoj0eorgwergrewgwgergrweg}
 We define a functor  $f_{*}\colon \bD^{G}(X )\to \bD^{G}(X' )$ as follows:
 \begin{enumerate}
 \item objects: $f_{*}(D,\rho,\mu) \coloneqq (D,\rho,f_{*}\mu)$, where $f_{*}\mu \coloneqq \mu\circ f^{-1}$.
 \item morphisms: $f_{*}$ acts as {the} identity. 
 \end{enumerate}
 \end{ddd}

 It is clear that $f_{*}\mu$ is again full and finitely additive.  Since $f$ is $G$-equivariant, $f_{*}\mu$ is again $G$-equivariant.
 Furthermore, the association $f\mapsto f_{*}$ is compatible with compositions.
 
 We fix $ \bD$ in $\Fun(BG,\Ccat)$.
 
 \begin{ddd}\label{wetgiojgreggjwregowiej} We define the functor $\bD^{G}(-)\colon G\Set\to \Ccat$ as follows:
 \begin{enumerate}
 \item objects: The functor $\bD^{G}(-)$ sends $X$ in $G\Set$ to the $C^{*}$-category   $\bD^{G}(X)$ defined in the Definition \ref{voiejvosdfvdfvsfvsfd}.
 \item morphisms: The functor $\bD^{G}(-)$ sends the morphism $f\colon X\to X'$ in $G\Set$ to the functor  $f_{*}$ defined in Definition \ref{erigoj0eorgwergrewgwgergrweg}.
 \end{enumerate}
\end{ddd}

We now consider $\bC$ in $\Fun(BG,\nCcat)$. Then we can apply the constructions above to the multiplier category $\bM\bC$ in  $\Fun(BG,\Ccat)$.
 Let $X$ be in $G\Set $, and let $(C,\rho,\mu)$ be in ${(\bM\bC)}^{G}(X)$.  
  
Recall the Definition \ref{qerghqiferwfqefedeq} of an AV-sum.
  
\begin{ddd}\label{rgoigsgsgfgg}
The object $(C,\rho,\mu)$ is determined on points if $C$ is the orthogonal AV-sum of the images of the family of projections $(\mu(\{x\}))_{x\in X}$. 
\end{ddd}
Explicitly this means that the projection $\mu(\{x\})$ is effective (see Definition \ref{wetijgowegreferfwrefwrefrf}) for every $x$ in $X$ and that $\sum_{x\in X} \mu(\{x\})$ converges strictly to $\id_{C}$.
Note that except for degenerate cases the  pathological  object  in Example~\ref{wiortgjowergwergregewfe} is not determined on points {if the ultrafilter is free.}

\begin{ddd} \label{ergoiejgwegergwergeg}
We define the $C^{*}$-category  $\bC^{G}_{\pt}(X)$  as   the full subcategory of ${(\bM\bC)}^{G}(X)$ of all objects which are determined on points.
\end{ddd}

 Let $X$ be in $G\Set $, $Y$ be a subset of $X$, and let $(C,\rho,\mu)$ be in ${(\bM\bC)}^{G}(X)$.  
Recall the Definition  \ref{wiotgwtgwregewf} of effective additivity. For categories with $G$-action we will apply this definition to the underlying $C^{*}$-category.
\begin{lem}\label{wtigowgfrefrewf}
  If $\bC$ is effectively additive and  
 $(C,\rho,\mu)$ is determined on points,  then $\mu(Y)$ is effective.  
\end{lem} 
\begin{proof}
We have a mutually orthogonal family of  effective projections $(\mu(\{x\}))_{x\in Y}$ and  $\sum_{x\in Y}\mu(\{x\})$  strictly converges to $\mu(Y)$.
Therefore $\mu(Y)$ is  effective by our assumption on $\bC$.
\end{proof}

%

Let $X$ be in $G\Set$
and let $(Y_{i})_{i\in I}$ let be a be partition of $X$ into subsets.
We consider 
  $(C,\rho,\mu)$ be in $\bC^{G}_{\pt}(X)$. 
  \begin{lem}\label{eriogjqwefqewfewfeqdewdq}
If     $\bC$ is effectively additive,  then  
 $C$ is the orthogonal AV-sum of the images of the family of projections $(\mu(Y_{i}))_{i\in I}$.
  \end{lem}
\begin{proof}   In view of Lemma \ref{wtigowgfrefrewf}
 the assumption implies that  the projections $\mu(Y_{i}) $ are effective for all $i$ in $I$.  
  By \cite[Lem. 7.10]{cank} we conclude that $C$
an AV-sum of a family of images of the family $(\mu(Y_{i}))_{i\in I}$. 
%
%
  \end{proof}

We now show that the categories $\bC_{\pt}(X)$ for $X$ in $G\Set$ assemble to a subfunctor of ${(\bM\bC)}^{G}(-)$ provided that $\bC$ is effectively additive.   
 \begin{lem}\label{ergioqjerogwfqrefeqwfqewf} If $\bC$ is effectively additive, then
 the morphism $f_{*}$ in Definition \ref{erigoj0eorgwergrewgwgergrweg} preserves the subcategories 
 of objects which are determined on points.  
 \end{lem}
 \begin{proof} Let $f\colon X\to X'$ be a morphism in $G\Set$,  and let $(C,\rho,\mu)$ be an object in $\bC^{G}_{\pt}(X)$.
 We must show that $ (C,\rho,f_{*}\mu)$ is determined on points, i.e., 
  that $C$ is isomorphic to the sum of the images of the family of projections $((f_{*}\mu)(\{x'\}))_{x'\in X'}$.
But this follows from Lemma \ref{eriogjqwefqewfewfeqdewdq} applied to the partition  $(f^{-1}(\{x'\}))_{x'\in X'}$ of $X$.
%
%
%
   \end{proof}
 
 In view of Lemma \ref{ergioqjerogwfqrefeqwfqewf} we can now make the following definition.
  Recall that $\bC$ is in  $\Fun(BG,\nCcat)$. In addition, we now  assume that $\bC$ is effectively additive.
 \begin{ddd}\label{wetgiojgreggjwregowiej1} We  define the functor $$\bC_{\pt}^{G}(-)\colon G\Set\to \Ccat$$ as follows:
 \begin{enumerate}
 \item objects: The functor $\bC_{\pt}^{G}(-)$ sends $X$ in $G\Set$ to the $C^{*}$-category   $\bC_{\pt}^{G}(X)$ defined in  Definition \ref{ergoiejgwegergwergeg}.
 \item morphisms: The functor $\bC_{\pt}^{G}(-)$ sends a morphism $f\colon X\to X'$ in $G\Set$ to the restriction of the  functor  $f_{*}$ defined in Definition \ref{erigoj0eorgwergrewgwgergrweg} to the subcategories.
 \end{enumerate}
\end{ddd}

    For fixed $X$ in $G\Set$ the constructions of ${(\bM\bC)}^{G}(X)$  and $\bC^{G}_{\pt}(X)$ are functorial for equivariant morphisms $\phi\colon \bC\to \bC'$ of   coefficient categories  
    which are in addition non-degenerate.

 We now employ the notation introduced in Definition \ref{weiotjgoegergerfwrefer}.
 Let  $  \phi \colon  \bC_{0}\to  \bC_{1}$ be a  morphism  in $\Fun(BG,\ndCcat)$. Then using  \eqref{wgegwregwergwergwegergw} we 
     define a functor
\begin{equation}\label{rewgweighogiergwerefwref}
\bM\phi^{G}(X)\colon {(\bM\bC_{0})}^{G}(X)\to {(\bM\bC_{1})}^{G}(X)
\end{equation}
    as follows:
     \begin{enumerate}
     \item objects: The functor $\bM\phi^{G}(X)$ sends the object $(C,\rho,\mu)$ in ${(\bM\bC_{0})}^{G}(X)$ to  the object $(\phi(C),\bM\phi(\rho),\bM\phi(\mu))$ in ${(\bM\bC_{1})}^{G}(X)$, where $\bM\phi(\mu)(Y) \coloneqq \bM\phi(\mu(Y))$ for all subsets $Y$ of $X$.
     \item morphisms: The functor $\bM\phi^{G}(X)$ sends a morphism $A\colon (C,\rho,\mu)\to (C',\rho',\mu')$ in ${(\bM\bC_{0})}^{G}(X)$ to the morphism $$\bM\phi(A) \colon (\phi(C),\bM\phi(\rho),\bM\phi(\mu))\to (\phi(C'),\bM\phi(\rho'),\bM\phi(\mu'))$$ in ${(\bM\bC_{1})}^{G}(X)$.
     \end{enumerate} 
We see that we actually  have constructed  a functor
\begin{equation}\label{freferwwrevervwervrvds}
G\Set\times \Fun(BG,\ndCcat)\to \Ccat\, , \quad (X,  \bC)\to {(\bM \bC)}^{G}(X)\, .
\end{equation}
The restriction of this functor to $G\Set\times \Fun(BG,\ndeCcat)$
admits the subfunctor 
  \begin{equation}\label{freferwwrevervwervrvds1}
G\Set\times \Fun(BG, \ndeCcat)\to \Ccat\, , \quad (X,  \bC)\to \bC_{\mathrm{pt}}^{G}(X)\, .
\end{equation} 
 

\section{Small objects and Roe categories}  \label{fiughiufvfdcasdcscdscca}


We consider  $ \bC$ in $\Fun(BG,\nCcat)$.

\begin{ass}\label{ogjkporgwegregrefwf} {\em As a standing hypothesis, when we talk about $\bC^{G}_{\pt}$ and objects derived from this like $ \bCgsmc $ we will in addition assume that $\bC$ is effectively additive.} \hB
\end{ass}

Later we will assume in addition that  $\bC$ admits countable AV-sums. 
 Then for $X$  in $G\Set$
the  $C^{*}$-categories ${(\bM\bC)}^{G}(X)$ and  $\bC^{G}_{\pt}(X)$  are in general too large to be interesting from a $K$-theoretic point of view. Indeed, the condition on $\bC$ implies that  ${(\bM\bC)}^{G}(X)$ and $\bC^{G}_{\pt}(X)$ also admit countable AV-sums  and are therefore flasque, see Definition \ref{erjigoqgfedewdq} and the subsequent text.
But flasque $C^{*}$-categories have trivial $K$-theoretic invariants by \cite[Prop.~13.13]{cank}.

 In this section we  define subcategories of small objects $\bCgsm(X)$ in ${(\bM\bC)}^{G}(X)$, and $\bCgsmc(X)$ in $\bC_{\pt}^{G}(X)$.
The notion of a small object   depends on  the choice of a $G$-bornology on $X$. 

We have a forgetful functor $G\Born\to G\Set$ which sends a $G$-bornological space to its underlying $G$-set (see Definition \ref{ebjoifdbsdbsfdbsbsdb}). By precomposition with this forgetful functor we can consider the functors ${(\bM\bC)}^{G}(-)$ (Definition \ref{wetgiojgreggjwregowiej}) and $\bCgc(-)$ (Definition \ref{wetgiojgreggjwregowiej1}) as functors on $G\Born$.
We will define a subfunctor 
$$\bCgsm \colon G\Born\to   \Ccat \quad \text{of}\quad G\Born \xrightarrow{{(\bM\bC)}^{G}(-)} \Ccat\,.$$
Taking Assumption \ref{ogjkporgwegregrefwf} into account 
we will furthermore define a  subfunctor 
 $$\bCgsmc\colon G\Born\to   \Ccat \quad \text{of}\quad G\Born \xrightarrow{\bCgc(-)} \Ccat\, .$$ 
 
 Let $X$ be in $G\Born$ with bornology $\cB_{X}$, and
 let $(C,\rho,\mu)$ be an object of ${(\bM\bC)}^{G}(X)$. Note that the $\mu$ takes values in  $\End_{\bM\bC}(C)$, and that $ \End_{\bC}(C)$ can be considered a subset of  $\End_{\bM\bC}(C)$.
  
\begin{ddd}\label{wqgfuihergiugwefqwfef}
$(C,\rho,\mu)$ is  small if  for every  $B$  in $\cB_{X}$ we have $\mu(B) \in \End_{\bC}(C) $.
  \end{ddd}


\begin{ddd}\label{ergioewrjgwergwergwergweg}
We  let  $\bCgsm(X)$ be the full subcategory of  ${(\bM\bC)}^{G}(X)$ of small objects.
\end{ddd}

Let $f\colon X\to X^{\prime}$ be a morphism in $G\Born$.
\begin{lem}\label{erogwergwergregweg}
The morphism $f_{*}\colon \bCg(X)\to \bCg(X')$ preserves  small objects.
\end{lem}
\begin{proof}
Let $(C,\rho,\mu)$ be in $\bCgsm(X)$. We must check that $f_{*}(C,\rho,\mu)=(C,\rho,f_{*}\mu)$ is small.
Let $B'$ be  in $\cB_{X'}$. Then we have $f^{-1}(B)\in \cB_{X}$ since  $f$ is proper. This implies that $(f_{*}\mu)(B')=\mu(f^{-1}(B))\in \End_{\bC}(C)$. 
\end{proof}

\begin{ddd}
We define the functor $\bCgsm\colon G\Born\to \Ccat$  as the subfunctor of 
 $\bC^{G}(-)\colon G\Born\to \Ccat$  which sends $X$ in $G\Born$ to the full subcategory $\bCgsm(X)$ of small objects in ${(\bM\bC)}^{G}(X)$ defined in Definition 
 \ref{ergioewrjgwergwergwergweg}.
\end{ddd}
 The functor $\bCgsm$ is well-defined in view of Lemma \ref{erogwergwergregweg}.
   
 Let $X$ be in $G\Set$, and  let
 $(C,\rho,\mu)$ be an object in $\bC^{G}(X)$.
 \begin{ddd}\label{erigojegerwfvc}
 We define the support of $(C,\rho,\mu)$ by
 $$\supp(C,\rho,\mu)\coloneqq \{x\in X\:|\: \mu(\{x\})\not=0\}\, .$$
 \end{ddd}

Note that $\supp(C,\rho,\mu)$ is a $G$-invariant subset of $X$.

Let $X$ be in {$G\Born$,} and let
  $(C,\rho,\mu)$ be an object of $ \bCgc(X)$.
  \begin{ddd}\label{rthoperthrtherthergtrgertgertretr}
  \label{wqgfuihergiugwefqwfef4444} $(C,\rho,\mu)$ is  locally
 finite if:
  \begin{enumerate}
  \item $(C,\rho,\mu)$ is small.
  \item \label{qwriuheivnfvkjvvcd9} $\supp(C,\rho,\mu)$ is a locally finite subset of $X$ (see Definition \ref{rgoiqrgrfqwfewfq}).
  \end{enumerate}
     \end{ddd}

\begin{ddd}
We let $\bCgsmc(X)$ denote the full subcategory of $\bCgc(X)$ of 
locally finite objects.
\end{ddd}


Let $f\colon X\to X^{\prime}$ be a morphism in $G\Born$.
Recall Assumption \ref{ogjkporgwegregrefwf}.
\begin{lem}\label{tgoiwgwregwegwegregw}
The morphism $f_{*}\colon \bCgc(X)\to  \bCgc(X')$ preserves
locally finite objects.
\end{lem}
\begin{proof}
By Lemma \ref{erogwergwergregweg} the morphism $f_{*}$ preserves small objects. By Lemma \ref{ergioqjerogwfqrefeqwfqewf} it preserves the condition of being determined on points.

We  now argue that  for an object $(C,\rho,\mu)$  in $\bCgc(X)$ we have
  $$\supp f_{*}(C,\rho,\mu)\subseteq f(\supp(C,\rho,\mu))\, .$$
  Let $x'$ be a point in $\supp f_{*}(C,\rho,\mu)$. Then we have $(f_{*}\mu)(\{{x'}\})=\mu(f^{-1}(\{x'\}))\not=0$.  
  Since $(C,\rho,\mu)$ is determined on points, there exists an $x$ in $f^{-1}(\{x'\})$ such that $\mu(\{x\})\not=0$. Hence $x'\in  f(\supp(C,\rho,\mu))$.
  
  Since $f$  is proper, it preserves locally finite subsets.  Consequently, $\supp f_{*}(C,\rho,\mu)$ is locally finite. 
  \end{proof}

\begin{ddd}
We define the functor $\bCgsmc\colon G\Born\to \Ccat$ as the subfunctor of $\bC^{G}_{\pt}\colon G\Born\to \Ccat$ which sends $X$ in $G\Born$ to 
the full subcategory $\bCgsmc(X)$ of   locally finite objects in $\bC^{G}_{\pt}(X)$  defined in Definition \ref{rthoperthrtherthergtrgertgertretr}.
\end{ddd}

Note that the functor $\bCgsmc$ is well-defined in view of Lemma \ref{tgoiwgwregwegwegregw}. 
%
%

Note that morphisms in $\bCgsm(X)$ or $\bCgsmc(X)$ are just morphisms in ${(\bM\bC)}^{G}$. 
In the following we define wide subcategories of these categories by putting additional conditions on the morphisms. 
These conditions will depend on an additional coarse structure on $X$ (see Definition \ref{trbertheheht}).

We have a forgetful functor $G\BC\to G\Born$ which forgets the coarse structure. By precomposition with this forgetful  functor we can consider
$\bCgsm$  and $\bCgsmc$ as functors on $G\BC$.
 Employing the coarse structure we will  construct  subfunctors
   \begin{equation}\label{sfdvpmvklsfdvsdfvs}
  \bCgtsm  \colon G\BC\to \Ccat \quad \text{of} \quad G\BC   \xrightarrow{\bCgsm} \Ccat\, .
\end{equation} and
 \begin{equation}\label{sfdvpmvklsfdvsdfverferferfs}
 \bCgtsmc\colon G\BC\to \Ccat \quad \text{of} \quad G\BC   \xrightarrow{\bCgsmc} \Ccat\, .
\end{equation}

  Let $X$ be a set,   
 $U$  be  in $\cP_{X\times X}$, and let   $Y,Y^{\prime}$ be in $\cP_{X}$.
 
 \begin{ddd}
 We say that $Y'$ is $U$-separated from $Y$ if  $ U[Y]\cap Y'=\emptyset$.
 \end{ddd}
  Let $X$ be in $G\Born$, let
   $(C,\rho,\mu) $ and $ (C^{\prime},\rho',\mu')$  be two objects in  $\bCg(X)$,  and  let $A\colon (C,\rho,\mu)\to(C^{\prime},\rho',\mu') $ be a morphism in $ \bCg(X) $. Let $U$ be in $\cP_{X\times X}$.
\begin{ddd}\label{giojerogefrefwefrewf}  We say that $A$
  is    $U$-controlled  if
for every two elements $Y,Y^{\prime}$ in $\cP_{X}$  such that $Y'$ is $U$-separated from $Y$
we have
\begin{equation}\label{qwfewfewqfewedqewd1rr}
\mu'(Y')A\mu(Y)=0\, .
\end{equation}
\end{ddd}


Assume now that $X$ is $G\BC$  with coarse structure $\cC_{X}$, and let $A$ be as above.
 \begin{ddd}
 We say that $A$ is   controlled  if there exists a  $V$ in $\cC_{X}$ such that $A$  is  $V$-controlled.
 
\end{ddd}

Let $X$ be in $G\BC$.
\begin{lem}
The controlled morphisms in $\bCgsm(X )$ form a unital $\C$-linear $*$-category.
 \end{lem}
\begin{proof} 
 We must show that the identities are controlled, and that the controlled morphisms are preserved by the $\C$-linear structure, the composition and the involution.
 The following assertions are straightforward to check.
 
 If $ (C,\rho,\mu) $ is $\bCgsm(X )$, then $\id_{C}\colon (C,\rho,\mu) \to  (C,\rho,\mu) $ is $\diag(X)$-controlled.
 
 If $A\colon (C,\rho,\mu)\to(C^{\prime},\rho',\mu') $ is $U$-controlled for some $U$ on $\cC_{X}$, then $\lambda A$ is $U$-controlled for all $\lambda$ in {$\IC$.}

 If $Y$ is $U^{-1}$-separated from $Y'$, then $Y'$ is $U$-separated from $Y$. This implies  that $A^{*}$ is $U^{-1}$-controlled provided $A$ was $U$-controlled.

 For $i$ in $\{0,1\}$ let 
  $A_{i}\colon (C,\rho,\mu)\to(C^{\prime},\rho',\mu') $ be $U_{i}$-controlled morphisms for $U_{i}$ in $\cC_{X}$.
  Then $A_{0}+A_{1}$ is $U_{0}\cup U_{1}$-controlled.
  
     We now consider the composition.
 Assume that $B\colon (C,\rho,\mu)\to(C^{\prime},\rho',\mu') $ is a $U$-controlled morphism  and
 $A\colon (C',\rho',\mu')\to(C'',\rho'',\mu'') $  is a $V$-controlled morphism in $\bCgsm(X )$.
 Then we shall see that $A\circ B$ is $V\circ U$-controlled. 
Assume that $Y''$ is $V\circ U$-separated from $Y$. 
 Then $Y''$  is $V$-separated from $U[Y]$ and $X\setminus U[Y] $ is $U$-separated from $Y$.
 We  now use the equalities $AB=A\mu'(X)B$ (since $\mu'$ is full) and $\mu'(X)=\mu'(U[Y]) +\mu'(X\setminus U[Y])$ (by finite additivity of $\mu'$) to get
 $$\mu''(Y'')AB\mu(Y)=\mu''(Y'')A\mu'(U[Y])B\mu(Y)+\mu''(Y'')A\mu'(X\setminus U[Y])B\mu(Y)=0$$
 since $\mu''(Y'')A\mu'(U[Y])=0$ and $\mu'(X\setminus U[Y])B\mu(Y)=0$.
 

 Since $\cC_{X}$ contains $\diag(X)$ and  is closed under unions, compositions and inverses we see that 
 the controlled morphisms in $\bCgsm(X)$ form a wide 
 unital $\C$-linear $*$-subcategory. 
 \end{proof}

\begin{ddd}\label{ergiowjegoerfwferfwefwe}
We let $\Cgtsm(X)$ in $\Clincat$ be the wide $\C$-linear $*$-subcategory of controlled morphisms in $\bCgsm(X)$. 
\end{ddd}

 Note that $\Cgtsm (X )$  is in general not a $C^{*}$-category. 
 \begin{ex} \label{qeroigqreffqewfqewfqewf}
 If $X$ has a maximal entourage, then $\Cgtsm (X )$ is a $C^{*}$-subcategory of $\bCgsm(X)$. 
 Indeed, if $U$ is such a maximal entourage, then 
 it suffices to require that the morphisms are $U$-controlled. 
 This condition determines a closed subcategory.\hB
  \end{ex}

Let $f\colon X\to X^{\prime}$ be  a morphism in $G\BC$.
\begin{lem}\label{ewgkowerpgwergergewgw}
The morphism $f_{*}\colon \bCgsm(X)\to \bCgsm(X')$ preserves controlled morphisms.
\end{lem}
\begin{proof}
Let $A\colon (C,\rho,\mu)\to (C',\rho',\mu')$ be a $V$-controlled morphism in $\bCgsm(X)$ for $V$ in $\cC_{X}$.
We claim that $f_{*}(A)=A\colon f_{*}(C,\rho,\mu)\to f_{*} (C',\rho',\mu')$ is $(f\times f)(V)$-controlled.

Let $Y, Y'$ be in $\cP_{X'}$ and assume that 
 $Y'$ is  $(f\times f)(V)$-separated from $Y$.
 Then $f^{-1}(Y')$ is $V$-separated from $f^{-1}(Y)$.
 This follows from the obvious inclusion
 $$V[f^{-1}(Y)]\subseteq f^{-1}(f\times f)(V)[Y]\, .$$
 {The claim follows now from}
\[(f_{*}\mu')( Y')A  (f_{*}\mu)( Y)= \mu'(  f^{-1}(  Y'))A( \mu( f^{-1}( Y))=0\,.\qedhere\]
  \end{proof}

\begin{ddd}
We define the functor
$\Cgtsm \colon G\BC\to \Clincat$
as the subfunctor of $\bCgsm \colon G\BC\to \Ccat $ which sends $X$ in $G\BC$ to
the wide $\C$-linear $*$-subcategory of controlled morphisms in $\bCgsm(X)$ defined in Definition \ref{ergiowjegoerfwferfwefwe}.
 \end{ddd}

Note that $\Cgtsm$ is well-defined in view of Lemma \ref{ewgkowerpgwergergewgw}.

Recall Assumption \ref{ogjkporgwegregrefwf}.
 \begin{ddd} \label{jergoiwregwerfrefwfref}We define the 
 functor 
  \begin{equation}\label{bwervewvcweveerwevwervwe}
\Cgtsmc\colon G\BC\to \Clincat\, , \quad X\mapsto \Cgtsm(X)\cap \bCgsmc (X)\, ,
\end{equation}
where the intersection takes place in $\bCgsm(X)$.
 \end{ddd}

Note that $\Cgtsm$ is a subfunctor of $\bCgsm$, and $\Cgtsmc$ is a subfunctor of $\bCgsmc$. 
The inclusions $\Cgtsm(X)\to  \bCgsm(X)$ or $\Cgtsmc(X)\to \bCgsmc(X)$ induce norms on the subcategories.
We can form the closures (on the level of morphism spaces) objectwise  in order to obtain $C^{*}$-category-valued subfunctors.

\begin{ddd}\label{eihioqwefqwfewfqwefqewf}\mbox{}
\begin{enumerate} \item 
We define the functor $$\bCgtsm\colon G\BC\to \Ccat$$ as follows:
\begin{enumerate}
\item objects: For $X$ in $G\BC$ we let  $\bCgtsm(X)$
be  the closure of $\Cgtsm(X )$ with respect to the norm induced from $\bCgsm(X)$.
\item\label{item_morphism_Gctrsm} morphisms:  If $f\colon X\to X^{\prime}$ is any morphism in $G\BC$, then  we let 
$f_{*} \colon \bCgtsm(X )\to \bCgtsm(X' )$ be the {corresponding} restriction of the morphism  $f_{*}\colon \bCgsm (X)\to \bCgsm(X^{\prime})$.
\end{enumerate}
 \item  \label{eghqfijewofewfewfqfedqdwedwqed}
We define the functor $$\bCgtsmc\colon G\BC\to \Ccat$$ as follows:
\begin{enumerate}
\item objects: For $X$ in $G\BC$ we let  $\bCgtsmc(X)$
be  the closure of $\Cgtsmc(X )$ with respect to the norm induced from $\bCgsmc(X)$.
\item\label{item_morphism_Gctrlf} morphisms:  If $f\colon X\to X^{\prime}$ is any morphism in $G\BC$, then  we let 
$f_{*} : \bCgtsmc(X )\to \bCgtsmc(X' )$ be the {corresponding} restriction of the morphism  $f_{*}\colon \bCgsmc (X)\to \bCgsmc(X^{\prime})$.
\end{enumerate}
\end{enumerate}

\end{ddd}
 \begin{rem}
In view of the similarity of the definition of the categories $ \bCgtsm(X )$ and $ \bCgtsmc(X )$ with the definition of Roe algebras
we call them Roe categories.
\hB
\end{rem}

 This finishes the construction of the functors \eqref{sfdvpmvklsfdvsdfvs}  {and \eqref{sfdvpmvklsfdvsdfverferferfs}.}

%
%
%
  
For fixed $X$ in $G\BC$  the constructions of $\bCgtsm(X)$ and
   $\bCgtsmc(X)$   are functorial in the datum $ \bC$ as we explain now. 
   Here we use the functors from
  \eqref{freferwwrevervwervrvds} and \eqref{freferwwrevervwervrvds1}.
   
    We consider  a morphism $ \phi\colon   \bC\to  \bD$ in $\Fun(BG,\ndCcat)$.   In the case of $\bCgtsmc$  we in addition assume that it belongs to $  \Fun(BG, \ndeCcat)$. 
 
 We first observe,  using that the restriction of $\bM\phi$ to $\bC$ is precisely $\phi$ itself,  that if $(C,\rho,\mu)$ is in $\bCgsm(X)$ or $\bCgsmc(X)$, then $(\phi(C),\bM\phi(\rho),\bM\phi(\mu))$ is in $\bD^{G}_{\mathrm{sm}}(X)$ or  $\bD^{G}_{\mathrm{lf}}(X)$. 
  
   If $A\colon (C,\rho,\mu)\to (C',\rho',\mu')$ is a morphism in $\bCgsm(X)$ which is $U$-controlled, then
   $\bM\phi(A)\colon (\phi(C),\bM\phi(\rho),\bM\phi(\mu))\to (\phi(C'),\bM\phi(\rho'),\bM\phi(\mu'))$ is also $U$-controlled.
   It follows that $\phi^{G}(X)$ in \eqref{rewgweighogiergwerefwref}  induces  by restriction functors
   \[\Cgtsm(X)\to \bD^{G,\mathrm{ctr}}_{\mathrm{sm}}(X)\, , \quad \Cgtsmc(X)\to \bD^{G,\mathrm{ctr}}_{\mathrm{lf}}(X)\, .\]  
   
    Consequently, $\phi^{G}(X)$ also preserves the completions and restricts to 
 functors
      \[\bCgtsm(X)\to \mathbf{\bar D}^{G,\mathrm{ctr}}_{\mathrm{sm}}(X)\, , \quad   \bCgtsmc(X)\to \mathbf{\bar D}^{G,\mathrm{ctr}}_{\mathrm{lf}}(X)\,.\]
In conclusion,  we actually have constructed functors 
  \begin{equation}\label{freferwwrevervwervrvdseewewe}
G\BC\times \Fun(BG, \ndCcat)\to \Ccat\, , \quad (X,\bC)\to \bCgtsm(X)
\end{equation}
and
\begin{equation}\label{freferwwrevervwervrvds1wewewewe}
G\BC\times \Fun(BG,\ndeCcat)\to \Ccat\, , \quad (X,\bC)\to \bCgtsmc(X)\, .
\end{equation}

 \section{Properties of  the functors \texorpdfstring{$\bCgtsm$}{CGctrsm} and \texorpdfstring{$\bCgtsmc$}{CGctrlf}}\label{eriugheriuvfvfvfvfvsvfvfdvsfs}

In Section \ref{fiughiufvfdcasdcscdscca} we introduced the functors   $\bCgtsm$ and  $\bCgtsmc$  from 
$G\BC$ to $\Ccat$.
 These functors are an intermediate step towards our desired coarse homology theories, which will be obtained by  composing them with  homological functors, see Definition  \ref{wetgjwieoferferfefwre}.  In the present section we prepare the verification of the axioms of a coarse homology theory by verifying similar properties for $\bCgtsm$ and  $\bCgtsmc$. Note that the target of these functors 
    is  not stable. So in particular the formulation of the corresponding excision
  property in Lemma \ref{eruigzheiugwergwergvwergr}  is not yet exactly the property described in  Definition \ref{wetgiwoergregwrgregwg}.

We fix $  \bC$ in $\Fun(BG, \nCcat)$ and adopt Assumption \ref{ogjkporgwegregrefwf}. 
Let $X$ be in $G\BC$. Recall Definition \ref{ewrgiojwgerwfweferfw} of an additive $C^{*}$-category.
\begin{lem}\label{refhfiueadscvadsc} 
 If $\bM\bC$ is   additive, then $ \bCgtsm  (X)$ and $  \bCgtsmc(X)$  are additive.
 
\end{lem}
\begin{proof}
We consider the case of $  \bCgtsm  (X)$.
 The orthogonal sum of two 
objects $(C,\rho,\mu)$ and $(C',\rho',\mu')$ of  $ \bCgtsm  (X)$
is given by 
$(C\oplus C',\rho\oplus \rho',\mu\oplus \mu')$, where $(C\oplus C',(e,e'))$ is an  orthogonal sum of the family $(C,C')$ in $\bM\bC$ which exists by our assumption. In order to see that this sum is small note that $(\mu\oplus \mu')(B)=e\mu(B)e^{*}+e'\mu'(B)e^{\prime,*}$.   Since 
  $\bC$ is an ideal in $\bM\bC$ and $\mu(B)$, $\mu'(B)$ belong to $\bC$ by the smallness of the summands,  we have $(\mu\oplus \mu')(B)\in \End_{\bC}(C\oplus C')$  for every $B$ in $\cB_X$.
  
The argument for   $  \bCgtsmc(X)$ is similar.   One checks in a straightforward manner that if
 both summands are locally finite, then so  is the sum.
\end{proof}

The functors  $ \bCgtsm  $  and $ \bCgtsmc$  are   not coarsely invariant except in degenerate cases.
But the  following lemma shows that they  are coarsely invariant (see Definition \ref{ergiweorgregergwerg}) if we consider them as functors with values  in the $(2,1)$-category $\Ccat_{2,1}$ of unital $C^{*}$-categories, unital functors and unitary isomorphisms. 

\begin{lem}\label{refhfiueareferfdscvadsc}
The functors $ \bCgtsm  $  and $ \bCgtsmc$  send  pairs of  close   morphisms (Definition \ref{wegwerfreweggerg})   to pairs of  unitarily isomorphic morphisms. 
\end{lem}
\begin{proof}  
Let $f,f^{\prime}\colon X\to  Y$ be two morphisms in $G\BC$ which are close to each other.  Then we define a
  unitary isomorphism  of functors 
$$u\colon f_{*}\to f'_{*}\colon  \bCgtsm (X)\to \bCgtsm (Y)\, , \quad u=(u_{(C,\rho,\mu)})_{(C,\rho,\mu)\in \Ob(\bCgtsm(X))}$$
such that  $u_{(C,\rho,\mu)}\colon ( C,\rho,f_{*}\mu)  \to ( C,\rho,f'_{*}\mu)$ is given by the identity  of the object $C$ considered in $\bM\bC$.
We observe that
 this morphism is equivariant, unitary, and controlled by the entourage  $(f^{\prime},f)(\diag(X))$ of $Y$.  Since $f$ and $f'$ are close to each other, this entourage belongs to $\cC_{Y}$  and we can conclude that $u_{(C,\rho,  \mu)}$ is controlled.
 
The argument for $ \bCgtsmc$ is the same.
\end{proof}

\begin{lem}\label{ergiuhwergwergergwegergerg}
The functors
$ \bCgtsm $ and $ \bCgtsmc $   are $u$-continuous (Definition \ref{wetklhgwoelgregweg9}).
\end{lem}
\begin{proof}
We consider the case of $ \bCgtsm $. The argument for  $ \bCgtsmc  $ is exactly the same.

Let $X$ be in $G\BC$.
Then we must show that the canonical morphism
\begin{equation}\label{wblkijobsbfsdfbsdfb}
\colim_{U\in \cC_{X}^{G}} \bCgtsm (X_{U})\to 
 \bCgtsm  (X)
\end{equation}
 is an isomorphism.
   For every entourage $U $ in $\cC_{X}^{G}$  
the canonical morphism 
$$i_{U}\colon \bCgtsm  (X_{U})\to  \bCgtsm (X)$$
  is the identity on the level of objects and an inclusion of a subspace on the level of 
  morphisms.  
   By  \cite[Lem.~4.7]{crosscat} we know that
   the colimit in \eqref{wblkijobsbfsdfbsdfb}  is realized by the wide subcategory of $ \bCgtsm  (X)$ whose morphism spaces  are given by $$
\overline{\bigcup_{U\in \cC_{X}^{G}} \Hom_{  \bCgtsm (X_{U})}((C,\rho,\mu),(C',\rho',\mu'))}\, .$$
Since this space contains the  subspace 
  $  \Hom_{  \Cgtsm (X )}((C,\rho,\mu),(C',\rho',\mu'))$ which is dense in
 $ \Hom_{ \bCgtsm (X)}((C,\rho,\mu),(C',\rho',\mu'))$ by definition, 
  we conclude that  \eqref{wblkijobsbfsdfbsdfb} is an isomorphism.
%
\end{proof}


Recall Definition \ref{erjigoqgfedewdq} 
of the notion of   a flasque $C^{*}$-category, 
and Definition \ref{qriugoqergreqwfqewfqef} of a flasque $G$-bornological coarse spaces.

\begin{lem} \label{ergoijeorgqgrwefqewfqwef} 
If $\bC$ admits countable AV-sums, then 
  the functors  $ \bCgtsm $  and $ \bCgtsmc$ send  flasque $G$-bornological coarse spaces   to flasque $C^{*}$-categories.
 \end{lem}
\begin{proof}  
 We first  consider  the case of $\bCgtsm$.
 Note that   finite AV-sums  in a $C^{*}$-category   are the same thing as sums in its multiplier category. Therefore our assumption on $\bC$ implies that $\bM\bC$     is additive.
It follows   from Lemma \ref{refhfiueadscvadsc} that
 $\bCgtsm (X)$ is additive.
 
 Assume that $X$ is flasque and that flasqueness is implemented by  $f\colon X\to X$.  We define an endofunctor
$S\colon  \bCgtsm (X)\to  \bCgtsm (X)$ as follows:
\begin{enumerate}
\item objects: The functor $S$ sends an object $(C,\rho,\mu)$ of $ \bCgtsm (X) $ to
$$S(C,\rho,\mu) \coloneqq \big(\bigoplus_{  \nat}  C  ,  \oplus_{  \nat}  \rho, \oplus_{n\in \nat} f^{n}_{*}\mu\big)\, .$$
Note that this involves a choice of a representative of an orthogonal AV-sum $(\bigoplus_{  \nat}  C  ,(e_{n})_{n\in \nat})$ which exists by the assumption on $\bC$. We use \cite[Lem. 7.8]{cank} in order to see that 
$ \oplus_{  \nat}  \rho$ and  $\oplus_{n\in \nat} f^{n}_{*}\mu  $ are well-defined.
\item morphisms: The functor $S$ sends a morphism $A\colon (C,\rho,\mu)\to (C',\rho',\mu')$ in $ \bCgtsm (X)$ to $  S(A) \coloneqq \bigoplus_{\nat}A \colon S(C,\rho,\mu)\to S(C',\rho',\mu')$, where we again use 
 \cite[Lem. 7.8]{cank} in order to see that $\oplus_{\nat }A$ is a well-defined multiplier morphism.
\end{enumerate}

 
 It is   easy to see that  $S(C,\rho,\mu)$ satisfies the conditions listed in Definition \ref{gwergoiegujowergwergwergwerg}. 
 In order check that $S(C,\rho,\mu)$ is small (see Definition \ref{wqgfuihergiugwefqwfef}), consider $B$ be in $\cB_{X}$. By Condition~\ref{qriugoqergreqwfqewfqef}.\ref{qrgioqrgfregqfqefqwefqef}  
   there exists an $n_{0}$ in $\nat$ such that $f^{n_{0}+1}(X)\cap B=\emptyset$. We then have
 $$(\oplus_{n\in \nat} f^{n}_{*}\mu)(B)
=\oplus_{n=0}^{n_{0}}   \mu((f^{n})^{-1}(B)) \oplus (\oplus_{n\ge n_{0}+1}0)\, .$$
Note that $\oplus_{n=0}^{n_{0}}   \mu((f^{n})^{-1}(B))$ is a short-hand {notation} for $\sum_{n=0}^{n_{0}} e_{n}\mu((f^{n})^{-1}(B)) e_{n}^{*}$. 
 Since   $(f^{n})^{-1}(B)\in \cB_{X}$
by properness of $f$  and {since} $(C,\rho,\mu)$ is small, we have $\mu((f^{n})^{-1}(B)) \in \End_{\bC}(C)$. Since $\bC$ is an ideal in $\bM\bC$, we have   $e_{n}\mu((f^{n})^{-1}(B)) e_{n}^{*}\in \End_{\bC}(\bigoplus_{\nat }C)$. It follows that $\oplus_{n=0}^{n_{0}}   \mu(f^{-1}(B))\in 
\End_{\bC}(\bigoplus_{\nat }C)$.

We now assume that  $A$ is $U$-controlled. Then one can check similarly as in the proof of Lemma \ref{ewgkowerpgwergergewgw} that
$S(A)$ is $V$-controlled for $V \coloneqq \bigcup_{n\in \nat} (f^{n}\times f^{n})(U)$.  Since $V\in \cC_{X}$ by
Condition \ref{qriugoqergreqwfqewfqef}.\ref{qrgioqrgfregqfqefqwefqef1},  
 we see that
$S(A)$ is controlled.
By continuity of the functor $S$ on morphism spaces, if $A$ can be approximated in norm by invariant controlled morphisms, then $S(A)$ has the same property.

We can adopt the choice for the sum of functors on the left-hand side {in the following equation} such that 
  \begin{equation}\label{veflknwlefvwevwevwev}
\id_{ \bCgtsm (X)}\oplus  (\bCgtsm (f)\circ S) = S\,.
\end{equation}
Again by assumption on $f$ we know that $f$ is close to $\id_{X}$.
From Lemma \ref{refhfiueareferfdscvadsc} we conclude  that $ \bCgtsm (f)$ is unitarily isomorphic to $\id_{ \bCgtsm (X)}$.  
Therefore  $ \bCgtsm (f)\circ S$ is unitarily isomorphic  to $S$, and {hence} $\id_{ \bCgtsm (X)}\oplus S$ is unitarily  isomorphic to $S$.

In the case of  $ \bCgtsmc$  we show by  an inspection of the construction  that the functor $S$ preserves locally finite objects. \end{proof}

We now  start  working towards excisiveness of the functors $\bCgtsm$ and $\bCgtsmc$ which will be  stated  in  a precise way in  Lemma~\ref{eruigzheiugwergwergvwergr}.
 
Let $X$ be in $G\BC$, and let  $Y$ be a $G$-invariant subset of $X$. \begin{ddd}  We   define the   wide subcategory
$\bCgtsm (Y\subseteq X)$ of $\bCgtsm(X)$ with the spaces of morphisms given by 
$$\Hom_{\bCgtsm(Y\subseteq X)}((C,\rho,\mu),(C^{\prime},\rho',\mu')) \coloneqq \mu'(Y) \Hom_{\bCgtsm( X)}((C,\rho,\mu),(C^{\prime},\rho',\mu'))\mu(Y)\, .$$
\end{ddd}
The $C^{*}$-category  $\bCgtsm(Y\subseteq X)$ is unital with the identity $\mu(Y)$ on the object $(C,\rho, \mu)$. But in general  the inclusion $\bCgtsm(Y\subseteq X)\to 
\bCgtsm(X)$ is only a  morphism  in $\nCcat$.
The inclusion $Y\to X$ induces  a  morphism
\begin{equation}\label{qerfjbjkqwefqwefqwef} 
\bCgtsm(Y)\to \bCgtsm(Y\subseteq X)
\end{equation} in $\Ccat$.
It is injective on objects and  fully faithful. 
Similarly we get a  morphism
 \begin{equation}\label{qerfjbjkqwefqwefqwefccc} 
\bCgtsmc(Y)\to\bCgtsmc (Y\subseteq X)\,.
\end{equation}

 \begin{lem} \label{oirjwoeigwrgwerg}
\ \begin{enumerate} \item If $\bM\bC$ is idempotent complete, then \eqref{qerfjbjkqwefqwefqwef} is an  unitary equivalence.
\item The functor  \eqref{qerfjbjkqwefqwefqwefccc} is a  unitary equivalence.
\end{enumerate}
%
\end{lem}
\begin{proof}
We give the argument for   \eqref{qerfjbjkqwefqwefqwef}. The case of \eqref{qerfjbjkqwefqwefqwefccc} is analoguous.
It 
 is similar to a part of the proof of \cite[Prop.~8.82]{buen}.
 Let us denote the functor \eqref{qerfjbjkqwefqwefqwef} by
 $$i\colon \bCgtsm(Y)\to \bCgtsm(Y\subseteq X)\, .$$
 We construct an inverse functor $$q \colon \bCgtsm(Y\subseteq X)\to \bCgtsm(Y)$$ 
  as follows:
  \begin{enumerate}
 \item \label{ergiowegergwregw}objects:  The functor $q$ sends an object $(C,\rho,\mu)$ of   $\bCgtsm(Y\subseteq X)$ to the object  $q(C,\rho,\mu) \coloneqq (C',e^{*}\rho e,e^{*}\mu e)$ in $\bCgtsm(Y)$, where    $(C',e)$ 
 is a choice of a representative  of {the} image of the projection $\mu(Y)$ which exists
since we assume that $\bM\bC$ is idempotent complete.
In the case of  \eqref{qerfjbjkqwefqwefqwefccc} the existence of $(C',e)$ already follows
 from the standing hypothesis of effective additivity.
 \item morphisms: Let $A\colon (C,\rho,\mu) \to (\tilde C,\tilde \rho,\tilde \mu)$ be a morphism in $\bCgtsm(Y\subseteq X)$. Then we define $q(A) \coloneqq \tilde e^{*}Ae\colon (C',e^{*}\rho e,e^{*}\mu e)\to (\tilde C',\tilde e^{*}\tilde \rho \tilde e,\tilde e^{*}\tilde \mu \tilde e)$, where {$(C',e)$ and} $(\tilde C',\tilde e)$ are the choices of the representatives for the images of {$\mu(Y)$ and $\tilde{\mu}(Y)$ made above.}
\end{enumerate}
 Above we have used a shortend notation. More precisely, $e^{*}\rho e$ is the family of morphisms $({(ge)}^{*}\rho_{g} e \colon C'\to g C')_{g\in G}$
 and $e^{*}\mu e$ is the function which sends $Z$ in $\cP_{Y}$ to the projection  $e^{*}\mu(Z)e$ in $\End_{\bM\bC}(C')$.   
 It is straightforward to verify that $q$ is well-defined. 
The definition of $q$ on morphisms extends to $\bCgtsm( X)$,
 but   this might not give a functor since  it  is then not necessarily compatible with the composition of morphisms.

We now construct the unitary isomorphism
$ u\colon q\circ i\to \id$. 
Let $(C,\rho,\mu)$ be in $\bCgtsm(Y)$. Assume that
$(C',e)$ is the chosen representative of  the image of $\mu(Y)$ in the construction of $q$. In this case $\mu(Y)=\id_{C}$ and $e$ is a unitary  isomorphism. We set
$$u_{(C,\rho,\mu)} \coloneqq e \colon (C',e^{*}\rho e,e^{*}\mu e)\to (C,\rho,\mu)\, .$$
It is straightforward to check that $u$ is natural.

Finally, we construct a unitary isomorphism $v \colon i \circ q\to \id$.
 Let $(C,\rho,\mu)$ be in $\bCgtsm(Y\subseteq X)$.
 Let $(C',e)$ be the representative of the image of $\mu(Y)$ 
 in the construction of $q$. Then
 $(i\circ q)(C,\rho,\mu)=(C',e^{*}\rho e,e^{*}\mu e)$. 
 We again set $$v_{(C,\rho,\mu)} \coloneqq e\colon (C',e^{*}\rho e,e^{*}\mu e)\to (C,\rho,\mu)$$ noting that this map is a unitary  isomorphism in $\bCgtsm(Y\subseteq X)$ with inverse $e^{*}$. It is straightforward to check that $v$ is natural.
 \end{proof}

Let $\cY=(Y_{i})_{i\in I}$ be an equivariant  big family in $X$ (see Definition \ref{gqrfqwerfqfwefwefqwf}).  
Using that $I$ is filtered, it is easy to see that we have  a  wide, possibly non-unital $\C$-linear subcategory 
$ \Cgtsm(\cY\subseteq X)$ of $\bCgtsm( X)$    with the morphism spaces \begin{equation}\label{vwvsdvasdvdsva}
\Hom_{ \Cgtsm(\cY\subseteq X) }((C,\rho,\mu),(C^{\prime},\rho',\mu')) \coloneqq \bigcup_{i\in I}\Hom_{\bCgtsm (Y_{i}\subseteq X)}((C,\rho,\mu),(C^{\prime},\rho',\mu')) \, .
\end{equation} 
 
   \begin{ddd} We define the subcategory 
 $\bCgtsm (\cY\subseteq X)$ of $ \bCgtsm( X)$ to  be the closure of  
   $  \Cgtsm(\cY\subseteq X)$  with respect to the norm induced from $ \bCgtsm(X)$.
\end{ddd} 
   
 By construction we have a canonical morphism
 $$  \bCgtsm(Y_{(-)}\subseteq X)\to \underline{\bCgtsm (\cY\subseteq X) }$$
 in $\Fun(I,\nCcat)$, where $\underline{\smash{-}}$ stands for forming the constant $I$-diagram. By adjunction
 it induces a functor
 \begin{equation}\label{oijoiervevw}
 \colim_{i\in I}\bCgtsm(Y_{i}\subseteq X)\to\bCgtsm(\cY\subseteq X) 
\end{equation}
in $\nCcat$.
By replacing  the subscript $\mathrm{sm}$ by $\mathrm{lf}$ in the preceeding discussion we get a  functor
\begin{equation}\label{oijoiervevwccc}
\colim_{i\in I}\bCgtsmc (Y_{i}\subseteq X)\to \bCgtsmc (\cY\subseteq X) \, .
\end{equation}

\begin{lem}\label{erguieqrogerwgregwergeg}
 The functors  in   \eqref{oijoiervevw} and \eqref{oijoiervevwccc}  are    isomorphisms. 
\end{lem}
\begin{proof}
We consider 
 the case of \eqref{oijoiervevw}. The case of \eqref{oijoiervevwccc} is analoguous.

  For all $i$ in $I$  
the canonical functor 
$\bCgtsm (Y_{i} \subseteq X)\to 
\bCgtsm (X)$
 is the  identity on the level of objects and it is injective on the spaces of morphisms. The assertion now follows from \cite[Lem.\ 4.7]{crosscat} which provides an explicit description of colimits of this type. 
\end{proof}
 
Recall the Definition \ref{rwgiojworgergrwegrefr} 
of an ideal inclusion in $\nCcat$.

 \begin{lem}\label{ewrgiowergrwegergwreg}
The functors
$$\bCgtsm(\cY\subseteq X)\to \bCgtsm(X) \, , \quad
\bCgtsmc(\cY\subseteq X)\to \bCgtsmc(X)$$
 are inclusions of  ideals.
\end{lem}
\begin{proof} 
We consider 
the case with subcript $\mathrm{sm}$. The other {case} is analogous.

By construction  the subcategory 
 $\bCgtsm(\cY\subseteq X)$ of $ \bCgtsm(X)$ is wide and its morphism spaces are closed subspaces. 
It remains to show that it is an ideal.    
Because of the presence of the involution it suffices to verify the condition of a left ideal.

  Let $i$ be in $I$, $A$ be $U$-controlled  in $\Cgtsm (X) $, and
$B$ be  in $\bCgtsm (Y_{i}\subseteq X) $ such that $A$ and $ B$ are composable. 
Then $A\circ B$ is   in $ \bCgtsm (Y_{i'}\subseteq X) $, where $i'$ in $I$ is such that $i\le i'$ and
$U[Y_{i}]\subseteq Y_{i'}$. Such an element $i'$ exists since the family $\cY$  is assumed to be   big, see Condition \ref{gqrfqwerfqfwefwefqwf}.\ref{weoigjwoegergergergwerg}.

Since $\Cgtsm (X)$ is dense in $\bCgtsm(X)$,  
$\Cgtsm (\cY \subseteq X)$ is dense in $\bCgtsm (\cY \subseteq X)$, and  
 $ \bCgtsm(\cY\subseteq X)$ is closed  in  $\bCgtsm(X)$ 
 we see that
 if  $A$  in $\bCgtsm(X)$ and
$B$  in $\bCgtsm(\cY\subseteq X) $ are composable, then 
$A\circ B\in  \bCgtsm (\cY \subseteq X) $.
\end{proof}

%
%


Let $X$ be in $G\BC$, and let
  $(\cY,Z)$ be an equivariant  complementary pair on $X$ (see Definition \ref{eqrgklefqwewefqeff}). 
 Recall the Definition \ref{wetogjowgrefwerfwrfw} of an excisive square of $C^{*}$-categories.
\begin{lem} \label{eruigzheiugwergwergvwergr}  In the case of  $\bCgtsm$ we will in addition  assume that $\bM\bC$ is  idempotent complete.   The squares
$$\xymatrix{ \bCgtsm((Z\cap \cY)\subseteq Z)\ar[r]\ar[d]& \bCgtsm(Z )\ar[d]\\
 \bCgtsm(\cY\subseteq X)\ar[r]& \bCgtsm(X)} \qquad \xymatrix{\bCgtsmc((Z\cap \cY)\subseteq Z)\ar[r]\ar[d]&\bCgtsmc(Z) \ar[d]\\
\bCgtsmc(\cY\subseteq X)\ar[r]&\bCgtsmc(X)}$$
are excisive.
\end{lem}
\begin{proof} 
 We   consider the case with subcript $\mathrm{sm}$.  
 The other case is analogous.

The horizontal functors are  inclusions of ideals by
 Lemma \ref{ewrgiowergrwegergwreg}. The right vertical functor  is a unital morphism between unital $C^{*}$-categories, and hence the quotients of the horizontal functors are unital and the induced functor  $$\bar i\colon  \bCgtsm (Z)/ \bCgtsm((Z\cap \cY)\subseteq Z)\to  \bCgtsm(X)/
   \bCgtsm(\cY\subseteq X)$$ between these quotients is a unital functor.
 It remains to verify that this functor
  is a unitary equivalence. We define a functor 
   $$\bar p\colon \bCgtsm(X)/
   \bCgtsm(\cY\subseteq X)\to  \bCgtsm (Z)/ \bCgtsm((Z\cap \cY)\subseteq Z)$$ as follows:
   \begin{enumerate} 
   \item  \label{wrtiohwgergwergweg}
   Let $(C,\rho,\mu)$ be an object of  $\bCgtsm(X)/
   \bCgtsm(\cY\subseteq X)$, i.e., an object of $\bCgtsm(X)$.
   Then we choose a representative $(C^{\prime},e)$ of the image of $\mu(Z)$. We define $\bar p(C,\rho,\mu) \coloneqq (C',e^{*}\rho e,e^{*}\mu e)$ 
 considered as an object of   $ \bCgtsm(Z)/ \bCgtsm((Z\cap \cY)\subseteq Z)$. It is here where use the assumption on $\bM\bC$. In the case of $\mathrm{lf}$ the existence of this image already follows from the effective additivity from Assumption \ref{ogjkporgwegregrefwf}.

   \item morphisms:
If $[A]$ is a morphism in $ \bCgtsm(X)/
   \bCgtsm(\cY\subseteq X)$ with representative  $A \colon (C,\rho,\mu)\to (\tilde C,\tilde \rho,\tilde \mu)$,   then we define
\[
\bar p([A]) \coloneqq [\tilde e^{*}A e] \colon (C',e^{*}\rho e,e^{*}\mu e)\to (\tilde C',\tilde e^{*}\tilde \rho \tilde e,\tilde e^{*}\tilde \mu \tilde e)\,.
\]
Here $(\tilde C',\tilde e)$ is the choice for
    $(\tilde C,\tilde \rho,\tilde \mu)$ made in \ref{wrtiohwgergwergweg}.
\end{enumerate}
    
  We first check that $\bar p$ is well-defined. Note that 
   $ \mu(Y)e=\mu(Y\cap Z)e=ee^{*}\mu(Y\cap Z)e$
   for $e$ as in the construction of $\bar p$ and any subset $Y$ of $X$.
   For $i$ in $I$ the endomorphism   $e^{*}\mu(Y_{i}\cap Z)e$ of the object $(C',e^{*}\rho e,e^{*}\mu e)$  in $\bCgtsm(Z)$ belongs to the ideal
   $\bCgtsm((Z\cap \cY)\subseteq Z)$.
  
  Assume that
  $A \colon (C,\rho,\mu)\to (\tilde C,\tilde \rho,\tilde \mu)$
   is in $ \bCgtsm(Y_{i}\subseteq X)$ for some $i$ in $I$. 
     Then we have the equalities
  $$\tilde e^{*}A e= \tilde e^{*} A  \mu(Y_{i}) e=   (\tilde e^{*} Ae) (e^{*}  \mu(Z\cap Y_{i}) e)\, .$$ This morphism belongs  to $\bCgtsm((Z\cap \cY)\subseteq Z)$.
  
  Since the ideal is closed, we conclude that
  $\tilde e^{*}A e$ belongs to $\bCgtsm((Z\cap \cY)\subseteq Z)$
  for all $A \colon (C,\rho,\mu)\to (\tilde C,\tilde \rho,\tilde \mu)$
 in $  \bCgtsm(\cY\subseteq X)$.

  We now show that $\bar p$ is compatible with compositions.
  Let  $A \colon (C,\rho,\mu)\to (\tilde C,\tilde \rho,\tilde \mu)$     and $B \colon (\tilde C,\tilde \rho,\tilde \mu) \to (\tilde{\tilde C},\tilde{\tilde \rho},\tilde{\tilde \mu})$ be morphisms  in $ \bCgtsm( X)$.  For the moment we assume that $A$ is $U^{-1}$-controlled for some entourage $U$ in $\cC_{X}$.  
 There exists $i$ in $I$ such that $X\setminus Z\subseteq Y_{i}$. Using that $\cY$ is big we can choose $i'$ in $I$ such that   $U[Y_{i}]\subseteq Y_{i'}$.    We have the equality 
  $$\tilde{\tilde e}^{*}BAe-\tilde{\tilde e}^{*}B\tilde e \tilde e^{*}Ae=
  \tilde{\tilde e}^{*}B\tilde  \mu(X\setminus Z) Ae  
  =(
  \tilde{\tilde e}^{*}B\tilde  \mu(X\setminus Z) Ae)(e^{*}\mu(Y_{i'}\cap Z)e)
   \, .$$
  The right-hand side belongs to the ideal $\bCgtsm((Z\cap \cY\subseteq Z)$. 
  Again, since this ideal is closed we can conclude that 
  $\tilde{\tilde e}^{*}BAe-\tilde{\tilde e}^{*}B\tilde e \tilde e^{*}Ae$ belongs to $\bCgtsm((Z\cap \cY\subseteq Z)$  {for} all
   $A\colon (C,\rho,\mu)\to (\tilde C,\tilde \rho,\tilde \mu)$     and $B\colon (\tilde C,\tilde \rho,\tilde \mu) \to (\tilde{\tilde C},\tilde{\tilde \rho},\tilde{\tilde \mu})$  in $ \bCgtsm( X)$.

%
  
  We now construct an isomorphism $u \colon \bar i\circ \bar p\to \id$.
 Let $(C,\rho,\mu)$ be an object of the quotient  $\bCgtsm(X)/
   \bCgtsm(\cY\subseteq X)$, i.e., an  object of $\bCgtsm(X)$.
   Then $$u_{(C,\rho,\mu)} \coloneqq [e] \colon  (C',e^{*}\rho e,e^{*}\mu e)\to (C,\rho,\mu)   $$
 is a unitary isomorphism in $ \bCgtsm(X)/
   \bCgtsm(\cY\subseteq X)$   with inverse $[e^{*}]$.
 It is straightforward to check that $u$ is natural.

 Finally, we construct the {unitary} isomorphism $v \colon \bar p\circ \bar i\to \id$.
 Let $(C,\rho,\mu)$ be an object of $\bCgtsm(Z)/ \bCgtsm((Z\cap \cY)\subseteq Z)$, i.e., an object of $\bCgtsm(Z)$ . Then $$v_{(C,\rho,\mu)} \coloneqq [e] \colon  (C',e^{*}\rho e,e^{*}\mu e)\to  (C,\rho,\mu) $$
 is an isomorphism in   $\bCgtsm(Z)/ \bCgtsm((Z\cap \cY)\subseteq Z)$. It is straightforward to check that $v$ is natural.
\end{proof}

In contrast to $\bCgtsm$ the functor $\bCgtsmc$ has an additional property:
\begin{lem}\label{ergiowergerergegwgre}
The functor $\bCgtsmc$ is continuous (Definition \ref{weifhqewiefjeefqefwefqwef}).
\end{lem}
\begin{proof}
Let $X$ be in $G\BC$ and 
let $\mathrm{LF}(X)$ be the poset of  invariant, locally finite subsets of $X$. 
 We must show that the canonical functor
 $$\colim_{L\in \mathrm{LF}(X)}\bCgtsmc(L)\to \bCgtsmc(X)$$
 is an isomorphism.

For every 
 $L$ in $  \mathrm{LF}(X)$ the inclusion $L\to X$ induces an inclusion $\bCgtsmc(L)\to \bCgtsmc(X)$
of full subcategories. By \cite[Lem.~4.7]{crosscat}  the colimit is given by the   union of these subcategories (note that this union is already a closed subcategory). 
 If $(C,\rho,\mu)$ is an object of $ \bCgtsmc(X)$, then $L \coloneqq \supp(C,\rho,\mu)$ belongs to $\mathrm{LF}(X)$.
 Consequently, $(C,\rho,\mu)$ belongs to the image of   $\bCgtsmc(L)\to \bCgtsmc(X)$.
 \end{proof}


  \section{The coarse homology theories \texorpdfstring{$\Homol\bC\cX^{G}$}{HgCXG} and \texorpdfstring{$\Homol\bC\cX_{c}^{G}$}{HgCXGc}}\label{qergiuhquifewdqewdewdqwedqewd}
 
 In this section we define the two functors $\Homol \bC\cX^{G}$ and $\Homol \bC\cX_{c}^{G}$ and show that they are coarse homology theories.
 
 Let $G$ be a group and   fix $  \bC$ in $\Fun(BG, \nCcat)$. As before we adopt Assumption \ref{ogjkporgwegregrefwf}. 
 Then by Definition \ref{eihioqwefqwfewfqwefqewf} we have the functor and subfunctor 
$$\bCgtsm \colon G\BC\to \Ccat\,, \quad \bCgtsmc \colon G\BC\to \Ccat\, .$$
Let $\bM$ be an   $\infty$-category, and  
let $\Homol\colon \nCcat\to \bM$ be a functor.
The following definition introduces some notation.
 \begin{ddd} \label{qwregojwoergwergerwffwfwerfwref}\mbox{}
 \begin{enumerate}
 \item 
  We define the  functor
 $$ \Homol\bC\cX^{G}\colon G\BC\to \bM$$ as the composition  
 $$ G\BC\xrightarrow{\bCgtsm} \Ccat\xrightarrow{\Homol} \bM\, .$$  
 \item \label{rgoijweorgwerferwfwerfwerfw} 
 We define the  functor
 $$\Homol\bC\cX_{c}^{G} \colon G\BC\to \bM$$ as the composition  
 $$ G\BC\xrightarrow{\bCgtsmc} \Ccat\xrightarrow{\Homol} \bM\, .$$  \end{enumerate}
 \end{ddd}

  The subscript $c$ in \ref{qwregojwoergwergerwffwfwerfwref}.\ref{rgoijweorgwerferwfwerfwerfw} indicates continuity, see Theorem  \ref{ergoiegwergewrgwergergwer}. 
  The superscript $G$ indicates that the construction uses $G$-objects in a $C^{*}$-category
with strict $G$-action. This should not be confused with the symbol $G$ appearing as a subscript  in Section   \ref{ewtiguhiufhriuewdwedqwdewd}
which stands for coinvariants, resp.\ crossed products. 


   Recall  Definition \ref{wefuihqfwefwffqwefefwq} of an equivariant coarse homology theory and Definition \ref{wetgjwieoferferfefwre} of a finitary homological functor.
   \begin{theorem} \label{werigowegwergregrewgreg}Assume:
\begin{enumerate} 
\item   
$\Homol$ is a finitary homological functor.
\item $\bC$  admits countable AV-sums and $\bM\bC$ is idempotent complete.
\end{enumerate} 
 Then 
$\Homol\bC\cX^{G}$  is an  equivariant  coarse homology theory.

 \end{theorem}
\begin{proof} 
Note that $\bM$ is stable since it is the target of a homological functor.
We verify the conditions listed in Definition \ref{wefuihqfwefwffqwefefwq}.

\begin{enumerate}
\item $\Homol\bC\cX^{G}$ is coarsely invariant: Let $f,f'$ be    two parallel morphisms in $G\BC$ which are close to each other. Then by Lemma \ref{refhfiueareferfdscvadsc} the functors $\bCgtsm(f)$ and $\bCgtsm(f')$ are unitarily isomorphic. 
Since $\Homol$ sends unitarily isomorphic functors to equivalent morphisms,  the morphisms 
$\Homol\bC\cX^{G}(f)$ and $\Homol\bC\cX^{G}(f')$ in $\bM$ are equivalent.
\item $\Homol\bC\cX^{G}$ is  $u$-continuous: 
Let $X$ be in $G\BC$. Then we must show that the canonical morphism
$$\colim_{U\in \cC^{G}_{X}} \Homol\bC\cX^{G}(X_{U})\to \Homol\bC\cX^{G}(X)$$ is an equivalence.
It has the following   factorization:
\begin{eqnarray*}
\colim_{U\in \cC^{G}_{X}} \Homol\bC\cX^{G}(X_{U})&=&\colim_{U\in \cC^{G}_{X}} \Homol(\bCgtsm (X_{U}))\\&\stackrel{(1)}{\to}&
 \Homol(\colim_{U\in \cC^{G}_{X}} \bCgtsm (X_{U}))\\&\stackrel{(2)}{\to}&
  \Homol( \bCgtsm (X))\\&=& \Homol\bC\cX^{G}(X)\, .
  \end{eqnarray*}
Since $\Homol$ is assumed to be finitary it preserves filtered colimits. Therefore $(1)$ is an equivalence. 
 By the Lemma \ref{ergiuhwergwergergwegergerg} the functor $\bCgtsm$ is $u$-continuous. Hence $(2)$ is also an equivalence.
 \item $\Homol\bC\cX^{G}$ is excisive:
Note that $\bCgtsm(\emptyset)$ is a zero category, see Example \ref{rehiuqwefqweewdqwdew}. Consequently, 
  $\Homol\bC\cX^{G}(\emptyset)\simeq 0_{\bM}$ since $\Homol$  
  is reduced by \cite[Lem. 13.6.2c]{cank}
    
Let now $X$ be in $G\BC$, and let $(\cY,Z)$ be an invariant complementary pair on $X$ with
 $\cY=(Y_{i})_{i\in I}$. For every $i$ in $I$ the morphism
$$\Homol\bC\cX^{G}(Y_{i})\to \Homol\bCgtsm(Y_{i}\subseteq X)$$ is an equivalence by Lemma \ref{oirjwoeigwrgwerg}  since $\Homol$  sends unitarily  equivalences to equivalences. By Lemma \ref{erguieqrogerwgregwergeg}
and  since $\Homol$ preserves filtered colimits we conclude that the canonically induced morphism
$$\Homol\bC\cX^{G}(\cY)\to \Homol\bCgtsm(\cY\subseteq X)$$  is an equivalence, too.   Similarly, 
$$\Homol\bC\cX^{G}(\cY\cap Z)\to \Homol \bCgtsm((\cY\cap Z)\subseteq Z)$$    is an equivalence.
We now consider the diagram 
$$\xymatrix{ \Homol\bC\cX^{G} (\cY\cap Z)\ar[r]^-{\simeq}\ar[d]&\Homol \bCgtsm((Z\cap \cY)\subseteq Z)\ar[r] \ar[d]& \Homol\bC\cX^{G}(Z )\ar[d]\\  \Homol\bC\cX^{G} (\cY)\ar[r]^-{\simeq}&\Homol
\bCgtsm(\cY\subseteq X)\ar[r]& \Homol\bC\cX^{G}(X)}$$
By Lemma \ref{eruigzheiugwergwergvwergr} the right square is the result of an application of $\Homol$ to an excisive square. 
It is a cocartesian square, because $\Homol$ sends excisive squares to cocartesian squares by \cite[Lem. 13.6.b]{cank}.   We conclude that the outer square is cocartesian. 
\item $\Homol\bC\cX^{G}$ vanishes on flasques: If $X$ is flasque, then $\bCgtsm(X)$ is a flasque $C^*$-category by Lemma \ref{ergoijeorgqgrwefqewfqwef}. By  \cite[Prop.\ 13.13]{cank} the functor $\Homol$ annihilates flasques. We conclude that
$\Homol\bC\cX^{G}(X)\simeq 0_{\bM}$.\qedhere
\end{enumerate}
\end{proof}

Recall the Definition  \ref{weifhqewiefjeefqefwefqwef} of continuity. 
\begin{theorem}\label{ergoiegwergewrgwergergwer}
Assume:
\begin{enumerate} 
\item   $\Homol$ is a finitary  homological functor.
\item $\bC$  is effectively additive and admits countable AV-sums. 
\end{enumerate}  Then 
$\Homol\bC\cX_{c}^{G}$  is a continuous  equivariant  coarse homology theory. 
 \end{theorem}
\begin{proof}
The proof {that it is an equivariant  coarse homology theory} is the same as for the Theorem \ref{werigowegwergregrewgreg}.  We just use the cases with subscript $\mathrm{lf}$ instead of $\mathrm{sm}$ in the cited lemmas.

Continuity follows from Lemma \ref{ergiowergerergegwgre} and the fact that $\Homol$ preserves filtered colimits.
\end{proof} 

%

Recall  Definition \ref{wqroijwoidfewdewqdqwdqew} of a strong equivariant coarse homology theory.
\begin{prop}\label{wergwjiofeewfewrq}
The coarse homology theories $\Homol \bC \cX^{G}$ and $\Homol \bC \cX^{G}_{c}$ are strong.
\end{prop}
\begin{proof}
We give the argument in the case of $\Homol \bC \cX^{G}$. The case of $\Homol \bC \cX^{G}_{c}$ is analoguous.

Let $X$  in $G\BC$ be weakly flasque (Definition \ref{qriugoqergreqwfqewfqefr324r34r3r34r34r34r}) with weak flasqueness implemented by  $f\colon X\to X$. Then we apply $\Homol$ to the Equation~\eqref{veflknwlefvwevwevwev}  
obtained in the proof of Lemma \ref{ergoijeorgqgrwefqewfqwef}. Using the additivity of $\Homol$ shown in \cite[Prop.\ 13.11]{cank}  we get 
$$\id_{\Homol \bC \cX^{G}(X)}+\Homol \bC \cX^{G}(X)(f) \circ \Homol(S) \simeq \Homol (S) $$
in $\End_{\bM}(\Homol \bC \cX^{G}(X))$.
By Condition \ref{qriugoqergreqwfqewfqefr324r34r3r34r34r34r}.\ref{toijgworegewrgwer}
we have 
$\Homol \bC \cX^{G}(X)(f) \simeq  \id_{\Homol \bC \cX^{G}(X)}$.
The resulting relation 
$\id_{\Homol \bC \cX^{G}(X)}+  \Homol(S) =\Homol (S)  $ implies that $\Homol \bC \cX^{G}(X)\simeq 0$.
\end{proof}

\begin{ex}
 We have seen in Theorem \ref{ergoiegwergewrgwergergwer}
that  the coarse homology theory $  \Homol\bC\cX^{G}_{c}$ is continuous. In the following example we show that
$  \Homol\bC\cX^{G}$ is not continuous in general.   The example is similar to \cite[Sec.~8.6]{equicoarse}.
 For simplicity we  consider the case of a trivial group and omit the symbols $G$.
 We further assume that the target $\bM$ of the homological functor $\Homol$ is compactly generated.
 
 We assume that $ \bC$ is a $C^{*}$-category which admits countable AV-sums, and which has the properties that   $\bM\bC$ is  idempotent complete and that $\Homol(\bC^{u})\not\simeq 0$ (see Example \ref{weopgwergwerrewgwerg}), where $\bC^{u}$ is the subcategory of $  \bC$ of unital objects.  Unraveling the definitions we have an equivalence  $ \bCtsm(*)\simeq  \bC^{u}$.

We  construct a factorization  
$$ \bCtsm(*)\xrightarrow{i} \bCtsm(\nat_{min,max})\to  \bCtsm(*)$$ of the identity of $\bCgtsm(*)$.
The second functor  is induced by   the projection $\nat_{min,max}\to *$ in $\BC$.  The first functor 
sends $C$  in $\bCtsm(*)$ (we omit the measure since it is a trivial datum) to the object $(\bC,\mu)$ in $\bCtsm(\nat_{min,max})$, where $\mu$ is the projection-valued measure constructed in
 Example~\ref{wiortgjowergwergregewfe} associated to some previously fixed choice of a free ultrafilter on $\nat$.
 The functor furthermore acts as the identity on morphisms.

 Since $\Homol\bC\cX(*)\simeq  \Homol(  \bCtsm(*))\not\simeq 0$ and $\bM$ is compactly generated we can find a compact object $M$ in $\bM$ and a morphism $m \colon M\to  \Homol\bC\cX(*)$ such that the composition
$$M\xrightarrow{m} \Homol\bC\cX(*)\xrightarrow{\Homol(i)} \Homol\bC\cX(\nat_{min,max})$$
is non-trivial. If $F$ is any finite subset of $\nat$, then we have a
 coarsely disjoint partition $(F,\nat\setminus F)$ of $\nat_{min,max}$ and $\mu(F)=0$ (since the ultrafilter is assumed to be free). By excision we {thus} have a decomposition
$$ \Homol\bC\cX(\nat_{min,max})\simeq 
\Homol\bC\cX( F) \oplus   \Homol\bC\cX(\nat_{min,max}\setminus F) \, , $$
 and {we} let $p_{F}\colon \Homol\bC\cX(\nat_{min,max})\to
 \Homol\bC\cX( F)$ be the projection into the first summand. 
 Since $\mu(F)=0$ it follows that $p_{F}\circ \Homol(i)\circ m\simeq 0$. 
 Since this holds for every finite subset $F$ it follows that
 $\Homol(i)\circ m$ is not in the image of 
 $$\colim_{F\in \LF(\nat_{min,max})}  \pi_{0}\map_{\bM}(M, \Homol\bC\cX( F))\xrightarrow{c} \pi_{0}\map_{\bM}(M, \Homol\bC\cX( \nat_{min,max})) \, .$$
 On the other hand, if $ \Homol\bC\cX$  would be  continuous, then the canonical map $c$ would be an isomorphism. Hence we would obtain a contradiction.
\hB
\end{ex}

   \begin{ex}\label{weopgwergwerrewgwerg}
   Here we explain how one can ensure the condition $\Homol(\bC^{u})\not\simeq 0$.
 Let  $A$ be a  unital $C^{*}$-algebra and consider the $C^{*}$-category
  $\bC=\Hilb_{c}(A)$ of very  small Hilbert $A$-modules and adjointable compact operators.   
 This category admits all very  small $AV$-sums and its multiplier category $\Hilb(A)$ is idempotent complete. Furthermore,
  $\bC^{u}= \Hilb(A)^{\fg,\mathrm{proj}}$ is the full subcategory of $\Hilb(A)$ of finitely generated, projective $A$-modules. The inclusion
  $A\to   \Hilb(A)^{\fg,\mathrm{proj}}$ is a Morita equivalence (\cite[Ex.\ 16.9]{cank}), where we consider $A$ as a $C^{*}$-category with a single object.
  If $\Homol$ is Morita invariant (see Definition \ref{wetgjwieoferferfefwre}), 
 then we have an equivalence
  $\Homol(A)\simeq \Homol(\bC^{u})$. 
  So if $\Homol(A)\not\simeq
0$, then also $\Homol(\bC^{u})\not\simeq 0$. 
We can apply this to the functor $\Homol \coloneqq \Kcat$ and choose for $A$ any unital $C^{*}$-algebra with $\Kast(A)\not\simeq 0$.
\hB
\end{ex}

%
%
%

%


Since $\Homol\bC\cX^{G}_{c}$ is a coarse homology theory it sends coarse equivalences to equivalences. In the following we show that under mild additional assumptions it also sends weak coarse equivalences to equivalences (Definition~\ref{wetgiojwergergergw}).


  Recall that  for  $X$  in $G\BC$ we {denote by}  $\LF(X)$   the poset of {$G$-}invariant  locally finite subsets of $Y$ (Definition \ref{rgoiqrgrfqwfewfq}).

If $L,H$ are sets, then a multivalued map $s \colon L\to H$ is a subset (also denoted by $s$) of $L\times H$     such that the projection  $s\to L$ is surjective.
The multivalued map $s$ is called surjective if the projection $ s\to H$ is surjective. 
For $\ell$ in $L$ the projection $s(\ell)$ of $ (\{\ell\}\times H)\cap s$ to $H$ is called the image of $\ell$ under $s$. Note that $s(\ell)$ is a subset of $H$.
In order to describe a multivalued map it suffices to describe the non-empty subsets $s(\ell)$ for all $\ell$ in $L$. For every $h$ we consider the subset
$L_{h} \coloneqq \{\ell\in L\:|\: h\in s(\ell)\}$ of $L$ and form $\tilde L \coloneqq \bigsqcup_{h \in H} L_{h}$. Then we have maps
\begin{equation}\label{3grgewwgregewg}
L\xleftarrow{q} \tilde L\xrightarrow{\tilde s} H\, ,
\end{equation}
where $\tilde s$ sends the elements of $L_{h}$ to $h$ for every $h$ in $H$ and $q$ is the natural map.
We say that $s$ is proper if $q$ and $\tilde s$ have finite fibres. 
If $G$ acts on $L$ and $H$, then $s$ is called equivariant if the subset $s$ is $G$-invariant.

 Let $f\colon X\to Y$ be a morphism in $G\BC$, and consider $L$ in  $\LF(X)$ and $H$ in $\LF(Y)$.

\begin{ddd}\label{qerigjowergergwerfwefrfwefr}
$L$ is an $f$-shadow of $H$ if there is a proper surjective equivariant  multivalued map  $s \colon L\to H$  such that $(f\times \id_{Y})(s)$ is a coarse  entourage of $Y$.
\end{ddd}

   \begin{ddd}
  $\LF(f)$ is almost surjective if
  every  element of $\LF(Y)$ admits an $f$-shadow   in $\LF(X)$.
 \end{ddd}


 \begin{prop}\label{ergijweogerwfwerfwerf}
 Assume:
 \begin{enumerate}
\item  \label{ergoijgiojgogwergrwewfe}$f$ is a weak coarse equivalence. 
 \item \label{qwkfuhqiweferwfqewfewdqewdewdeq} $\LF(f) $ is almost surjective.
 \item \label{eroigjerogerthwtg} $\bC$ admits all very small AV-sums. 
 \item\label{eioghijoffdcsdadf} $\Homol$ is Morita invariant. 
 \end{enumerate}
 Then {$f \colon \Homol \bC \cX^{G}_c(X)\to  \Homol \bC \cX^{G}_c(Y)$} is an equivalence.
  \end{prop}
\begin{proof}
In view of Assumption \ref{eioghijoffdcsdadf} it suffices to show that $f_{*} \colon \bCgtsmc(X)\to  \bCgtsmc(Y)$ is a Morita equivalence, see Definition \ref{weroigjowregrefwerferwfw}. 
To this end it suffices to show that $f_{*}$ is fully faithful, and that every object in $ \bCgtsmc(Y)$ is a subobject of an object in the image of $f_{*}$.

We first show that $f_{*}$ is fully faithful. 
Let $(C,\rho,\mu)$ and $(C',\rho',\mu')$ be objects in $\bCgtsmc(X)$.
Recall that $f_{*}(C,\rho,\mu)=(C,\rho,f_{*}\mu)$.
We have a canonical inclusion
$$\Hom_{ \bCgtsmc(X)}((C,\rho,\mu),(C',\rho',\mu'))\subseteq \Hom_{\bCgtsmc(Y)}(f_{*}(C,\rho,\mu),f_{*}(C',\rho',\mu'))\,,$$
where we consider both morphism spaces as closed subspaces of $\Hom_{{(\bM\bC)}^{G}}((C,\rho),(C',\rho'))$.
We have to show that this inclusion is an equality.
To this end we consider a morphism $A$  in the space $  \Hom_{\bCgtsmc(Y)}(f_{*}(C,\rho,\mu),f_{*}(C',\rho',\mu')) $.
Then there exists  a family  of morphisms $(A_{n})_{n\in \nat}$ in $\Hom_{\bCg}((C,\rho),(C',\rho'))$ and a family $(U_{n})_{n\in \nat}$ of coarse entourages of $Y$ such that
$\lim_{n\to \infty} A_{n}=A$ in the norm of $\Hom_{{(\bM\bC)}^{G}}((C,\rho),(C',\rho'))$ (equivalently, in $\Hom_{\bM\bC}(C,C')$), and    $A_{n} \colon (C,\rho,f_{*}\mu)\to (C',\rho',f_{*}\mu')$ is $U_{n}$-controlled  (Definition \ref{giojerogefrefwefrewf}) for every $n$ in $\nat$.

We fix $n$ in $\nat$.
By Assumption \ref{ergoijgiojgogwergrwewfe}  the subset $V_{n}:=(f\times f)^{-1}(U_{n})$ of $X\times X$ is a coarse entourage of $X$. We show now that  $A_{n} \colon (C,\rho,\mu)\to  (C',\rho',\mu')$ is $V_{n}$-controlled.
To this end we 
consider  two subsets $Z,Z'$    of $X$ such that $Z'$ is $V_{n}$-separated from $Z$.
Then $f(Z')$ is $U_{n}$-separated from $f(Z)$ and hence $f_{*}\mu'(f(Z'))A_{n}f_{*}\mu( f(Z))=0$ since $A_{n}$ is $U_{n}$-controlled.
Using $Z\subseteq f^{-1}(f(Z))$ and therefore $\mu(Z)=\mu(f^{-1}(f(Z)) \mu(Z) $, {and analogously for $Z'$,} we get 
$$\mu'(Z')A\mu(Z)= \mu'(Z')\big(   f_{*}\mu'(f(Z'))Af_{*}\mu( f(Z))\big) \mu(Z)=0\, .$$
 This implies that
$A\in \Hom_{ \bCgtsmc(X)}((C,\rho,\mu),(C',\rho',\mu'))$.

We now show that every object of $   \bCgtsmc(Y)$ is a direct summand of the image of an object of $  \bCgtsmc(X)$ under $f_{*}$.
To this end we fix  $(C,\rho,\mu)$ in $  \bCgtsmc(Y)$. Then $H \coloneqq \supp((C,\rho,\mu))$ (Definition \ref{erigojegerwfvc}) is an element of $\LF(Y)$ by Definition \ref{rthoperthrtherthergtrgertgertretr}.\ref{qwriuheivnfvkjvvcd9}.  By Definition \ref{qerigjowergergwerfwefrfwefr}  we can choose an $f$-shadow $L$ in $\LF(X)$ and a proper  surjective equivariant multivalued map $s\colon L\to H$ such that
\begin{equation}\label{eq_entourage_V_weak_coarse_equiv}
V \coloneqq (f\times \id_{Y})(s) 
\end{equation}
is a coarse entourage of $Y$.

Since $(C,\rho,\mu)$ is  determined on points (Definition \ref{rgoigsgsgfgg})  
we can choose images $u_{h} \colon C(h)\to C$ in $\bC$ of the projections $\mu(\{h\})$ for all $h$ in $H$ such that $(C,(u_{h})_{h\in H})$ represents the AV-sum of the family $(C(h))_{h\in H}$.
Since $(C,\rho,\mu)$ is locally finite and hence small we can furthermore conclude that   $C(h)$ belongs to $\bC^{u}$.
 
Associated to the multivalued map $s$ we construct the maps   $$L\xleftarrow{q} \tilde L\xrightarrow{\tilde s} H\, .$$ as in  \eqref{3grgewwgregewg}.
By assumption $q$ has finite fibres, and $\tilde s$ has finite, non-empty fibres.
 The group $G$ acts on $\tilde L$ in the canonical way such that $q$ and $\tilde s$ become equivariant.

Since $\bC$ admits all very small  orthogonal AV-sums by Assumption \ref{eroigjerogerthwtg}
we can choose an orthogonal AV-sum $(C',(e_{\tilde \ell})_{\tilde \ell\in \tilde L} )$ of the family $(C(\tilde s(\tilde \ell)))_{\tilde \ell\in \tilde L} $.
Since the summands belong to $\bC^{u}$ we see that the morphisms $e_{\tilde \ell}$ belong to $\bC$ for all $\tilde \ell$ in $\tilde L$.

In the following we extend $C'$ to an object $(C',\rho',\mu')$ of $  \bCgtsmc(X)$. We start with the construction of $\rho'$.
Let $g$ be in $G$. Then we define
${\rho'_{g}} \colon C'\to gC'$ such that the following diagram commutes for every $\tilde \ell$ in $\tilde L$:
\begin{equation}\label{trfeqwrfewfqewvfdv}
\xymatrix{C'\ar[r]^{\rho_{g}'}& gC'\\
C(\tilde s(\tilde \ell))\ar[u]^{e_{\tilde \ell}}\ar[d]_{u_{\tilde s(\tilde \ell)}}&gC(\tilde s(g^{-1}\tilde \ell))\ar[u]_{ge_{g^{-1}\tilde \ell}}\ar[d]^{gu_{\tilde s(g^{-1}\tilde \ell)}}\\
C\ar[r]^{\rho_{g}}&gC}
\end{equation}
In other words, 
we set
$$\rho'_{g} \coloneqq \sum_{\tilde \ell\in \tilde L} ge_{g^{-1}\tilde \ell} gu_{\tilde s (g^{-1}\tilde \ell)}^{*}\rho_{g}u_{\tilde s(\tilde \ell)} e_{\tilde \ell}^{*}\, .$$
One checks using \cite[Lem. 7.8]{cank} that this sum converges strictly and defines a unitary morphism in $\bM\bC$.
Using Diagram~\eqref{trfeqwrfewfqewvfdv} we further check that the cocycle condition $g\rho_{h} \rho_{g}=\rho_{gh}$
(see Definition \ref{etgiowergergwegwerg}.\ref{ewrgiuhefdcsdcsdc}) implies the condition $g\rho_{h}'\rho_{g}'=\rho'_{gh}$ for all $g,h$ in $G$.
  We set $\rho' \coloneqq (\rho'_{g})_{g\in G}$.
  
We define the projection-valued measure
 $\mu'$ on $C'$ such that $$\mu'(\{\ell\})=\sum_{\tilde \ell\in  q^{-1}(\{\ell\})}e_{\tilde \ell}e_{\tilde \ell}^{*}\, .$$   Since $s$  is proper these  sums are finite, and since the $e_{\tilde \ell}$ belong  $\bC$ also $\mu'(\{\ell\})$ belongs to $\bC$ for every $\ell$ in $L$. 
 Using the explicit formula one checks that 
 $\mu'$ satisfies   the conditions  listed Definition~\ref{gwergoiegujowergwergwergwerg}. 
  Then $(C',\rho',\mu')$ is an object of $ \bCgtsmc(X)$.

We now define a morphism
$$ U\colon (C,\rho,\mu)\to  f_{*}  (C',\rho',\mu')$$
 by 
 $$U \coloneqq \sum_{h\in H} \Big(\frac{1}{\sqrt{|\tilde s^{-1}(\{h\})|}} \sum_{\tilde \ell\in \tilde s^{-1}(\{h\})} e_{\tilde \ell}\Big) u^{*}_{h}\ .$$
Note that the fibres of $\tilde s $ are non-empty so that the prefactor is well-defined.
In order to see that the sum defining $U$ converges strictly we 
  first observe that
 $$p_{h} \coloneqq \frac{1}{|\tilde s^{-1}(\{h\})|} \sum_{\tilde \ell,\tilde \ell'\in \tilde s^{-1}(\{h\})} e_{\tilde \ell} e_{\tilde \ell'}^{*}
$$ is a projection for every $h$ in $H$, and that  the family $(p_{h})_{h\in H}$ is mutually orthogonal. We then argue using \cite[Lem. 7.4.3]{cank} and the involution that the sum  over $H$ defining $U$ converges strictly. Using the explicit formula 
we then check using  \cite[Lem. 7.4.1]{cank}    that $U$ is an isometry, and that 
 $gU \rho_{g}=\rho'_{g}U$. By construction, 
 the morphism $U$ is $V$-controlled, {where $V$ is the entourage from \eqref{eq_entourage_V_weak_coarse_equiv}.}
 
 The isometry $U$ exhibits $(C,\rho,\mu)$ as a subobject of $  f_{*}  (C',\rho',\mu')$.
\end{proof}

 \begin{ex}\label{ex_Gmaxmax}
 The map $G_{max,max}\to *$ is a weak coarse equivalence. If $G$ is infinite we have
 $\LF(G_{max,max})=\emptyset$, but $\LF(*)$ is not empty.
 Then in this case
 $\Homol \bC\cX^{G}_{c}(G_{max,max})\simeq 0$, but in general $\Homol \bC\cX^{G}_{c}(*)\simeq {\Homol}(\bC^{u,G})\not\simeq 0$, where $\bC^{u,G}$ denotes that category of  $G$-objects in the category of unital objects of $  \bC^{u}$.
 This shows the relevance of Condition~\ref{ergijweogerwfwerfwerf}.\ref{qwkfuhqiweferwfqewfewdqewdewdeq}.
\hB
\end{ex}

In order to check the Condition \ref{ergijweogerwfwerfwerf}.\ref{qwkfuhqiweferwfqewfewdqewdewdeq} the following lemma is sometimes helpful.

 Let $f \colon X\to Y$ be a morphism in $G\BC$.
  \begin{lem}\label{erigoergwferferfwrefwre}
 Assume:
 \begin{enumerate}
 \item $f$ is a weak coarse equivalence.
 \item There exists an entourage $U$ of $Y$ such that for every point $y$ in $Y$ there is $x$ in $X$ with:
 \begin{enumerate}
 \item $f(x)\in U[\{y\}]$.
 \item  $G_{x}\subseteq G_{y}$, where $G_{x}$ and $G_{y}$ denote the stabilizers of $x$ and $y$. 
 \item\label{4thgijortgwergfwrefwref} $|G_{y}/G_{x}|<\infty$.
 \end{enumerate}
  \end{enumerate}
Then $\LF(f)$ is almost surjective.   
 \end{lem}
 \begin{proof}
 After   enlarging $U$ if necessary can assume that $U$ is $G$-invariant.
Let $H$ be in $\LF(Y)$. 
 We fix a representative $h$ for every $[h]$ in $H/G$.  Then we can find $\ell_{h}$ in $X$ such that
$(f(\ell_{h}),h)\in U$,  $G_{\ell_{h}}\subseteq G_{h}$ and $|G_{h}/G_{\ell_{h}}|<\infty$.  We set $L \coloneqq \bigcup_{[h]\in H/G} G\ell_{h}$ and 
 define the $G$-invariant subset
$s \coloneqq \bigcup_{[h]\in H/G} G(\ell_{h},h)$ of $ L\times H$. Then $s$ is a $G$-equivariant surjective multivalued map and
$(f\times \id_{Y})(s)\subseteq U$.

We have $\tilde L\cong  \bigsqcup_{[h]\in H/G} G\ell_{h}$. For every $[h]$ in $H/G$ we have $|\tilde s^{-1}(h)|=|G_{h}/G_{\ell_{h}}|<\infty$. Thus by $G$-invariance all fibres of $\tilde s$ are finite.

Let $\ell$ be in $L$.
If $g\ell_{h}\in q^{-1}(\ell)$, then $gh\in U^{-1}[f(\{\ell\})]$.
Since $H$ is locally finite and $ U^{-1}[f(\{\ell\})]$ is a bounded  subset of $Y$ the intersection
  $H\cap U^{-1}[f(\{\ell\})]$ is finite. Since
  $q^{-1}(\ell)\subseteq \tilde s^{-1}(H\cap  U^{-1}[f(\{\ell\})])$ 
  and 
  the fibres of $\tilde s$ are finite we conclude that the fibres of $q$ are finite, too.
   Hence $s$ is proper.

   It remains to show that $L$ is locally finite.
       Let $B$ be a bounded subset of $X$. Since  $f$ is a coarse equivalence, it is bornological and 
   $ U^{-1}[f(B)]$ is bounded in $Y$.
       Since $H$ is locally finite, $H \cap U^{-1}[f(B)]$ is finite. Since $\tilde s$ has finite fibres and   $\tilde s(q^{-1}(L\cap B))\subseteq H \cap U^{-1}[f(B)]$, we conclude that $L\cap B$ is finite.
\end{proof}
       
\begin{ex}
Let $X$ be in $G\BC$, and let $U$ be an invariant coarse entourage. Then
  $P_{U}(X)$ is defined as the $G$-simplicial set of finitely supported  probability  measures on $X$ with $U$-bounded support. We consider $P_{U}(X)$
  as a $G$-bornological coarse space with the coarse structure induced by the path quasi-metric coming from the spherical path metric on the simplices, and with the bornology generated by the subcomplexes
  $P_{U}(B)$ for the bounded subsets $B$ of $X$.

  We consider the canonical inclusion
  $f \colon X\to P_{U}(X)$ which sends $x$ in $X$ to the $\delta$-measure at $x$.   
This inclusion satisfies the assumptions of Lemma \ref{erigoergwferferfwrefwre}. 


Let us make Proposition \ref{ergijweogerwfwerfwerf} explicit in this case for the functor $\Homol=\Kcat$ from \eqref{fqwerfoihjqiowefjqwedewdqwd}.
Note that $\Kcat$ is a Morita invariant, finitary homological functor by  \cite[Thms.\ 14.4 \& 16.17]{cank}.
We assume that $  \bC$ is in $\Fun(BG, \nCcat)$ and admits all very small AV-sums. 
Then by Proposition \ref{ergijweogerwfwerfwerf}
\begin{equation}\label{revfeervwfrewfe}
\Kcat \bC\cX^{G}_{c}(  X)\to  \Kcat \bC\cX^{G}_{c}( P_{U}(X))
\end{equation} 
is 
an equivalence for every invariant entourage $U$ of $X$.

This can be applied in order to define the equivariant coarse assembly map for $X$ by the naive equivariant generalization of \cite[Defn.~9.7]{ass} by
\begin{align*}
\mu_{X}\colon \colim_{U\in \cC^{G}} \Kcat \bC\cX^{G}_{c}(\cO^{\infty}(P_{U}(X))) & \xrightarrow{\partial} \colim_{U\in \cC^{G}} {\Sigma} \Kcat \bC\cX^{G}_{c}( P_{U}(X))\\
& \stackrel{!}{\simeq} \Sigma \Kcat \bC\cX^{G}_{c}(X)\, ,
\end{align*}
where the equivalences \eqref{revfeervwfrewfe} are used in order to get the marked equivalence, and
$\partial$ is the boundary map of the cone sequence \cite[Cor.~9.30]{equicoarse}.
 
 By     Proposition \ref{wergwjiofeewfewrq}  the equivariant coarse homology theory $ \Kcat \bC\cX^{G}_{c}$ is strong.    Using   \cite[Sec.\ 11.3]{equicoarse} we see that
$$ X\mapsto \Kcat \bC\cX^{G}_{c}\cO^{\infty}\bP(X) \coloneqq \colim_{U\in \cC^{G}} \Kcat \bC\cX^{G}_{c}(\cO^{\infty}(P_{U}(X)))$$
  gives rise  to an equivariant coarse homology theory, and that the    assembly map   refines to a natural transformation between equivariant coarse homology theories 
   $$\mu\colon \Kcat \bC\cX^{G}_{c}\cO^{\infty}\bP\to \Sigma \Kcat \bC\cX^{G}_{c}\, .$$
This is the analogue of  \cite[Cor.~11.26]{equicoarse}, but
without having to twist by a bornological coarse space with a free $G$-action (denoted by $Q$ in the reference).
\hB
\end{ex}

\begin{ex}\label{ex_classical_coarse_K_hom}
In \cite[Sec.~8.6]{buen} the coarse $K$-homology $\KX\colon \BC \to \Sp$ was defined as the composition $\Kcat \circ \bC^*$, where the functor $\bC^*\colon \BC \to \Ccat$  sends $X$ in $\BC$ to the Roe category $\bC^{*}(X)$ as introduced in \cite[Defn.~8.74]{buen}.  In this example we explain how this construction fits into the framework of the present paper. 

We work in the case of the trivial group   and therefore  in the following drop  all the notation referring to the group action. For $\bC$ we 
take  $C^*$-category $\Hilb_{c}(\C)$ of   very small Hilbert spaces and  
compact operators. It  admits all very small AV-sums.
 We  will show that there is  an equivalence of coarse homology theories
$$\KX\simeq \Kcat \bC\cX_{c} \colon G\BC\to \Sp\, .$$
This will follow from an equality of functors 
$$\bCtsmc= \bC^{*} \colon \BC\to \Ccat\, .$$  We first compare the objects. In both cases an object $(C,\mu)$  on $X$ in $\BC$ consists of a Hilbert space $C$ together with a projection valued, finitely additive measure $\mu$. For both sides we require  that $C$ is the orthogonal sum of the images of the family $(\mu(\{x\}))_{x\in X}$.  Here it is important that classical orthogonal sums of Hilbert spaces in $\Hilb(\C)$ coincide with AV-sums in $\Hilb_{c}(\C)$ by \cite[Thm. 8.4]{cank}.

For $\bC^{*}(X)$ we in addition require that $\mu(B)H$ is finite-dimensional for every
bounded subset $B$ of $X$. This condition is equivalent to the condition that
$\mu(B)\in \bC$  required for $\bCtsmc(X)$. For $\bCtsmc(X)$ we in addition require that $\supp(C,\mu)$ is locally finite. Since $\mu(B)$ is finite-dimensional and every point of the support contributes non-trivially to the dimension the set   $\supp(C,\mu)\cap B$ is indeed finite for any bounded subset $B$ of $X$.

In both cases the morphism spaces are defined as the closures of the spaces of bounded operators of controlled propagation. In order to get an actual equality of morphism spaces, we use 
 $\Hilb(\C)$ as a concrete model of the multiplier 
 category $ \bM\Hilb_{c}(\C)$  (this is possible by  \cite[Lem. 8.1]{cank})   in the construction of $\bCtsmc(X)$.

 Finally, the functoriality for morphisms in $\BC$ is defined in the same way for 
$\bCtsmc(X)$ and $  \bC^{*}(X)$.
\hB
\end{ex}

\section{The coarse homology theories \texorpdfstring{$\Homol\bC\cX_{G}$}{HgCXG} and \texorpdfstring{$\Homol\bC\cX_{c,G}$}{HgCXcG}}\label{ewtiguhiufhriuewdwedqwdewd}

In this section we introduce two further coarse homology theories $\Homol \bC\cX_{G}$ and $\Homol \bC\cX_{c,G}$. They are the analogues of the colimit theories considered in \cite{unik} which are an important ingredient in the proof of an isomorphism result for the assembly map \cite{fj}.

From  \eqref{freferwwrevervwervrvdseewewe} and \eqref{freferwwrevervwervrvds1wewewewe} 
  in the case of the trivial group we  obtain  functors  \begin{equation}\label{adsvoihqjoi3rffsdadfc}
 \BC\times  \ndCcat\to \Ccat\, , \quad (X, \bC)\mapsto \bCtsm(X)
\end{equation} and 
 \begin{equation}\label{adsvoihqjoi3rffsdadfcccc}
 \BC\times \ndeCcat \to \Ccat\, , \quad (X,\bC)\mapsto \bCtsmc(X)
\end{equation}

For a group $G$  
%
we have a fully faithful functor $$G\BC\to \Fun(BG,\BC)$$
which sends a $G$-bornological {coarse} space to the underlying bornological coarse space with the $G$-action by automorphisms.
This functor just forgets the condition that $\cC^{G}_{X}$ must be cofinal in $\cC_{X}$ (see Example \ref{egiojweogergregwefwerfwrefw}). In formulas we will not write this functor explicitly.

We now fix $  \bC$ in $\Fun(BG,\nCcat )$ and adopt Assumption \ref{ogjkporgwegregrefwf}. 
By composition we get the functors
\begin{equation}\label{dfsvwrgfvsfdvsdfvsdfvsdfvsdfv}
\btCtsm \colon G\BC  \to  \Fun(BG,\BC)  \xrightarrow{\eqref{adsvoihqjoi3rffsdadfc}} \Fun(BG,\Ccat)
\end{equation}
%
and
  \begin{equation}\label{dfsvwrgfvsfdvsdfvsdfvsdfvsdfvccc}
\btCtsmc \colon G\BC \to \Fun(BG,\BC )  \xrightarrow{\eqref{adsvoihqjoi3rffsdadfcccc}} \Fun(BG,\Ccat)\, . 
\end{equation}
We now use the maximal crossed product functor
$$-\rtimes G\colon \Fun(BG,\nCcat)\to \nCcat$$ from \cite[Defn.~5.9]{crosscat} and its restriction to unital $C^{*}$-categories.
\begin{ddd}\label{toiwjopgregwrewg}\mbox{}
\begin{enumerate}
\item
We define the functor 
$$ \bCGtsm \coloneqq \btCtsm(-) \rtimes G\colon G\BC\to \Ccat\ .$$
 \item \label{werijgowegwergwerfewrf}  We 
 define the functor 
$$ \bCGtsmc \coloneqq \btCtsmc(-) \rtimes G\colon G\BC\to \Ccat\, .$$\end{enumerate}
\end{ddd}
 
 We use the maximal crossed product since it has the necessary exactness properties   needed for the verification of   the axioms for a coarse homology theory below.
It should not be confused with the reduced crossed product
 \cite[Sec.\ 12]{cank} which  occurs in the present paper   in the statement of Proposition \ref{qergioqefweqwecqcasdc}.\ref{egijogrgergewrgwegwrewg2} below.


Recall the  Definition \ref{wetgiojwergergergw}.\ref{wetgiojwergergergw1} of a weak coarse equivalence.
 \begin{lem}\label{igjqorgqrqewfqef}
 $ \bCGtsm $ and   $ \bCGtsmc$ send  weak coarse equivalences  to unitary equivalences.
 \end{lem}
\begin{proof}
We  give the argument in the  case of  $ \bCGtsm $. The case  of    $ \bCGtsmc$ is analoguous.
If $X\to X^{\prime}$ is a weak coarse equivalence, then by the 
Lemma \ref{refhfiueareferfdscvadsc}  the  morphism  $\btCtsm (X)\to\btCtsm
(X')$  in $\Fun(BG,\Ccat)$ is a weak equivalence in the sense of   \cite[Defn.~7.6]{crosscat},  i.e., a unitary equivalence after forgetting the $G$-action. Furthermore, by   \cite[Prop.~7.9]{crosscat} the functor $-\rtimes G$ sends weak equivalences to unitary equivalences. Consequently, 
$\bCGtsm (X) \to \bCGtsm (X')$ is a unitary equivalence.
\end{proof}

\begin{lem}\label{efiowgrewgwegrg}
The functors $\bCGtsm $ and $\bCGtsmc $  are $u$-continuous.
\end{lem}
\begin{proof}
We  give the argument in the  case of  $ \bCGtsm $. The case  of    $ \bCGtsmc$ is analoguous.
{First of all note} that colimits in functor categories are calculated objectwise.
Let $X$ be in $G\BC$.  
By Lemma \ref{ergiuhwergwergergwegergerg} and since $\cC_{X}^{G}$ is cofinal in $\cC_{X}$ (see Definition \ref{trbertheheht}.\ref{igwoegwergergwrgrg}) we get
\begin{equation}\label{fvdsiuhqeiufvdfsvsdfv}\colim_{U\in \cC_{X}^{G}} \btCtsm (X_{U})\cong  \btCtsm (X) \end{equation}
 in $\Fun(BG,\Ccat)$. 
Applying $-\rtimes G$ gives
$$( \colim_{U\in \cC_{X}^{G}}  \btCtsm(X_{U}))\rtimes G\cong    \btCtsm(X)\rtimes G \,  . $$
It remains to show that the canonical functor
\begin{equation}\label{vefwvwevfvwrvwrvw}  \colim_{U\in \cC_{X}^{G}} (\btCtsm(X_{U})\rtimes G) \to  (\colim_{U\in \cC_{X}^{G}}  \btCtsm(X_{U}))\rtimes G
\end{equation}
is an isomorphism.  This follows from \cite[Prop. 7.25]{KKG}
which states that the maximal crossed product functor preserves filtered colimits of systems whose structure maps    are injective on objects. In our case the structure maps are identities.
 \end{proof}

Let $X$ be in $G\BC$, and let $(\cY,Z)$ be a complementary pair. 
\begin{lem}\label{qeriogoegregerwrec}
We have a diagram 
 \begin{equation}\label{sdfvoijoifevjsdfvfdvsfvfdvccc}
\xymatrix{ \bCGtsmc (\cY\cap Z)\ar[r]\ar[d]&  \btCtsmc ((\cY\cap Z)\subseteq Z)\rtimes G \ar[d] \ar[r]&  \bCGtsmc( Z)\ar[d]\\ \bCGtsmc (\cY)\ar[r]&  \btCtsmc( \cY \subseteq X) \rtimes G\ar[r]& \bCGtsmc (X)}
\end{equation}
and, provided $\bM\bC$ is idempotent complete, also a diagram
\begin{equation}\label{sdfvoijoifevjsdfvfdvsfvfdv}
\xymatrix{ \bCGtsm (\cY\cap Z)\ar[r]\ar[d]&  \btCtsm ((\cY\cap Z)\subseteq Z)\rtimes G \ar[d] \ar[r]&  \bCGtsm( Z)\ar[d]\\ \bCGtsm (\cY)\ar[r]&  \btCtsm( \cY \subseteq X) \rtimes G\ar[r]& \bCGtsm (X)}
\end{equation} 
where the left horizontal maps are filtered colimits of unitary equivalences and the right squares are excisive.
\end{lem}
\begin{proof} 
We  give the argument in the  case of  $ \bCGtsm $. The case  of    $ \bCGtsmc$ is analoguous.
Let $\cY=(Y_{i})_{i\in I}$.
For every $i$ in $I$ we have a diagram 
$$\xymatrix{   \btCtsm(Y_{i}\cap Z)\ar[r]\ar[d]&  \btCtsm((Y_{i}\cap Z)\subseteq Z) \ar[d] \ar[r]&  \btCtsm( Z)\ar[d]\\ \btCtsm (Y_{i} )\ar[r]&  \btCtsm( Y_{i} \subseteq X) \ar[r]&  \btCtsm ( X)}$$
in $\Fun(BG,\nCcat)$. The left horizontal  morphisms are weak equivalences by Lemma~\ref{oirjwoeigwrgwerg}. 
We apply the functor $-\rtimes G$ and get 
$$\xymatrix{  \btCtsm(Y_{i}\cap Z)\rtimes G\ar[r]\ar[d]&  \btCtsm((Y_{i}\cap Z)\subseteq Z)\rtimes G \ar[d] \ar[r]&  \btCtsm ( Z)\rtimes G\ar[d]\\  \btCtsm (Y_{i} )\rtimes G\ar[r]&  \btCtsm( Y_{i} \subseteq X)\rtimes G \ar[r]&  \btCtsm( X)\rtimes G}$$
in $\nCcat$.   Since the crossed product sends weak equivalences to unitary equivalences by   \cite[Prop.~7.9]{crosscat}, the left horizontal functors are unitary equivalences. 
We now form the colimit over $i$ in $I$ and use the notation $ \bCGtsm$ at the outer corners.
Note that the  evaluation of a functor on a big family is by definition the colimit of its values on the members.
We furthermore use that $I$ is filtered in order to remove the colimit symbol for the constant  $I$-diagrams in the right part.
We obtain {the diagram}
$$\xymatrix{ \bCGtsm(\cY\cap Z) \ar[r]\ar[d]&\colim_{i\in I}\big( \btCtsm((Y_{i}\cap Z)\subseteq Z)\rtimes G\big) \ar[d] \ar[r]&\bCGtsm(Z) \ar[d]\\ \bCGtsm (\cY  ) \ar[r]&\colim_{i\in I}\big( \btCtsm( Y_{i} \subseteq X)\rtimes G\big) \ar[r]& \bCGtsm( X) }\ .$$
In order to  get Diagram~\eqref{sdfvoijoifevjsdfvfdvsfvfdv} must  interchange the order of the colimits and the functor $-\rtimes G$ in the middle column. 
Since the structure maps of the systems are injective on objects  this is again justified by   \cite[Prop. 7.25]{KKG}.
  
It remains to show that the  right square is excisive.
  {To this end we} note that the right square is obtained from the   square
  \begin{equation}\label{vfdqrjovfevfsvdfv}
\xymatrix{ \btCtsm((\cY\cap Z)\subseteq Z) \ar[d] \ar[r]& \btCtsm ( Z)\ar[d]\\\ \btCtsm( \cY \subseteq X) \ar[r]& \btCtsm( X)}
\end{equation}
in $\Fun(BG,\nCcat)$ by applying $-\rtimes G$.
The square \eqref{vfdqrjovfevfsvdfv} is excisive after forgetting the $G$-action by Lemma \ref{eruigzheiugwergwergvwergr}. We finally use \cite[Thm.~8.14]{crosscat} stating that $-\rtimes G$ preserves excisive squares  whose right corners are unital.
\end{proof}

\begin{lem}\label{ewiogwtgwergergergwe}
If $\bC$ admits countable AV-sums, then 
the functors $ \bCGtsm$ and $ \bCGtsmc$ send  flasque $G$-bornological coarse spaces to flasque $C^{*}$-categories.
\end{lem}
     \begin{proof} 
     We  give the argument in the  case of  $ \bCGtsm $. The case  of    $ \bCGtsmc$ is analoguous.
We argue as in the proof of Lemma \ref{ergoijeorgqgrwefqewfqwef}. Let $X$ be in $G\BC$ flasque with flasqueness implemented by $f\colon X\to X$. We construct the endofunctor
     $S\colon \bCtsm (X)\to \bCtsm(X)$ (not equivariant) as in the proof of  Lemma \ref{ergoijeorgqgrwefqewfqwef}. 
      
      Using \cite[Prop.~11.4]{cank}  we extend the endofunctor $\bigoplus_{\nat} \id_{\bC}$ to a weakly equivariant
     functor $(\bigoplus_{\nat} \id_{\bC},\theta)$ (see \cite[Def.~4.1]{crosscat}). The same cocycle then provides an extension 
      $(S,\theta)$ of $S$ to  a weakly equivariant endofunctor.
 By \cite[Prop.~7.12]{crosscat} we get an endofunctor
 $$(S,\theta)\rtimes G\colon  \bCGtsm (X)\to  \bCGtsm(X)\, .$$
 The Equality~\eqref{veflknwlefvwevwevwev} induces an equality (adopting appropriate choices for the sum)
\begin{equation}\label{eroifjqoifcwecdscasdc}
\id_{ \bCGtsm(X)}\oplus   \bCGtsm(f)\circ    (S\rtimes G)=S\rtimes G\, .
\end{equation}
 By the Lemma \ref{igjqorgqrqewfqef}  and since $f$ is close to the identity we have a unitary isomorphism $  \bCGtsm(f)\circ  (S\rtimes G)\simeq S\rtimes G$. Consequently, we have a unitary isomorphism 
 $$\id_{ \bCGtsm(X)}\oplus  (S\rtimes G) \cong S\rtimes G$$
 showing that 
  $ \bCGtsm(X)$ is flasque. 
     \end{proof}

Let $\Homol \colon \nCcat\to \bM$ be a functor. 
\begin{ddd}\label{woeihgwegewrgregw9}\mbox{}
We define the functors 
$$\Homol\bC\cX_{G} \colon G\BC\xrightarrow{ \bCGtsm } \Ccat\xrightarrow{\Homol} \bM$$ 
and
$$\Homol\bC\cX_{c,G}\colon G\BC\xrightarrow{ \bCGtsmc } \Ccat\xrightarrow{\Homol}\bM\,.$$
\end{ddd}

The subscript $G$ should indicate that the functors involve taking $G$-coinvariants which is here explicitly realized in terms of  the crossed product.
The subscript $c$ indicates that 
the construction involves the continuous functor $\bCtsmc$. But note that  the functors  introduced in Definition \ref{woeihgwegewrgregw9}  are not continuous except in degenerate cases (see Example~\ref{ex_GmaxGmax_not_continuous}).

Recall Definition \ref{wetgjwieoferferfefwre} of a finitary homological functor 
and 
Definition \ref{wefuihqfwefwffqwefefwq} of an equivariant coarse homology theory.
\begin{theorem} \label{ergoijewowerfrefrefwf} Assume:
\begin{enumerate}
\item $\Homol$ is a  finitary homological functor. 
  \item $\bC$ admits countable AV-sums and  is effectively additive (and $\bM\bC$ is idempotent complete 
 in the case of $\Homol\bC\cX_{G}$).
\end{enumerate} 
  Then
$\Homol\bC\cX_{G}$  and  $\Homol\bC\cX_{c,G}$  are equivariant coarse homology theories which in addition send weak coarse equivalences to equivalences.
\end{theorem}

 \begin{proof} 
We give the proof for $\Homol\bC\cX_{G}$. The case of $\Homol\bC\cX_{c,G}$ is analoguous.
Note that $\bM$ is stable since it is the target of a homological functor.
We now verify the conditions listed in Definition \ref{wefuihqfwefwffqwefefwq}.

\begin{enumerate}
\item\label{ergiojowergrefwfwrefwer} $\Homol\bC\cX_{G}$ is coarsely invariant: Indeed, by Lemma \ref{igjqorgqrqewfqef} the functor $\bCGtsm$ sends    weak coarse equivalences to unitary equivalences. Further, $\Homol$ sends unitary equivalences to equivalences. So $\Homol\bC\cX_{G}$  sends (even weak) coarse equivalences to equivalences.
\item $\Homol\bC\cX_{G}$ is $u$-continuous: This follows from Lemma \ref{efiowgrewgwegrg} stating that $\bCGtsm$ is $u$-continuous and the fact that $\Homol$ preserves filtered colimits, since we assume that it is finitary.
 \item $\Homol\bC\cX_{G}$ is excisive: We clearly have $\Homol\bC\cX_{G}(\emptyset)\simeq 0_{\bM}$.
   Let  now $X$ be in $G\BC$, and let $(\cY,Z)$ be an invariant  complementary pair on $X$.    
We apply $\Homol$ to the Diagram~\eqref{sdfvoijoifevjsdfvfdvsfvfdv} {and get}
 $$\xymatrix{\Homol\bCGtsm(\cY \cap Z)\ar[r]\ar[d]&\Homol ( \btCtsm((\cY\cap Z)\subseteq Z)\rtimes G) \ar[d] \ar[r]& \Homol\bCGtsm ( Z)\ar[d]\\\Homol\bCGtsm(\cY )\ar[r]&\Homol (\btCtsm ( \cY \subseteq X) \rtimes G)\ar[r]&\Homol\bCGtsm ( X)}$$
 Since $\Homol$ preserves filtered colimits and sends unitary equivalences to equivalences,
 the left horizontal maps are filtered colimits of equivalences and hence equivalences. 
 The right square is the image under the homological functor $\Homol$ of an excisive square and hence a cocartesian square. 
Consequently, the outer square is cocartesian.
\item  $\Homol\bC\cX_{G}$ annihilates flasques:   Indeed, if $X$ is flasque, then $  \bCGtsm (X)$ is flasque by Lemma 
\ref{ewiogwtgwergergergwe}. Then $\Homol\bC\cX_{G}(X)\simeq 0$ by   \cite[Prop.~13.13]{cank}.\qedhere
\end{enumerate}
\end{proof}

\begin{ass}\label{wergijwoergerwfwr}{\em 
From now it will be a standing hypothesis that $\bC$ in $\Fun(BG,\nCcat)$ is effectively additive
and admits countable AV-sums, when we talk about the coarse homology theories
$\Homol\bC\cX_{c}^{G}$ and  $\Homol\bC\cX_{c,G}$.

When we consider
 $\Homol\bC\cX^{G}$ or  $\Homol\bC\cX_{G}$, then
 we will assume that $\bC$ admits countable AV-sums and that $\bM\bC$ is idempotent complete.
 
Additional assumptions will be stated if necessary.}
\hB
 \end{ass}

    

\begin{rem}
Note that the equivariant  coarse homology theories from Theorem \ref{ergoijewowerfrefrefwf} send all  weak coarse  equivalences to equivalences.  In contrast, by 
  Proposition  \ref{ergijweogerwfwerfwerf}  the functor
  $\Homol \bC\cX^{G}_{c}$  sends  weak coarse  equivalences to equivalences only under {additional} assumptions. 
\hB
\end{rem}
   

\begin{ex}\label{ex_GmaxGmax_not_continuous}
 Let  $X$ be  in $G\BC$. We assume that there exists a maximal entourage $U$ in $\cC_{X}$, and that all coarse components  of $X$  (Definition \ref{gheruierfefersfsdfsggs}) are bounded, i.e., the equivalence classes for the equivalence relation $U$   belong to $\cB_{X}$. Then 
    the canonical projection $X\to \pi_{0}(X)_{min,min}$ is a weak coarse equivalence, where $\pi_{0}(X)$ denotes the $G$-set of coarse components of $X$.
    In this case
    $$ \Homol\bC\cX_{c,G}(X)\to  \Homol\bC\cX_{c,G}(\pi_{0}(X)_{min,min})$$
  is an equivalence. 
  
A concrete example of such a space is $G_{max,max}$. 
We have $\pi_{0}(G_{max,max})\cong *$. 
   Hence   $\Homol\bC\cX_{c,G}(G_{max,max})\to \Homol\bC\cX_{c,G}(*)$ is an equivalence. Note that the target of this morphism  is given by $\Homol\bC\cX_{c,G}(*) \cong \Homol(\bC^u \rtimes G)$, which is in general non-trivial. If $G$ is infinite, then  $G_{max,max}$ does not have any invariant, locally finite subsets. 
This shows that the coarse homology theory $\Homol\bC\cX_{c,G}$ is in general not continuous. See  also Example~\ref{ex_Gmaxmax}.
\hB
\end{ex}
 
    Recall   the Definition \ref{wqroijwoidfewdewqdqwdqew}  of a strong equivariant coarse homology theory.  
\begin{prop}\label{wergwjiofeewfewrqggg}
The coarse homology theories $\Homol \bC \cX_{G}$ and $\Homol \bC \cX_{c,G}$ are strong. 
\end{prop}
\begin{proof}
We give the argument in the case of $\Homol \bC \cX_{G}$. The case of $\Homol \bC \cX_{c,G}$ is analoguous.
Let $X$  in $G\BC$ be weakly flasque (Definition \ref{qriugoqergreqwfqewfqefr324r34r3r34r34r34r}) with weak flasqueness implemented by  $f\colon X\to X$. We apply $\Homol$ to the Equation~\eqref{eroifjqoifcwecdscasdc}  
obtained in the proof of Lemma \ref{ewiogwtgwergergergwe}.
Using the additivity of $\Homol$ shown in \cite[Prop.\ 13.11]{cank}  we get 
$$\id_{\Homol \bC \cX_{G}(X)}+\Homol \bC \cX_{G}(X)(f) \circ \Homol(S\rtimes G) =\Homol (S\rtimes G) $$
in $\End_{\bM}(\Homol \bC \cX_{G}(X))$.
From Condition \ref{qriugoqergreqwfqewfqefr324r34r3r34r34r34r}.\ref{toijgworegewrgwer}
we get 
$\Homol \bC \cX_{G}(X)(f) \simeq  \id_{\Homol \bC \cX_{G}(X)}$, and the resulting relation 
$\id_{\Homol \bC \cX_{G}(X)}+  \Homol(S\rtimes G) =\Homol (S\rtimes G)  $ implies $\Homol \bC \cX_{G}(X)\simeq 0$.
\end{proof}

   \section{Calculations}\label{qwrgioeqrgefefvsfvfdvfds}

In this section we calculate the values of some of the  functors introduced in the preceding sections on bornological coarse spaces derived from $G$-sets. 
 The calculations for $\Homol \bC\cX^{G}_{c}$ are of particular relevance
 for the results stated in detail in Section~\ref{sec_CPfunctors}.


We fix $ \bC$ in $\Fun(BG,\nCcat)$  and assume that it is effectively additive. 
Applying the functor from  \eqref{dfsvwrgfvsfdvsdfvsdfvsdfvsdfvccc} to $X$ in $G\BC$ we obtain the object $\btCtsmc(X)$ in $\Fun(BG,\Ccat)$.  Instead of the reduced crossed product as in Definition \ref{toiwjopgregwrewg} we  can form, as a start, the algebraic crossed product $\btCtsmc(X)\rtimes^{\alg}G$    in $\Clincat$, see \cite[Defn.~5.1]{crosscat}.
On the other hand we have the uncompleted Roe category  $\Cgtsmc(X \otimes G_{can,min})$ in $\Clincat$, introduced in Definition \ref{jergoiwregwerfrefwfref},  where $G_{can,min}$ is defined in Example~\ref{etwgokergpoergegregegwergrg} and the symmetric monoidal structure $\otimes$ on $G\BC$  is explained in Definition~\ref{defn_symmetric_monoidal_GBC}.

 \begin{prop}\label{rgiowegrerewef}
Assume:
 \begin{enumerate}
 \item \label{efrvgervfvfdvsdfvsdfvs}
 Every locally finite subset of $X$ has a maximal entourage.
  \item \label{wkrhjorhbgrhdfghfgh9} $\bC$ admits orthogonal  AV-sums of cardinality  $|G|$.
\end{enumerate} Then
there exists a unitary equivalence
$$\phi\colon \btCtsmc(X) \rtimes^{\alg}G\to \Cgtsmc(X \otimes G_{can,min})$$
 in $\Clincat$ which is natural in $X$.
\end{prop}
\begin{proof}  
By Assumption \ref{wkrhjorhbgrhdfghfgh9}  
for every object $C$ in $\bC$ we can fix an orthogonal  AV-sum $(\bigoplus_{g\in G} gC,(e^{C}_{g})_{g\in G})$.  
 For any $X$ in $G\BC$ satisfying  Condition \ref{efrvgervfvfdvsdfvsdfvs} we will then construct a morphism 
 $$\phi \colon \btCtsmc(X )\rtimes^{\alg}G\to \Cgtsmc(X \otimes G_{can,min})$$  in $\Clincat$ which turns out to be natural in $X$. 
  To show that it is a unitary equivalence we will furthermore construct a  morphism
  $$\psi \colon \Cgtsmc(X \otimes G_{can,min})\to \btCtsmc(X )\rtimes^{\alg}G$$  in $\Clincat$  (with no naturality requirements)
and unitary  isomorphisms 
$$u \colon \psi \circ \phi \xrightarrow{\cong} \id_{ \btCtsmc(X )\rtimes^{\alg}G} \, , \quad 
 v \colon \phi \circ \psi \xrightarrow{\cong} \id_{\Cgtsmc(X \otimes G_{can,min})}\, .$$

Condition \ref{efrvgervfvfdvsdfvsdfvs} implies that  every morphism in $\btCtsmc(X)$ is controlled. 
Indeed, if $(C, \mu)$ and $ (C', \mu')$ are objects of $\btCtsmc(X )$, then by Definition \ref{rthoperthrtherthergtrgertgertretr}.\ref{qwriuheivnfvkjvvcd9} $$L \coloneqq \supp(C, \mu)\cup \supp  (C', \mu')$$   is a locally finite  subset of $X$. By Condition \ref{efrvgervfvfdvsdfvsdfvs} we can find an entourage $U$ in $\cC_{X}$ such that $
U\cap (L\times L)$ is maximal in the coarse structure on $L$ induced from $X$.  Then every morphism   $(C, \mu)\to  (C', \mu')$   in  $\btCtsmc(X ) $  is $U$-controlled (see Example \ref{qeroigqreffqewfqewfqewf}). 

Note that for a general $G$-bornological space $Y$ the morphisms 
between objects in $\btCtsmc(Y )$ are limits of sequences of controlled morphisms, but in general not controlled themselves. The fact that the morphisms are themselves controlled in our present situation 
will be used below in order to see that the morphism defined by \eqref{qfefqewfewfewfqdwedd} is controlled.

   
In order to construct $\phi$, by \cite[Lem.  5.7]{crosscat} 
it suffices to construct a covariant representation $(\sigma,\pi)$ of $\btCtsmc(X )$ on $\Cgtsmc(X \otimes G_{can,min})$.
We construct the morphism $\sigma:\btCtsmc(X )\to \Cgtsmc(X \otimes G_{can,min})$ as follows:
\begin{enumerate}
\item objects:  $\sigma$ sends the object $(C,\mu)$ in $ \btCtsmc(X )$ to
$$\sigma(C,\mu):=\big(\bigoplus_{g\in G} gC,\rho, \tilde \mu\big)$$ in
$\Cgtsmc(X \otimes G_{can,min})$. The  
 first entry uses the choice of the  orthogonal AV-sums adopted at the beginning of this proof.  The projection-valued measure   $\tilde\mu$ is determined by $$\tilde\mu(\{(x,g)\}) \coloneqq e^{C}_{g}  g\mu(g^{-1}(\{ x\}))e_{g}^{C,*}$$ for all $(x,g)$ in $X\times G$. Finally,    the unitary cocycle  $\rho \coloneqq (\rho_{h})_{h\in G}$  is defined by   
 \begin{equation}\label{qfefqewfewfewfqdwedd}
\rho_{h} \coloneqq \sum_{g\in G} he^{C}_{g}e_{hg}^{C,*}\colon \bigoplus_{g\in G} gC\to h\bigoplus_{g\in G}gC\ .
\end{equation} 
In order to show that $\rho$ and $\tilde \mu$ exist as morphisms in $\bM\bC$ we use \cite[Lem. 7.8]{cank}.
One checks that  $\rho$ satisfies the cocycle conditions, and that $\tilde\mu$ is equivariant (see Definitions \ref{etgiowergergwegwerg}.\ref{ewrgiuhefdcsdcsdc} and \ref{gwergoiegujowergwergwergwerg}.\ref{etgijwogerfrewrr}).
One furthermore checks that $\phi(C,\mu)$ is locally finite.
\item morphisms: $\sigma$ sends $A \colon (C,\mu)\to  (C',\mu')$ in $\btCtsmc(X ) $  to 
 \begin{equation}\label{qfefqewfewfefwefwefwwfqdwedd1}\sigma(A) \coloneqq \sum_{h\in G} e^{C'}_{h } hA e_{h}^{C,*}\, .\end{equation} where we 
 use  \cite[Lem. 7.8]{cank} 
 in order to see that this formula defines a multiplier morphism. 
  One checks in a straightforward manner that
$\sigma(A)$ is equivariant (see Definition \ref{etgiowergergwegwerg}.\ref{ewbiojhiobewrvcdfvsdv}).
In order to check that $\sigma(A)$ is controlled observe that 
  if  
 $A$ is $U$-controlled for $U$ in $\cC_{X}$, then $\sigma(A)$ is  $ G( U \times \diag(G))$-controlled, and that $ G( U \times  \diag(G) )$ is a coarse entourage of $X\otimes G_{can,min}$.
\end{enumerate}
We next construct  the family  
$\pi= (\pi(h))_{h\in G}$ of natural transformations
$\pi(h):\sigma\to h^{*}\sigma$ by setting  \begin{equation}\label{wergrefdfvsdfvsdfvsdfvsdfvsdfv}
\pi(h)_{(C,\mu)}:=\sum_{g\in G} e^{hC}_{g}e^{C,*}_{gh}:\bigoplus_{g\in G} gC \to  \bigoplus_{g}ghC
\end{equation}
for every $(C,\mu)$ in $\Ob( \btCtsmc(X ) )$.
 We again use   \cite[Lem. 7.8]{cank} 
 in order to see that this formula defines a unitary  multiplier morphism. As a morphism 
 $$\big(\bigoplus_{g\in G} gC,\rho, \tilde \mu\big)\to \big(\bigoplus_{g\in G} ghC,\rho, \tilde \mu\big)$$  
 in $\Cgtsmc(X \otimes G_{can,min}) $
 the morphism $\pi(h)_{(C,\mu)}$ is 
 $G(\diag(X)\times \{(e,h^{-1})\})$-controlled, and $G(\diag(X)\times \{(e,h^{-1})\})$ is a coarse entourage of $X\otimes G_{can,min}$.
   One checks that
 $(\pi(h)_{(C,\mu)})_{(C,\mu)\in \Ob( \btCtsmc(X ) )}$ is a natural transformation from $\sigma$ to $h^{*}\sigma$, and that the family
 $(\phi(h))_{h\in G}$ satisfies the cocycle condition $\ell^{*}\pi(h)\pi(\ell)=\pi(h\ell)$ for all $h,\ell $ in $G$.

We can now read off the formulas for $\phi$:
\begin{enumerate}
\item objects:  
$$\phi(C,\mu) \coloneqq \big(\bigoplus_{g\in G} gC,\rho, \tilde \mu\big)\, .  $$
  \item morphisms: We will use {the} pair notation for morphisms as in \cite[Defn.~5.1]{crosscat}.   \begin{equation}\label{qfefqewfewfewfqdwedd1}\phi(A,g) \coloneqq \sum_{h\in G} e^{C'}_{hg^{-1}} hA e_{h}^{C,*}\, .\end{equation} 
 \end{enumerate}
By an inspection one checks that $\phi$ is  is natural $X$.

Next we construct $\psi$.
\begin{enumerate}
\item objects: Let $(C,\rho,\tilde\mu)$ be an object of $ \Cgtsmc(X \otimes G_{can,min})$. Since $\bC$ is effectively additive by assumption 
and $(C,\rho,\tilde \mu)$ is locally finite we can 
 choose an image $(\tilde C,\tilde u)$  in $\bM\bC$ of the projection ${\tilde \mu}(X\times \{e\})$. Then we define $\mu\colon\cP_{X}\to \Proj(\tilde C)$ to be the restriction of ${\tilde u}^{*} {\tilde \mu} {\tilde u}$ to
$ X\times \{e\}$. We now set
\begin{equation}\label{eq_defn_psi_computation_alg}
\psi(C,\rho,\tilde \mu) \coloneqq (\tilde C,\mu)\, .
\end{equation}
One checks that $(\tilde C,\mu)$ is locally finite.
\item morphisms:   Let $A \colon (C,\rho,\tilde \mu)\to (C',\rho',{\tilde \mu}')$ be a morphism in $\Cgtsmc(X \otimes G_{can,min})$.
For every $g$ in $G$ we set
\begin{equation}\label{ebowjoervefvsfevsfdvvs}
A_{g}\coloneqq (g^{-1}{\tilde u}')^{*}\rho'_{g^{-1}} A{\tilde u} \colon \tilde C\to g^{-1} \tilde C'\, .
\end{equation}
If $A$ is $U$-controlled for $U$ in $\cC_{X\otimes G_{can,min}}$, then $A_{g}$ is $U\cap ((X\times X)\times {(e,e)})$-controlled for all $g$ in $G$. Furthermore, if $A_{g}\not=0$, then  $g^{-1} \in \pr_{G\times G}(U)[\{e\}]$. 
In view of the definition of the coarse structure on $X\otimes G_{can,min}$ we have $ \pr_{G\times G}(U)\in \cC_{G_{can,min}}$ and therefore  $  \pr_{G\times G}(U)[\{e\}] \in \cB_{G_{can,min}}$, i.e., this set is finite.
We now define
\begin{equation}\label{rvqrcsdcadscscsdcascc}
\psi(A) \coloneqq \sum_{g\in G}(A_{g},g) \, . 
\end{equation}
We just have seen that this  sum has only finitely many non-trivial terms and therefore defines a morphism from $(\tilde C,\mu)$ to $(\tilde C', \mu')$ in 
$\btCtsmc(X)\rtimes^{\alg}G$. One checks that this construction is compatible with the composition and the involution.
\end{enumerate}

We now construct the unitary   isomorphisms $u$ and $v$.
\begin{enumerate}
\item $u\colon \psi\circ \phi\xrightarrow{\cong}\id$:
Let $(C,\mu)$ be an object in $\btCtsmc(X)$. Then $\phi(C,\mu) = (\bigoplus_{g\in G}gC,{\rho,\tilde{\mu}})$.
In the construction of $\psi$ we have  chosen the image $(\widetilde{\bigoplus_{g\in G}gC}, {\tilde{u}})$ of $\tilde{\mu}(X\times \{e\})$
{such that}
$$u_{(C,\mu)} \coloneqq e^{C,*}_{e} {\tilde{u}} \colon \widetilde{\bigoplus_{g\in G}gC}\to C$$
is a unitary  multiplier isomorphism.  
The family $u \coloneqq (u_{(C,\mu)})_{(C,\mu)\in \Ob(\btCtsmc(X))}$ provides the desired unitary isomorphism. 
\item  $v\colon \phi\circ \psi\xrightarrow{\cong}\id$: Let $(C,\rho,\tilde \mu)$ be an object of $\Cgtsmc(X  \otimes G_{can,min})$. 
In the construction of $\psi$ we have fixed an image $(\tilde C,\tilde u)$ of ${\tilde \mu}(X\times \{e\})$ such that $\psi(C,\rho,\tilde \mu) = (\tilde C, \mu)$, see~\eqref{eq_defn_psi_computation_alg}.
Then $\phi(\tilde C,\mu) = (\bigoplus_{g\in G}g \tilde{C},\rho,\tilde{\tilde{\mu}})$.

Using the invariance  condition for $\tilde \mu$  stated in Definition \ref{gwergoiegujowergwergwergwerg}.\ref{etgijwogerfrewrr}, one  checks that   $(g\tilde C,\rho_{g}^{-1} g(u))$ represents an image of ${\tilde \mu}(X\times \{g\})$. Since $(C,\rho,\mu)$ is determined on points, by Lemma \ref{eriogjqwefqewfewfeqdewdq}
 the object $C$ is the orthogonal AV-sum of the  images of the family of projections $\tilde \mu(X\times \{g\})$. 
Therefore, setting
\begin{equation}\label{eq_unitary_iso_v_defn_ug}
u_{g} \coloneqq \rho_{g}^{-1} g(u)\, ,
\end{equation}   using \cite[Lem. 7.8]{cank} we get a unitary multiplier  isomorphism
\begin{equation}
\label{rgiowegrerewef_unitary_iso_v}
v_{(C,\rho,\mu)} \colon \bigoplus_{g\in G} g \tilde{C} \xrightarrow{\sum_{g \in G} u_g e^{\tilde{C},*}_{g}} C\,.
\end{equation}
The family
 $(v_{(C,\rho,\mu)})_{(C,\rho,\mu)\in \Ob( \Cgtsmc(X \otimes G_{can,min}))}$
 is the desired unitary isomorphism~$v$.\qedhere
\end{enumerate}
\end{proof}

We consider $  \bC$ in $\Fun(BG,\nCcat)$  and its full subcategory $\bC^{u}$ of unital objects  in $\Fun(BG,\Ccat)$. Under the condition that  $\bC^{u}$    is (finitely) additive, in 
  \cite[Constr.\ 19.6]{cank} and  \cite[Def.\ 19.9]{cank} we   introduced the functors
 \begin{equation}\label{qwefktbvrtfavkoavd}\bC^{u}[-]\colon G\Set \to {\Fun(BG,\Ccat)}\, , \quad \bC^{u}[-]\rtimes_{r}G\colon G\Set\to \Ccat\, .
\end{equation}
In Assertion \ref{egijogrgergewrgwegwrewg2} of the following proposition we use the reduced crossed product from \cite[Thm. 12.1]{cank}.
\begin{prop}\label{qergioqefweqwecqcasdc}
We assume that  $ \bC$ is effectively additive and $\bC^{u}$ is additive
\footnote{For the first two statements we could drop the essential additivity assumption. 
}.
For $X$ in $G\Set$ we have the following natural isomorphisms (or {unitary} equivalence, respectively):
\begin{enumerate}
\item\label{item_computation_2} $ \btCtsmc(X_{min,max})\cong \bC^{u} [X]$ in $\Fun(BG,\Ccat)$.
\item\label{egijogrgergewrgwegwrewg1}  
$ \bCGtsmc(X_{min,max})\cong \bC^{u} [ X]\rtimes G$ in $\Ccat$.
\item\label{egijogrgergewrgwegwrewg2} $ \bCgtsmc(X_{min,max}\otimes G_{can,min})\simeq \bC^{u} [ X]\rtimes_{r} G$ in $\Ccat$  (under the additional assumption that  
 $\bC$ admits orthogonal  AV-sums of cardinality  $|G|$)
\end{enumerate}
\end{prop}
\begin{proof}  
 {This isomorphism}  in Assertion \ref{item_computation_2} follows from a comparison of the construction \cite[Constr. 19.6]{cank} of $\bC^{u}[X]$ with the definition of
$ \btCtsmc(X_{min,max})$, as we explain in the following.
 Recall that an object of $ \btCtsmc(X_{min,max})$ is a pair $(C,\mu)$ of an object $C$ of $\bC$ 
and a function $\mu\colon \cP_{X}\to \Proj(C)$. 
Since $(C,\mu)$ is locally finite the projection  $\mu(\{x\})$ belongs to $\bC$ for every $x$ in $X$. Furthermore, since $X_{min,max}$ has the maximal bornology,  the set $ \supp(C,\mu)=\{x\in X\mid \mu(\{x\}) \not=0\}$ is finite so that $\sum_{x\in X} \mu(\{x\})=\id_{C}$ is a morphism in $\bC$. In particular, $C$ belongs to $\bC^{u}$. On objects the asserted isomorphism sends $(C,\mu)$ to  $(C,(\mu(\{x\}))_{x\in X})$ of $\bC^{u}[\tilde X]$.

Since $X_{min,max}$ has the minimal   coarse structure the condition that {a morphism} $A\colon (C,\mu)\to (C',\mu')$ is controlled corresponds to the condition $\mu'(\{x'\})A\mu(\{x\})=0$ for all $x,x'$ in $X$ with $x\not=x$ {as} required for a morphism in $ \bC^{u}[X]$.  
The asserted isomorphism identifies $A\colon (C,\mu)\to (C',\mu')$ as above with the same morphism  now considered as a morphism in $A:
(C,(\mu(\{x\}))_{x\in X})\to (C',(\mu'(\{x\}))_{x\in X})$.

Assertion \ref{egijogrgergewrgwegwrewg1} immediately follows from
Assertion \ref{item_computation_2} and the Definition \ref{toiwjopgregwrewg}.\ref{werijgowegwergwerfewrf}.

In order to show Assertion \ref{egijogrgergewrgwegwrewg2} note that we have a faithful morphism $  \btCtsmc(X_{min,max})\to \bC$
which sends $(C,\mu)$ to $C$ and is the obvious inclusion on the level of morphisms. Since $-\rtimes_{r}G$ preserves faithful (or equivalently, isometric) morphisms by \cite[Thm. 12.1]{cank} 
we can define the reduced norm on $ \btCtsmc(X_{min,max})\rtimes^{\alg}G$ using the induced covariant representation  
of $ \btCtsmc(X_{min,max}) $ on $ \bL^{2}(G,\bW^{\mathrm{nu}}\bC)$ (see \cite[Defs. 2.35 \& 12.2]{cank}).
By an inspection one observes that the formulas   \cite[(12.4)]{cank} and \cite[(12.5)]{cank} for this covariant representation coincide with the formulas  
  \eqref{qfefqewfewfefwefwefwwfqdwedd1}   and  \eqref{wergrefdfvsdfvsdfvsdfvsdfvsdfv}.
  We conclude that the norm induced  from  $\bCgtsmc(X \otimes G_{can,min})$  through the unitary equivalence $\phi$  from Proposition \ref{rgiowegrerewef} 
  is the norm of the reduced crossed product. 
  Therefore the unitary equivalence $\phi$ extends  by continuity to the unitary  equivalence as asserted in 
 Assertion \ref{egijogrgergewrgwegwrewg2}.
\end{proof}

\begin{rem}
In general the maximal and the reduced crossed products do not coincide. We have a canonical morphism from the maximal crossed product to the reduced one. Combining this with Proposition \ref{qergioqefweqwecqcasdc}.(\ref{egijogrgergewrgwegwrewg1} \& \ref{egijogrgergewrgwegwrewg2})
we get a functor
\begin{equation}\label{ewfiuqhiufqwefqwfqwefweff}
\bCGtsmc(X_{min,max})\to \bCgtsmc(X_{min,max}\otimes G_{can,min})
\end{equation}
which in general is not an equivalence. 
This  stays in contrast to
\cite[Cor.~5.4]{unik} which asserts that the analogue of this morphism in the case of controlled objects in left-exact $\infty$-categories
is always  an equivalence.

If $G$ is amenable, then as a consequence of \cite[Thm.\ 12.27]{cank}  the functor in \eqref{ewfiuqhiufqwefqwfqwefweff} is an isomorphism.
\hB
\end{rem}

\section{Transfers}\label{qeriughqifqwcqeqc}

In  \cite[Defn.~2.27]{coarsetrans} we introduced an extension
\begin{equation}\label{ergegwergewrgewrgefwefwerf}
\iota\colon G\BC\to G\BC_{\tr}
\end{equation}  of the category of $G$-bornological coarse spaces to an  $\infty$-category $G\BC_{\tr}$. 
It  has the same objects as $G\BC$.  Its $1$-morphisms are spans
$$\xymatrix{&W\ar@{->>}[dl]_{w}\ar[rd]^{f}&\\X&&Y}$$
where $f$ is a morphism in $G\BC$ which is in addition bornological, see Definition \ref{etgiohjroifgjqrofiqfewfqew}(\ref{rfoijoiffqefefeq}), and where  $w$ is a bounded covering (see  Definition \ref{wergijowergergrefwef} below).  
The higher morphisms are diagrams which exhibit compositions of spans. We refer to  \cite[Defn.~2.27]{coarsetrans} for more details. 
They will not be needed in the present paper.

We consider a functor $E\colon G\BC\to \bM$.

\begin{ddd}\label{poqjkpowckwdcqwcqwcwc}  $E$  admits transfers if there exists a functor $E_{\tr}\colon G\BC_{\tr}\to \bM$ and an equivalence
$E_{\tr}\circ \iota\simeq E$.
\end{ddd}

In the following it is useful to model $\infty$-categories by quasi-categories, and to distinguish the notation for  categories $\cC$ from the notation $\Nerve(\cC)$ for the corresponding quasi-categories.  Basic references for this model are   \cite{htt,HA}.

Recall that the category $\Ccat$ is the underlying $1$-category of a  strict $(2,1)$-category  $ \Ccat_{2,1}$ whose $2$-isomorphisms are unitary isomorphisms of functors. We have a  $2$-nerve functor $ \Nerve_{2}$ from $(2,1)$-categories to quasi-categories. This functor    sends a $(2,1)$-category to the coherent nerve of the fibrant simplicial category obtained by applying the usual nerve functor to the morphism groupoids.  
 We have a canonical embedding $$\ell_{2}\colon \Nerve(\Ccat)\to \Nerve_{2}(\Ccat_{2,1})\, .$$ 

Assume that $  \bC$  in $\Fun(BG, \nCcat)$ 
 admits all very small AV-sums and is therefore in particular effectively additive.
  We then have  the functor
 $$\bCgtsmc\colon G\BC\to  \Ccat$$ defined in Definition \ref{eihioqwefqwfewfqwefqewf}.\ref{eghqfijewofewfewfqfedqdwedwqed}.
We can consider the functor
$$\Nerve(G\BC)\xrightarrow{\Nerve(\bCgtsmc)} \Nerve(\Ccat)\xrightarrow{\ell_{2}} \Nerve_{2}(\Ccat_{2,1})\, .$$
We want to assert that this functor admits transfers. 
The more precise formulation of Definition \ref{poqjkpowckwdcqwcqwcwc} requires that there is
a functor
$\bCgtsmctr \colon G\BC_{\tr}\to \Nerve_{2}(\Ccat_{2,1})$ such that  \begin{equation}\label{t24bkjbfjkbnvkfvfdvaf}
\ell_{2}\circ \Nerve(\bCgtsmc)\simeq \bCgtsmctr\circ \iota\, .
\end{equation}

The following is the main theorem of the present section.
\begin{theorem}\label{ethgiuwieogfergwgrergweg} If $  \bC$  admits all very small AV-sums, then
 the functor $$\ell_{2}\circ \Nerve(\bCgtsmc) \colon \Nerve(G\BC)\to  \Nerve_{2}(\Ccat_{2,1})$$
admits transfers.
\end{theorem}

\begin{rem}
In the proof of Theorem \ref{ethgiuwieogfergwgrergweg}  we need the existence of AV-sums in $\bC$  for arbitrary very small sets.  This is stronger then  the requirements for the existence of the coarse homology theory $\Homol \bC\cX^{G}
_{c}$ which are effective additivity and countable AV-sums.


We can not expect that the functor $\bCgtsmc$ with values in the ordinary category $\Ccat$ admits transfers. The problematic point is that the contravariant extension of the functor
to bounded coverings involves choices of orthogonal AV-sums on the level of objects.
As a consequence,  functoriality is satisfied only up to canonical $2$-isomorphisms.
 \hB
 \end{rem}

Before we start with the proof of  Theorem \ref{ethgiuwieogfergwgrergweg}
 we recall some further notions from coarse geometry which will be used in the arguments.

Let $W$ be in $G\Coarse$. Then we consider the  element 
$$U_{\pi_{0}(W)} \coloneqq \bigcup_{U\in \cC_{W}} U$$
of $\cP_{W\times W}$.
 Using the conditions on $\cC_{W}$ listed in Definition \ref{trbertheheht} we see that $U_{\pi_{0}(W)}$ is a $G$-invariant equivalence relation on $W$.
\begin{ddd}\label{gheruierfefersfsdfsggs}
The equivalence classes in $W$ for $U_{\pi_{0}(W)}$ are the coarse components of~$W$.
\end{ddd}
We let $\pi_{0}(W)$ denote the set of coarse components.  Since $U_{\pi_{0}(W)}$ is $G$-invariant, the group $G$ acts on the set
 $\pi_{0}(W)$.

Let $W,X$ be objects of $G\BC$ and $w\colon W\to X$  be a morphism between the underlying $G$-coarse spaces.
\begin{ddd}\label{wergijowergergrefwef}
$w$ is a bounded covering if:
\begin{enumerate}
\item For every coarse component $V$ of $W$ the restriction $f_{|V}\colon V\to f(V)$ is an isomorphism of coarse spaces from $V$ to a coarse component of $X$.
\item \label{eqrgoihjqoiefewfewdqewdew}The coarse structure of $W$ is generated by the entourages $U_{\pi_{0}(W)}\cap (w\times w)^{-1}(U)$ for all $U$ in $\cC_{X}$.
\item $w$ is bornological.
\item \label{tgoiwjegoierfrefweferfwerfre} For every $B$ in $\cB_{W}$ there exists a finite bound (depending on $B$) on the cardinality of the fibres of $\pi_{0}(B)\to \pi_{0}(X)$.
\end{enumerate}
\end{ddd}
We will indicate bounded coverings using the double-headed  arrow symbol $\xymatrix{\ar@{->>}[r]&}$.

We consider objects $W,V,U,Z$ in $G\BC$ and a square  
  $$\xymatrix{&W\ar[dr]^{f}\ar@{->>}[dl]_{w}&\\V\ar[dr]_{g}&&U\ar@{->>}[dl]^{u}\\&Z&}$$
  in $G\Coarse$.
  \begin{ddd}[{\cite[Def.\ 2.19]{coarsetrans}}]\label{riogoqefewfqfeweqd}
 The square is admissible if it has the following properties: \begin{enumerate}
 \item The square is cartesian in $G\Coarse$.
 \item $f$ and $g$ are proper and bornological.
 \item $u$ and $w$ are bounded coverings.
 \end{enumerate}
 \end{ddd}

\begin{proof}[Proof of Theorem \ref{ethgiuwieogfergwgrergweg}]
We must construct a functor $$\bCgtsmctr\colon G\BC_{\tr}\to  \Nerve_{2}(\Ccat_{2,1})$$ satisfying \eqref{t24bkjbfjkbnvkfvfdvaf}.
We can use the method explained
in  \cite[Sec.~2.1]{coarsetrans}. Following this reference, in order to construct  $\bCgtsmctr $ we must provide the following data:
\begin{enumerate}
\item We start with the functor $\bCgtsmc\colon G\BC\to \Ccat$.
\item For every bounded covering $w\colon W\to Z$ we must provide a $1$-morphism $$w^{*}:\bCgtsmc(Z)\to \bCgtsmc(W)\, .$$ 
\item For every two composable bounded coverings $w\colon W\to Z$ and $v\colon V\to W$ we must provide a $2$-isomorphism
$$a_{v,w}\colon (w\circ v)^{*}\Rightarrow  v^{*}\circ w^{*}\, .$$
\item For every admissible square
$$
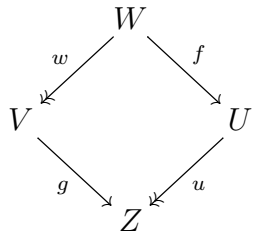
\xymatrix{&W\ar[dr]^{f}\ar@{->>}[dl]_{w}&\\
V\ar[dr]_{g}&&U\ar@{->>}[dl]^{u}\\&Z&}$$
we must provide a $2$-isomorphism
$$b_{g,u}\colon f_{*}\circ w^{*}\Rightarrow u^{*}\circ g_{*}\, .$$
\end{enumerate}
This data must satisfy the following conditions:
\begin{enumerate}
\item If the bounded covering $w\colon W\to Z$ is an isomorphism, then $w^{*}=w^{-1}_{*}$.
\item If the composable bounded coverings  $w\colon W\to Z$ and $v\colon V\to W$  are both isomorphisms, then
$a_{v,w}=\id$. This equality makes sense, because in this case $v^{*}w^{*}=(wv)^{*}$ by the strict functoriality of $\bCgtsmc$.
\item If $w\colon W\to Z$ and $v\colon V\to W$   and $u\colon U\to V$ are three composable bounded coverings, then 
the square of  $2$-isomorphisms
$$\xymatrix{(w\circ v\circ u)^{*}\ar@{=>}[rr]^{a_{vu,w}}\ar@{=>}[d]^{a_{u,wv}}&&(v\circ u)^{*}\circ w^{*}\ar@{=>}[d]^{a_{u,v}\circ w^{*}}\\
u^{*}\circ (w\circ v)^{*}\ar@{=>}[rr]^{u^{*}\circ a_{v,w}}&&u^{*}\circ v^{*}\circ w^{*}}$$
commutes.
\item In the case of an admissible square with morphisms $w,f,g,u$, if $u$ (and therefore $w$) is an isomorphism,
then we have  $b_{g,u}=\id$. This equality again makes sense since $$f_{*}w^{*}=f_{*}w^{-1}_{*}=u^{-1}_{*}g_{*}=u^{*}g_{*}$$
 by the strict functoriality of $\bCgtsmc$.
\item  In the case of an admissible square with morphisms $w,f,g,u$, if $f$ and $g$ are identities and therefore $w=u$,  then we have  $b_{g,u}=\id$.
\item For every diagram
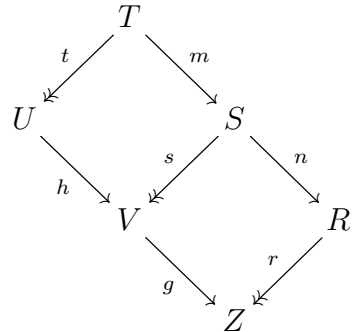
\[
\xymatrix{&&&T\ar@{->>}[dl]_{t}\ar[dr]^{m}&&&\\
&&U \ar[dr]_{h}&&S\ar[dr]^{n}\ar@{->>}[dl]_{s}&&\\
& &&V \ar[dr]_{g}&&R\ar@{->>}[dl]_{r} &\\&&&&Z&&}
\] 
consisting of two admissible squares we have the relation 
\begin{equation}\label{g34glkjkgl999234g34}
b_{gh,r}=(b_{g,r}\circ h_*)(n_*\circ b_{h,s})\, .
\end{equation}
\item For every diagram  
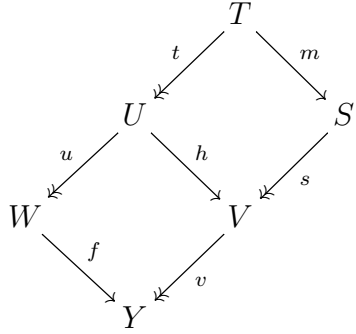
\[\xymatrix{&&&T\ar@{->>}[dl]_{t}\ar[dr]^{m}&&&\\&&U\ar@{->>}[dl]_{u}\ar[dr]^{h}&&S \ar@{->>}[dl]^{s}&&\\&W \ar[dr]^{f}&&V\ar@{->>}[dl]^{v} && &\\&&Y&&&&}\] 
consisting of  two admissible squares we have the relation \begin{equation}\label{fce32999r32r23r5}
(a_{s,v}\circ f_*)b_{f,vs}= (s^*\circ b_{f,v})(b_{h,s}\circ u^*)(m_*\circ a_{t,u})\, .
\end{equation}
\end{enumerate}

%
 
 Let $X$ be in $G\BC$, and let   $(C,\rho,\mu)$    be an object of $ \bCgtsmc(X)$.  For a subset $Y$ of $X$
 we will use the notation $(C(Y),u^{C}_{Y})$ in order to denote a representative of the image in $\bM\bC$ of the projection $\mu(Y)$.  
Our language will be adapted to $\bC$ so that in the following morphisms in $\bM\bC$ will be called multiplier morphisms.
 
 We let  $( \bigoplus_{Y\in \pi_{0}(X)} C( Y), (e_{Y})_{Y\in \pi_{0}(X)})$ denote  a representative of the orthogonal AV-sum in $\bC$ of the family $(C(Y))_{Y\in \pi_{0}(X)}$.
 By   Lemma \ref{eriogjqwefqewfewfeqdewdq} applied to the partition $(Y)_{Y\in \pi_{0}(X)}$ of $X$
 we have   a unitary multiplier   isomorphism 
$$ \sum_{Y\in \pi_{0}(X)} u^{C}_{Y}e_{ Y}^{*}\colon  \bigoplus_{Y\in \pi_{0}(X)} C( Y)\xrightarrow{\cong} C\, .$$
 
Let $(C',\rho',\mu')$ be a second object, $( \bigoplus_{Y\in \pi_{0}(X)} C'( Y), (e'_{Y})_{Y\in \pi_{0}(X)})$ be the corresponding choice of an orthogonal AV-sum, and let $A\colon
(C,\rho,\mu)\to (C',\rho',\mu')$ be a morphism in $\bCgtsmc(X)$.
For $Y,Y'$ in $\pi_{0}(X)$ we define the multiplier morphism
$$A_{Y',Y} \coloneqq u_{Y'}^{C',*}Au^{C}_{Y} \colon C(Y)\to C'(Y') \, .$$
Since $A$ is $U_{\pi_{0}}(X)$-controlled we actually have $A_{Y',Y}=0$ if $Y\not=Y'$. 
In order to simplify the notation we set
$$A_{Y} \coloneqq A_{Y,Y}\, .$$

The morphism $A$ is uniquely determined by the family $(A_{Y})_{Y\in \pi_{0}(X)}$. 
In the following we discuss how one can reconstruct the morphism $A$  from a family  $(A_{Y})_{Y\in \pi_{0}(X)}$. 
Assume that a family of multiplier morphisms $(A_{Y}\colon C( Y )\to C'( Y))_{Y\in \pi_{0}(X)}$ is given. 
We say that $(A_{Y})_{Y\in \pi_{0}(X)}$ is bounded if 
 $\sup_{Y\in \pi_{0}(X)}\|A_{Y}\|<\infty$. 
 If $(A_{Y})_{Y\in \pi_{0}(X)}$ is bounded, then  using \cite[Lem. 7.8]{cank} we can define a multiplier morphism $A\colon C\to C'$  uniquely such that
$$\xymatrix{
\bigoplus_{Y\in \pi_{0}(X)} C( Y )\ar[rr]^{\oplus_{Y\in \pi_{0}(X)}A_{Y}}\ar[d]^{\cong}_{\sum_{Y\in \pi_{0}(X)} u^{C}_{Y}e^{*}_{ Y }}&& \bigoplus_{Y\in \pi_{0}(X)} C'( Y )\ar[d]_{\cong}^{\sum_{Y\in \pi_{0}(X)} u^{C'}_{Y}e^{\prime,*}_{ Y }}\\
C\ar[rr]^{A}&&C'
}$$
commutes.

By Definition \ref{gwergoiegujowergwergwergwerg}.\ref{etgijwogerfrewrr} we have the relation $\rho_{g}\circ \mu(Y)=g\mu(g^{-1}Y)\circ \rho_{g}$. We therefore get a canonical unitary multiplier \begin{equation}\label{refrewhiu3hfiuherifwerfrfw}
\rho_{g,Y} \coloneqq  (gu^{C}_{ g^{-1}Y})^{*} \circ \rho_{g}\circ u^{C}_{ Y} \colon C(Y)\to g(C(g^{-1}Y))\, .
\end{equation} 
We say that the family $(A_{Y})_{Y\in \pi_{0}(X)}$   is equivariant   if it  satisfies 
$$gA_{g^{-1}Y}\circ \rho_{g,Y} = \rho'_{g,Y}\circ A_{Y}$$ for all $g$ in $G$. If $(A_{Y})_{Y\in \pi_{0}(X)}$   is bounded  and equivariant,  then it   induces a morphism of $G$-objects $A\colon (C,\rho)\to (C',\rho')$ in ${(\bM\bC)}^{G}$, see Definition \ref{etgiowergergwegwerg}.   

Finally one must ensure that $A$ is controlled. We consider the projection-valued function
\begin{equation}\label{fewveoivhroivjoieverwcecwec}
\mu_{Y} \coloneqq u^{C,*}_{Y}\mu_{|\cP_{Y}} u^{C}_{Y} \colon \cP_{Y}\to \Proj(C(Y))\, .
\end{equation}
 Then  
 $(C(Y), \mu_{Y})$ is an object of $ \bCtsmc(Y)$.  
 Assume that $(A_{Y})_{Y\in \pi_{0}(X)}$ is an equivariant and bounded family.
 Let $U$ be  in $\cP_{X\times X}$. We say that  $(A_{Y})_{Y\in \pi_{0}(X)}$ is $U$-controlled if $A_{Y}$ is $U\cap (Y\times Y)$-controlled for all $Y$ in $\pi_{0}(X)$. 
 If  $(A_{Y})_{Y\in \pi_{0}(X)}$ is $U$-controlled, then the induced morphism  $A \colon (C,\rho,\mu)\to (C',\rho',\mu')$ is a $U$-controlled morphism in $\Cgtsmc(X)$.
 
 We again consider an equivariant and bounded family  $(A_{Y})_{Y\in \pi_{0}(X)}$. 
  We now assume that there exists a sequence $((A_{n,Y})_{Y\in \pi_{0}(X)})_{n\in \nat}$
of bounded equivariant  families  and a family $(U_{n})_{n\in \nat}$ in $\cC_{X}$ such that $(A_{n,Y})_{Y\in \pi_{0}(X)}$ is 
$U_{n}$-controlled for every 
 $n$ in $\nat$,    and 
$$\lim_{n\to \infty} \sup_{Y\in \pi_{0}(X)}\|A_{n,Y}-A_{Y}\|=0\, . $$ Let $A_{n} \colon (C,\rho,\mu)\to (C',\rho',\mu')$ be the  $U_{n}$-controlled morphism in $\Cgtsmc(X)$ defined by the family  $(A_{n,Y})_{Y\in \pi_{0}(X)}$.
Then {we have} $\lim_{n\to \infty} A_{n}=A$ in $\Hom_{\bC}(C,C')$, and consequently $A\colon (C,\rho,\mu)\to (C',\rho',\mu')$ is a morphism in $\bCgtsmc(X)$.

Note that if $A \colon (C,\rho,\mu)\to (C',\rho',\mu')$ is a morphism in $\bCgtsmc(X)$, then the associated family
$(A_{Y})_{Y\in \pi_{0}(X)}$ has all the properties needed for the reconstruction of $A$.

%
%
%


We now start with the construction of the data for $\bCgtsmctr$.

\begin{enumerate}
   \item The functor $\bCgtsmc \colon G\BC\to \Ccat$ is already constructed, see Definition \ref{eihioqwefqwfewfqwefqewf}.\ref{eghqfijewofewfewfqfedqdwedwqed}.
   \item Let $w \colon W\to Z$ be a bounded covering. We construct the functor
  $$w^{*} \colon \bCgtsmc(Z)\to\bCgtsmc(W)\, .$$
\begin{enumerate}
\item \label{rgiowegwergerfweffw} objects: Let $(C,\rho,\mu)$ be an object in {$\bCgtsmc(Z)$.} We choose an orthonormal AV-sum $(\bigoplus_{Y\in \pi_{0}(W)} C(w(Y)), (e^{w^{*}}_{Y})_{Y\in \pi_{0}(W)})$ in $\bC$.
In particular, if  $w$ is an isomorphism, then we choose the representative $(C, (u^{C}_{w(Y)})_{Y\in \pi_{0}(W)})$.
We then 
set
  $$w^{*}(C,\rho,\mu) \coloneqq (w^{*}C,w^{*}\rho,w^{*}\mu)\, ,$$ where the right-hand side has the following description.
\begin{enumerate}
\item We set 
$$w^{*}C \coloneqq \bigoplus_{Y\in \pi_{0}(W)} C(w(Y))\, .$$
 \item For a subset $B$ of $W$ and $Y$ in $\pi_{0}(W)$ we define the projection
$$(w^{*}\mu)(B)_{Y} \coloneqq \mu_{w(Y)}(w(B\cap Y)) $$ on $C(w(Y))$, see \eqref{fewveoivhroivjoieverwcecwec}.
We then define $(w^{*}\mu)(B)$ as the projection on $w^{*}C$  
given by
$$(w^{*}\mu)(B) \coloneqq \oplus_{Y\in \pi_{0}(W)}(w^{*}\mu)(B)_{Y}\, .$$
By Definition \ref{wergijowergergrefwef}.\ref{tgoiwjegoierfrefweferfwerfre}, if $B$ is in $\cB_{W}$, then the sum has finitely many non-zero summands.
One now checks that $w^{*}\mu$ is   a finitely additive, projection-valued measure which is full.  Furthermore, 
$(w^{*}C,w^{*}\mu)$ is locally finite.
 \item  For  $g$ in $G$, using \eqref{refrewhiu3hfiuherifwerfrfw},  
 we define
\[
\mathclap{
(w^{*}\rho)_{g}\colon \bigoplus_{Y\in \pi_{0}(W)} C( w(Y)) 
 \xrightarrow{\oplus_{Y}\rho_{g,Y}} \bigoplus_{Y\in \pi_{0}(W)} g(C( w(g^{-1}Y)))\cong g\bigoplus_{Y\in \pi_{0}(W)} C( w(Y)) \,.
}\]
  The last unitary  multiplier isomorphism is fixed by the commutativity of the diagram
\[\mathclap{\xymatrix{
&g(C( w(g^{-1}Y)))\ar[dl]_-{\tilde e}\ar[dr]^-{ge^{w^{*}}_{g^{-1}Y}}&\\\bigoplus_{Y\in \pi_{0}(W)} g(C(w(g^{-1}Y)))\ar[rr]_{\cong}&&g\bigoplus_{Y\in \pi_{0}(W)} C(w(Y))
}}\,.\]
  where $\tilde e$ is the canonical embedding of the summand with index $Y$.
\end{enumerate}
One now checks  in a  straightforward manner that 
\begin{equation*}
w^{*}(C,\phi,\rho)=(w^{*}C,w^{*}\phi,w^{*}\rho)\, .
\end{equation*}
 is a well-defined  object in $ \bCgtsmc(W)$.
\item morphisms: Let $A \colon (C,\rho,\mu)\to (C',\rho',\mu')$ be a morphism in $\bCgtsmc(Z)$.
  Then we let $w^{*}(A) \colon w^{*}C\to w^{*}C'$ be the morphism which is uniquely determined by the family 
  $(A_{w(Y)})_{Y\in \pi_{0}(W)}$. 
  We now explain why $w^{*}A$ is well-defined. First of all note that $(A_{B})_{B\in \pi_{0}(Z)}$ is bounded and equivariant. 
We conclude that $(A_{w(Y)})_{Y\in \pi_{0}(W)}$ is also bounded. Since  $w$ is equivariant we furthermore see that $   (A_{w(Y)})_{Y\in \pi_{0}(W)}$ is also equivariant. 
There exists a family $(A_{n})_{n\in \nat}$ of morphisms $A_{n} \colon (C,\rho,\mu)\to (C',\rho',\mu')$  in $\Cgtsmc(Z)$ and a family $(U_{n})_{\in \nat}$ in $\cC_{Z}$ such that
$A_{n}$ is $U_{n}$-controlled for every $n$ in $\nat$, and $\lim_{n\to \infty}A_{n}=A$ in $\Hom_{\bC}(C,C')$.  
Then $(A_{n,w(Y)})_{Y\in \pi_{0}(W)}$ is  $U_{\pi_{0}(W)}\cap (w\times w)^{-1}(U_{n})$-controlled  for every  $n$ in $\nat$. 
By  Definition \ref{wergijowergergrefwef}.\ref{eqrgoihjqoiefewfewdqewdew}  we have $ U_{\pi_{0}(W)}\cap (w\times w)^{-1}(U_{n})\in \cC_{W}$, and furthermore 
 $\lim_{n\to \infty} \sup_{Y\in \pi_{0}(W)} \|A_{n,w(Y)}-A_{w(Y)}\|=0$.
By the discussion above we see that $A$ is controlled. 
We have the following formula:
\begin{equation}\label{fdsvvdfvdsfvsfv}
w^{*}A=\sum_{B\in \pi_{0}(W)} e^{\prime, w^{*} }_{B} u^{C',*}_{w(B)}A u_{w(B)}^{C}e^{w^{*},*}_{B}\, .
\end{equation}
\end{enumerate}
One checks in a straightforward manner that $w^{*}$ is compatible with compositions and the involution.
%
%
%
\item We consider two composable bounded coverings $\xymatrix{V\ar@{->>}[r]^{v}&W\ar@{->>}[r]^{w}&Z}$. We construct the $2$-isomorphism
$$a_{v,w} \colon (w v)^{*}\Rightarrow  v^{*}w^{*}$$ as follows.
Let $(C,\rho,\mu)$ be an object of  $  \bCgtsmc (Z)$.
Then we define
\[
a_{v,w}(C,  \rho,\mu) \colon ((wv)^{*}C,(wv)^{*}\phi,(wv)^{*}\rho)\to (v^{*}w^{*}C,v^{*}w^{*}\phi,v^{*}w^{*}\rho)
\]
by
\begin{equation}\label{sfdvwrgfevfvsdfvsfdvsfdv1}
a_{v,w}(C,  \rho,\mu) \coloneqq \sum_{B\in\pi_{0}(V)} e_{B}^{v^{*}} u_{v(B)}^{w^{*}C,*} e_{v(B)}^{w^{*}} e_{B}^{(wv)^{*},*}\,.
\end{equation}
One checks by a calculation  that $a_{v,w}(C,\rho,\mu)$ is a unitary multiplier isomorphism. %
One furthermore checks that
$$a_{v,w} \coloneqq (a_{v,w}(C,\rho,\mu))_{(C,\rho,\mu)\in \Ob(\bCgtsmc(Z))}$$
is a natural  unitary  multiplier isomorphism. %
\item\label{qwefhjqriofqwecwefcqewfewcfqwfe}
We now consider an admissible square
$$\xymatrix{&W\ar[dr]^{f}\ar@{->>}[dl]_{w}&\\V\ar[dr]_{g}&&U\ar@{->>}[dl]^{u}\\&Z&}$$
Then must define the  transformation $$b_{g,u} \colon f_{*}\circ w^{*}\Rightarrow u^{*}\circ g_{*}\, .$$ 
Let $(C,\rho,\mu)$ be an object of  $\bCgtsmc(V)$. 
Then {we} define  
\[
b_{g,u}(C,\rho,\mu) \colon f_{*}w^{*}(C,\rho,\mu)\to u^{*}g_{*}(C,\rho,\mu)
\]
by
\begin{equation}\label{tgwjfirjfowfwerfewrfwefewrfrf}
b_{g,u}(C,\rho,\mu) \coloneqq \sum_{B\in \pi_{0}(U)} e^{u^{*}}_{B}u^{g_{*}C,*}_{u(B)}\sum_{A\in \pi_{0}(f^{-1}(B))} u^{C}_{w(A)}e_{A}^{w^{*},*}\,.
\end{equation}
One checks  by a calculation that 
 $b_{g,u}(C,\rho,\mu)$ is a unitary multiplier  isomorphism.
\end{enumerate}
We must verify that the 
  functors and $2$-isomorphisms constructed above satisfy the  seven conditions listed above. All these conditions are verified by straightforward (but lengthy) calculations. These calculations start with inserting the  explicit  definitions of the transformations given above and then 
proceed by  employing  the orthogonality relations for the families  of canonical isometries. 
\end{proof}

Let $\Homol\colon \Ccat\to \bM$ be a functor to an $\infty$-category $\bM$. 
\begin{theorem}\label{qwrgiherogwefewfqewfqwefqwfewf} If $\bC$ admits all very small AV-sums and    $\Homol$ sends unitary equivalences to equivalences, then the composition
\[
\Homol\circ \bCgtsmc \colon G\BC\to \bM
\]
admits transfers.
\end{theorem}
\begin{proof} 
 The proof is analogous to the proof of  \cite[Prop.~2.63]{Bunke:2018aa} combined with the discussion around the diagram in \cite[(3.8)]{coarsetrans}.
 
 We consider $\Homol$ as a functor $\Nerve(\Homol) \colon \Nerve(\Ccat)\to \bM$.

As explained in \cite{DellAmbrogio:2010aa} or \cite[Thm.~1.4]{startcats} the   category $\Ccat$   has a simplicial model category structure whose weak equivalences  $W$ are the unitary equivalences.
Applying the coherent nerve functor we obtain an $\infty$-category $\Nerve^{\mathrm{coh}}(\Ccat)$. 
  Then we have a canonical inclusion functor $\Nerve(\Ccat)\to \Nerve^{\mathrm{coh}}(\Ccat)$.
This inclusion is an explicit model of the 
Dwyer--Kan localization
$\ell\colon \Nerve(\Ccat)\to \Nerve(\Ccat)[W^{-1}]$. 

We now construct the following  commutative diagram
\begin{equation}\label{sdvasljknkjwqvafsvasfdvasdv}
\xymatrix{
\Nerve(\Ccat)\ar@/^1cm/[rr]^{\Nerve(\Homol)}\ar[r]^-{\ell}\ar[dr]^{(1)}\ar[ddr]_{\ell_{2}}&\Nerve(\Ccat)[W^{-1}]\ar[d]_-{(2)}^-{\simeq}\ar@{..>}[r]^-{\tilde{\Homol}} &\bM\,.\\
&\Nerve^{\mathrm{coh}}(\Ccat) &
\\&\Nerve_{2}(\Ccat_{2,1})\ar[u]^-{(3)}_-{\cong}\ar@{-->}[ruu]_-{\Homol_{2,1}}&
}
\end{equation}

The morphism $(1)$ sends unitary equivalences to equivalences. Consequently we obtain the morphism $(2)$ and the filler of the corresponding triangle from the universal property of the  Dwyer--Kan localization $\ell$.
As explained above the morphism (2) is an equivalence. The morphism (3) is actually an isomorphism of simplicial sets. 

Since $\Homol$ sends unitary equivalences to equivalences
we obtain the factorization $\tilde \Homol$ and the upper triangle from the universal property of the Dwyer--Kan localization $\ell$.
We now define $\Homol_{2,1} \coloneqq \tilde \Homol\circ (2)^{-1}\circ (3)$.  The diagram \eqref{sdvasljknkjwqvafsvasfdvasdv} provides   an equivalence of functors 
$$\Nerve(\Homol) \simeq \Homol_{2,1}\circ \ell_{2} \colon \Nerve(\Cat)\to \bM\, .$$
This equivalence provides the right triangle in the following  diagram:
\begin{equation}\label{fwevpoiejvopevececqwec}
\xymatrix{
\Nerve(G\BC)\ar[d]^{\iota}\ar[rr]^-{\Nerve( \bCgtsmc)}&&\Nerve(\Ccat)\ar[d]^{\ell_{2}}\ar[r]^-{\Nerve(\Homol)}& \bM\,.\\
G\BC_{\tr}\ar[rr]^-{\bCgtsmctr}&&\Nerve_{2}(\Ccat_{2,1}) \ar[ur]_-{\Homol_{1,2}}&
}
\end{equation} 
The existence of the left square is exactly the statement of  Theorem \ref{ethgiuwieogfergwgrergweg}. 
The outer part of the diagram \eqref{fwevpoiejvopevececqwec} exhibits the required extension of $\Homol\circ \bCgtsmc$.
  \end{proof}

\section{Induction from and restriction to subgroups}\label{eoigjortwgwegergwegreg}

Let $G$ be a group and consider a subgroup  $H$. 
We consider a coefficient category  $ \bC$  in $\Fun(BG,\nCcat)$
and assume that it is effectively additive.
  Then we can form the functor  $$\bCgtsmc \colon G\BC\to \Ccat$$ (Definition \ref{eihioqwefqwfewfqwefqewf}.\ref{eghqfijewofewfewfqfedqdwedwqed}). 
Furthermore, we can restrict the $G$-action on $  \bC$  to an $H$-action (this will not be written explicitly) and define the functor
$$\bChtsmc \colon H\BC\to \Ccat\, .$$

We have an induction functor 
\begin{equation}\label{fqewpofkopfweqfqwefwqf}
\Ind_H^G\colon H\BC \to G\BC\,,
\end{equation}
see \cite[Sec.\ 6]{equicoarse} (it will be  recalled in the proof of Proposition \ref{woigjwergregwregwerg}).

  For the following statement we implicitly compose the functors $\bCgtsmc$ and
 $ \bChtsmc$ with the functor $\Ccat\to \Ccat_{2,1}$, where 
  $\Ccat_{2,1}$ is the $(2,1)$-category of $C^{*}$-categories, functors and unitary isomorphisms.

\begin{prop}\label{woigjwergregwregwerg}
If $\bC$ admits  AV-sums of size $|H\backslash G|$, then we have a natural equivalence
$$\phi \colon \bCgtsmc\circ \Ind_{H}^{G}\xrightarrow{\simeq} \bChtsmc$$
of functors from $H\BC$ to the $(2,1)$-category $\Ccat_{2,1}$.
\end{prop}

\begin{rem}
The reason for considering the target of the two functors in the $(2,1)$-category is that the transformation $\phi$ which we will construct is only natural for $X$ in $H\BC$ up to unique isomorphisms which are in general not equalities. This is because in  the construction we must choose objects representing images of projections, and these choices are only determined up to unique unitary isomorphisms, see also Remark \ref{gwoptjgopgergwerg}. \hB
\end{rem}

\begin{proof}[Proof of Proposition \ref{woigjwergregwregwerg}]

%

We start with recalling the construction of the induction functor \eqref{fqewpofkopfweqfqwefwqf} from \cite[Sec.\ 6]{equicoarse}.
Let $X$ be in $H\BC$. We define $\Ind_H^G(X)$ in $G\BC$ as follows.
\begin{itemize}
\item The underlying $G$-set of $\Ind_H^G(X)$ is $G \times_H X$ with the left-action of $G$ on the left factor.  Here $G \times_H X$ is the quotient of the set  $G \times X$ by $H$ with respect to the action given by $h(g,x) = (g h^{-1},hx)$.
\item The bornology on $G \times_H X$ is generated by the images under the natural projection $G\times X\to G\times_{H}X$ of the subsets  $\{g\} \times B$ for all bounded subsets $B$ of $X$.
\item The coarse structure is generated by the images of $\diag_G \times U$ under this projection for all entourages $U$ of $X$.
\end{itemize}

The functor \eqref{fqewpofkopfweqfqwefwqf} 
is now defined as follows:
\begin{itemize}
\item  {objects:} The functor $\Ind_{H}^{G}$ sends $X$ in $H\BC$ to   $\Ind_H^G(X)$ in $G\BC$.
\item {morphisms:} The functor $\Ind_{H}^{G}$ sends   a morphism $f\colon X \to Y$ in $H\BC$  to the morphism in $G\BC$ given by the map  $G\times_{H} f \colon G\times_{H} X\to G\times_{H} Y$ of underlying sets induced from $\id_{G}\times f \colon G\times X\to G\times Y$.
\end{itemize}


For $X$ in $H\BC$ we now construct a  morphism
$$\phi_{X}\colon \bCgtsmc(\Ind_H^G(X)) \to \bChtsmc(X)\, .$$
\begin{enumerate}
\item objects: Let $(C,\rho,\mu)$ be an object of $\bCgtsmc(\Ind_H^G(X))$.
We have a $H$-equivariant inclusion $i \colon X\to G\times_{H}X$ given by
$x\mapsto [e,x]$. Here $[g,x]$ denotes the image in $G\times_{H}X$ of $(g,x)$ under the natural projection $G\times X\to G\times_{H}X$.  Since $\bC$ is effectively additive  and $(C,\rho,\mu)$ is determined on points  we can choose an image $(\tilde C,u)$  in $\bM\bC$ of the projection $\mu(i(X))$ on $C$.
Then we define $
\tilde{\rho}= (\tilde{\rho}_h)_{h \in H} $  by $   \tilde{\rho}_h\coloneqq (hu^* ) \rho_h   u \colon \tilde{C} \to h\tilde{C}
$ and  
$\tilde{\mu} \colon \cP(X) \to \Proj(\tilde{C})$ by $  \tilde{\mu}(Y) \coloneqq u^*   \mu(Y)   u$. The functor $\phi_{X}$ sends an object  $(C,\rho,\mu)$   to the object $(\tilde C,\tilde \rho,\tilde \mu)$ in $ \bChtsmc(X)$. One checks the conditions listed in Definitions \ref{gwergoiegujowergwergwergwerg} and \ref{rthoperthrtherthergtrgertgertretr}  by straightforward   calculations and thus verifies that the object $ (\tilde C,\tilde \rho,\tilde \mu)$ is well-defined. 
\item morphisms: If $A \colon (C,\rho,\mu)\to  (C',\rho',\mu') $ is a morphism in  $\bCgtsmc(\Ind_H^G(X))$, then 
the functor $\phi_{X}$ sends $A$ to $u^{\prime,*}Au \colon (\tilde C,\tilde \rho,\tilde \mu)\to  (\tilde C',\tilde \rho',\tilde \mu') $.
One  checks  that this is a morphism in $ \bChtsmc(X) $ and that this construction is compatible with the involution and composition. Here it is important that $i(X)$ is a union of coarse components of $G\times_{H}X$ and therefore $A$ does not ``propagate'' between $i(X)$ and its complement.
\end{enumerate}

We now argue that $\phi_{X}$ is a unitary equivalence. 
We first see using the $G$-equivariance of morphisms in $ \bCgtsmc(\Ind_{H}^{G}(X)) $ 
and the fact that the $G$-translates of $i(X)$ form a coarsely disjoint partition of $\Ind_{H}^{G}(X)$
that $\phi_{X}$ is  bijective on morphism spaces.  This immediately implies that $\phi_{X}$ is fully faithful.
 It remains to  argue that $\phi_{X}$ is essentially surjective. Let 
 $(\tilde C,\tilde \rho,\tilde \mu)$  in  $\bChtsmc(X)$ be given. We choose a section $s \colon H\backslash G\to G$ of the projection $G\to H\backslash G$ such that $s([e])=e$.
 Then we choose an orthogonal AV-sum $(C,(e_{[k]})_{[k]\in H\backslash G})$ in $\bC$ of the family of objects $(s(k)\tilde C)_{[k]\in H\backslash G}$ in $\bC$. 
 It is at this point where we use the additional assumption on $\bC$.
 We then observe using \cite[Lem. 7.8]{cank} that there exist a unique extension of 
   the $H$-cocycle $\tilde \rho$ on $\tilde C$ to a $G$-cocycle $\rho$ on $C$ 
   such that $he^{*}_{[e]} \rho_{h}e_{[e]}=\tilde \rho_{h}$ for all $h$ in $H$.  Furthermore,
   there exists a unique extension of
    $\tilde \mu$ to a $G$-equivariant  $\Proj(C)$-valued measure on $G\times_{H}X$ such that for any subset $Y$ of $X$ we have
    $\tilde \mu(Y)=e_{[e]}^{*} \mu(i(Y)) e_{[e]}$.
    Applying $\phi_{X}$ to the object $(C,\rho,\mu)$ of $\bCgtsmc(\Ind_{H}^{G}(X)) $ we get an object which is unitarily isomorphic to the original object $(\tilde C,\tilde \rho,\tilde \mu)$.

Let $f \colon X\to X'$ be a morphism in $H\BC$, and let  $(C,\rho,\mu)$ be  in  
$\bCgtsmc(\Ind_H^G(X))$. In the construction of $\phi_{X}$ we have chosen an image $(\tilde C,u)$ of $\mu(i(X))$ in $\bM\bC$. Similarly, in the construction of
$\phi_{X'}$ applied to  $\Ind_{H}^{G}(f)_{*}(C,\rho,\mu)$ we have chosen an image $(\tilde C',u')$ in $\bM\bC$ of the projection $\Ind_{H}^{G}(f)_{*}\mu(i'(X'))$ on $C$.  We now note that $\mu(i(X))= \Ind_{H}^{G}(f)_{*}\mu(i'(X'))$.
Consequently, there exists a unique unitary multiplier  isomorphism $v_{f,(C,\rho,\mu)} \colon \tilde C\to \tilde C'$ such that $v_{f,(C,\rho,\mu)}u=u'$.
The family $$v_{f} \coloneqq (v_{f,(C,\rho,\mu)})_{(C,\rho,\mu)\in \Ob(\bCgtsmc(\Ind_H^G(X)))} $$
is a unitary isomorphism of functors $v_{f} \colon f_{*}\phi_{X}\to \phi_{X'}\Ind_{H}^{G}(f)_{*}$ filling the diagram
$$\xymatrix{
\bCgtsmc(\Ind_H^G(X))\ar@{}[drr]^-{\Rightarrow,v_{f}}\ar[rr]^{\Ind_{H}^{G}(f)_{*}}\ar[d]^{\phi_{X}} && \bCgtsmc(\Ind_H^G(X'))\ar[d]^{\phi_{X'}}\\
\bChtsmc(X)\ar[rr]^{f_{*}} && \bChtsmc(X')
}$$
If $f' \colon X'\to X''$ is a second morphism in $H\BC$, then using the uniqueness of the choices of   the unitary isomorphisms $v_{f,(C,\rho,\mu)}$, etc.\ one checks that
\[
(v_{f'}\circ \Ind_{H}^{G}(f)_{*}) (f'_{*}\circ v_{f})=v_{f'f}\,.\qedhere
\]
\end{proof}

\begin{rem}\label{gwoptjgopgergwerg}
One could explicitly  strictify {the} functor $\phi$ to a natural transformation of $\Ccat$-valued functors as follows. In a first step one chooses for every object $C$ in $\bC$ and every projection $p$ on $C$ a representative $(\tilde C,u)$ of the image of $p$.
In the construction of $\phi_{X}$ we then  use this choice for the mage of $\mu(i(X))$.

If we proceed in this way, then the transformations $v_{f}$ for morphisms $f$ in $H\BC$ all become identities.
\hB
\end{rem}

Let $$\Res^{G}_{H} \colon G\BC\to H\BC$$ be the functor which sends a $G$-bornological coarse space to the $H$-bornological coarse space obtained by restriction of the action from $G$ to $H$.
We have a canonical natural transformation
\begin{equation}\label{adfadsffwqfefqewfqwef}
\psi \colon \bCgtsmc\to \bChtsmc\circ \Res^{G}_{H}
\end{equation}
of functors from $G\BC$ to $\Ccat$. Its evaluation $\psi_{X}$   on $X$ in $G\BC$ is given as follows:
\begin{enumerate}
\item objects: The functor $\psi_{X}$ sends the object $(C,\rho,\mu)$ of $\bCgtsmc (X)$  with $\rho=(\rho_{g})_{g\in G}$ to the object $(C,\rho_{|H},\mu)$ of $ \bChtsmc( \Res^{G}_{H}(X))$, where 
$\rho_{|H} \coloneqq (\rho_{h})_{h\in H}$.
\item morphisms: The functor $\psi_{X}$ is the canonical inclusion of $G$-invariant into $H$-invariant morphisms.
\end{enumerate}

 Assume that $X$ is in $G\BC$.  Then we have an isomorphism
$$c_{X} \colon \Ind_{H}^{G}(\Res^{G}_{H}(X))\to (G/H)_{min,min}\otimes X$$
in $G\BC$
given by  the map of underlying sets
$$G\times_{H}X\to G/H \times X\, , \quad [g,x]\mapsto (g,gx)\, .$$
 The family
$(c_{X})_{X\in G\BC}$ is a natural isomorphism of functors
\begin{equation}\label{dsfbsdbkmlsdfbsdfb}
c:\Ind_{H}^{G}\circ \Res^{G}_{H}\to (G/H)_{min,min}\otimes - \colon G\BC\to G\BC\, .
\end{equation} 
We note that the projection $p_{X} \colon (G/H)_{min,min}\otimes X\to X$ is a bounded covering (Definition \ref{wergijowergergrefwef}).
The family $p \colon (p_{X})_{X\in G\BC}$ is a natural transformation of functors  \begin{equation}\label{adsogijaoisdgasga}
p \colon (G/H)_{min,min}\otimes \iota\to \iota\ ,
\end{equation}
where $\iota \colon G\BC\to G\BC_{\tr}$ is as in \eqref{ergegwergewrgewrgefwefwerf}.

Recall that $\bC$ is in $\Fun(BG,\nCcat)$. 
We  assume that   $  \bC$   admits all very small AV-sums. Then the pull-back $p_{X}^{*}$ is defined by Theorem \ref{ethgiuwieogfergwgrergweg}.
We consider the composition 
\begin{align*}
\mathclap{
\bCgtsmc(X) \xrightarrow{p_{X}^{*}} \bCgtsmc( (G/H)_{min,min}\otimes X)\stackrel{c_{X,*}^{-1}}{\cong}  \bCgtsmc(\Ind_{H}^{G}(\Res^{G}_{H}(X)))\stackrel{\phi_{{\Res^{G}_{H}(X)}}}{\simeq}  \bChtsmc(\Res^{G}_{H}(X))
}
\end{align*}
which we will denote by $\tilde \psi_{X}$. The family
$\tilde \psi \coloneqq (\tilde \psi_{X})_{X\in G\BC}$ is a natural transformation 
\begin{equation}\label{sdvsdvavfwfqwefewqf}
\tilde \psi \colon \bCgtsmc\to  \bChtsmc\circ \Res^{G}_{H}
\end{equation} 
 of functors from $G\BC$ to $\Ccat_{2,1}$.
 We must again consider values in the $(2,1)$-category $\Cat_{2,1}$ since the naturality of $\tilde \psi$ is only up to canonical unitary   isomorphisms instead of equalities. Even if we would strictify $\phi$ as in Remark \ref{gwoptjgopgergwerg}
 we would still get non-trivial fillers of the naturality squares from the $2$-categorial nature of $p$.

 The following proposition compares $\psi$ in \eqref{adfadsffwqfefqewfqwef} with $\tilde \psi$ in  \eqref{sdvsdvavfwfqwefewqf}.
\begin{prop}\label{rgoijwtogwerergwgregwr}
There exists a unitary isomorphism $u \colon \psi\to \tilde \psi$ of natural transformations from $\bCgtsmc$ to $\bChtsmc\circ \Res^{G}_{H}$.
\end{prop}

\begin{proof}
For $X$ in $G\BC$ we must construct a unitary isomorphism
$$u_{X} \colon \psi_{X}\to \tilde \psi_{X}\, .$$
Let $(C,\rho,\mu)$ be in $\bCgtsmc(X)$. Then we must construct an isomorphism
$$u_{X,(C,\rho,\mu)} \colon \psi_{X}(C,\rho,\mu)\to \tilde \psi_{X}(C,\rho,\mu)\, .$$
In order to avoid lengthy formulas we explain the object  $\tilde \psi_{X}(C,\rho,\mu)
$ in words.  In the first step of the construction of $p_{X}^{*}$ in the proof  of Theorem \ref{ethgiuwieogfergwgrergweg} for every coarse component $Y$ of $X$ we have chosen an image
$(C(Y), u_{Y})$ of $\mu(Y)$ in $\bM\bC$.  
  By Lemma \ref{eriogjqwefqewfewfeqdewdq}  the object $C$ is isomorphic to the orthogonal AV-sum of the family $(C(Y))_{Y\in  \pi_{0}(X)}$.  The underlying object of $p_{X}^{*}(C,\rho,\mu)$ is
given by an orthogonal  AV-sum over   the objects $C(p_{X}(Z))$ indexed by the set of coarse components $Z$ of $(G/H)_{min,min}\otimes X$.    The subsum over the components in $[e]\times  X$ is canonically unitarily isomorphic in $\bM\bC$ to $C$. Since $c_{X}^{-1}([e]\times  X)=i(X)$ we conclude, using the details of the construction of $\phi_{X}$ in the proof of Proposition \ref{woigjwergregwregwerg}, that the underlying object of $\tilde \psi_{X}(C,\rho,\mu)$ is canonically unitarily isomorphic to
$C$.  
We then check that this isomorphism is compatible with the cocycles and measures and yields the desired unitary isomorphism $u_{X,(C,\rho,\mu)}$.

In all cases above the canonical unitary isomorphisms are determined uniquely by their compatibilty with the structure maps of the chosen images of projections.

One then checks that the family
$u_{X} \coloneqq (u_{X,(C,\rho,\mu)})_{(C,\rho,\mu)\in \Ob(\bCgtsmc)}$ is a unitary isomorphism of functors.

Let $f \colon X\to X'$ be a morphism in $G\BC$. 
Then  from the $2$-categorial naturality of $(p_{X}^{*})$ and $ \phi$
we get a unitary isomorphism
$$v_{f} \colon  \Res^{G}_{H}(f)_{*}\tilde \psi_{X}\to \tilde \psi_{X'}  f_{*}\, .$$
Then one finally checks, using the uniqueness of the unitary isomorphisms going into the constructions of $u_{X}$ and $v_{f}$, that the composition
 $$\xymatrix{\bCgtsmc(X) \ar@/^1cm/[rr]^{\psi_{X}} \ar@/^-1cm/[rr]^{\tilde \psi_{X}}\ar[dd]^{f_{*}}& \Downarrow u_{X}&\bChtsmc(\Res^{G}_{H}(X))\ar[dd]^{\Res^{G}_{H}(f)_{*}}\\&\Downarrow v_{f}&\\\bCgtsmc(X')\ar@/^1cm/[rr]_{\tilde \psi_{X'}} \ar@/^-1cm/[rr]_{\psi_{X'}}&\Downarrow u_{X'}^{-1}&\bChtsmc(\Res^{G}_{H}(X'))}$$
 of isomorphisms of functors is the identity.
\end{proof}

Let $\Homol \colon \nCcat\to \bM$ be a finitary homological functor.

\begin{kor}\label{refwerferfrwergwreg}
Assume that  $\bC$ admits all very small AV-sums. 
\begin{enumerate}
\item\label{oiejgwoegergwerg} We have a natural equivalence
$\Homol \bC\cX_{c}^{G}\circ \Ind_{H}^{G}\to \Homol \bC\cX^{H}_{c}$ of $H$-equivariant  coarse homology theories.
\item\label{regejwgoijoergwegwreg} We have a natural transformation $\Homol \bC\cX^{G}_{c}\to \Homol \bC\cX_{c}^{H}\circ  \Res^{G}_{H}$ of $G$-equivariant coarse homology theories induced by $\psi$ in \eqref{adfadsffwqfefqewfqwef}.

\end{enumerate}

\end{kor}
\begin{proof}
We use Theorem  \ref{ergoiegwergewrgwergergwer} in order to conclude that
$\Homol \bC\cX_{c}^{G}$ and $\Homol \bC\cX_{c}^{H}$ are $G$- and $H$-equivariant coarse homology theories, respectively.
The composition $\Homol \bC\cX_{c}^{G}\circ \Ind_{H}^{G}$ is an $H$-equivariant coarse homology theory by \cite[Lem.\ 4.19]{desc} (which settles the universal case).
The  natural equivalence  in Assertion \ref{oiejgwoegergwerg} is obtained by applying the extension ${\Homol}_{1,2}$ (see \eqref{fwevpoiejvopevececqwec}) to the transformation $\phi$  obtained in Proposition \ref{woigjwergregwregwerg}. 
\end{proof}

\begin{rem}
Assume that $E^{G} \colon G\BC\to \bM$ is a $G$-equivariant coarse homology theory. Then  by \cite[Lem.\ 4.19]{desc} we can define an $H$-equivariant coarse homology theory by
$  E^{G}\circ \Ind_{H}^{G}:H\BC\to \bM$. 

If $E^{G}$ admits transfers, then we can furthermore define a natural transformation
\begin{equation}\label{3rfjoifqwefqwfqwf}
  E^{G}\to E^{G}\circ \Ind_{H}^{G} \circ \Res^{G}_{H}
\end{equation}
 as the composition
 $$E^{G}\xrightarrow{p^{*}\circ c_{*}^{-1}} E^{G}\circ \Ind^{G}_{H}\circ \Res^{G}_{H}
 \, ,$$
 where $p$ and $c$ are as in \eqref{adsogijaoisdgasga} and  \eqref{dsfbsdbkmlsdfbsdfb}.

 For 
 $  \bC$  in $\Fun(BG, \nCcat)$ admitting all very small AV-sums
 we have 
 the two equivariant coarse homology theories
 $\Homol \bC\cX^{G}_{c}  $ and
 $\Homol \bC\cX^{H}_{c} $ which are defined independently from each other. By Corollary \ref{refwerferfrwergwreg}.\ref{oiejgwoegergwerg}
  we know that the transformation $\phi$ provides an equivalence of  
    $H$-equivariant coarse homology theories
 \begin{equation}\label{ewfqwfewfwqfqwef4}
\Homol \bC\cX^{G}_{c}\circ \Ind_{H}^{G}\xrightarrow{\Homol(\phi),\simeq} \Homol \bC\cX^{H}_{c}\, .
\end{equation}
Consequently, we can consider the natural transformation \begin{equation}\label{wefoijqoiwefqwefe}
r\colon \Homol \bC\cX^{G}_{c}\xrightarrow{\eqref{3rfjoifqwefqwfqwf}}    \Homol \bC\cX^{G}_{c}   \circ \Ind^{G}_{H}\circ \Res^{G}_{H} \stackrel{\eqref{ewfqwfewfwqfqwef4}}{\simeq}       \Homol \bC\cX^{H}_{c}\circ \Res^{G}_{H}\, .
\end{equation} 
The combination of Proposition \ref{rgoijwtogwerergwgregwr} and Corollary \ref{refwerferfrwergwreg}.\ref{regejwgoijoergwegwreg}  then says that  
  the abstract transformation \eqref{wefoijqoiwefqwefe}
is induced by the obvious inclusion  of $C^{*}$-categories $\psi$ in \eqref{adfadsffwqfefqewfqwef}. This fact will be further employed in \cite{bl}.
\hB
\end{rem}

\section{Strong additivity of \texorpdfstring{$K\bC\cX_{c}^{G}$}{KCXGc}}\label{3qrgiuhiurlferfvrefvfs}

Recall from Definition \ref{ergoijerwogergerfwfref9} {the notion} of a strongly additive functor on $G\BC$.
Recall that the topological $K$-theory functor (see \cite[Def.\ 14.3]{cank})
$$\Kcat\colon \nCcat\to \Sp$$
is a finitary homological functor  \cite[Thm.\ 14.4]{cank}. 
For the purpose of the present section its main additional property is
that it  preserves arbitrary very small products of finitely additive $C^{*}$-categories \cite[Thm.~15.7]{cank} .


%
%
%
%

Assume that $  \bC$ in $\Fun(BG, \nCcat)$ is countably and effectively additive.
By Theorem \ref{ergoiegwergewrgwergergwer} we have a continuous equivariant coarse homology theory
$$\Kcat \bC\cX_{c}^{G}\colon G\BC\to \Sp\, .$$

\begin{theorem}\label{rhrherthertgergertg}
If $\bC$ admits all very small AV-sums, then 
the functor   $\Kcat \bC\cX_{c}^{G}$ is  strongly additive.
\end{theorem}
 
{The proof of Theorem~\ref{rhrherthertgergertg} will be given further below. We first start with the corresponding additivity statement for the functor $\bCgtsmc$.}

Let $(X_{i})_{i\in I}$ be a very small family in $G\BC$. Recall Definition \ref{ergoiejoqfqefweqfqwef} of the free union $\bigsqcup^{\free}_{i\in I}X_{i}$ of this family. For $i$ in $I$ let $f_{i} \colon X_{i}\to \bigsqcup^{\free}_{i\in I}X_{i}$ denote the inclusion.
In the formulation of the following lemma the symbol $0$ stands for a functor which sends every object to a zero object.
\begin{lem}\label{rgqoijfoewfewfqweff}  
 If $\bC$ admits all very small AV-sums, then  there exists a unitary equivalence
 $$\phi \colon \bCgtsmc\big(\bigsqcup^{\free}_{i\in I}X_{i}\big) \to \prod_{i\in I}\bCgtsmc(X_{i})$$ in $\Ccat$ such that for all $i,j$ in $I$ we have a unitary isomorphism
\begin{equation}\label{eq_additivity_functor}
\pr_{j}\circ \phi\circ f_{i,*}\cong \begin{cases}\id_{\bCgtsmc(X_{i})}&\text{if }i=j\,,\\0&\text{if }j\not=i\,.\end{cases}
\end{equation}
\end{lem}
\begin{proof}
We will construct functors
$$\phi\colon \bCgtsmc\big(\bigsqcup^{\free}_{i\in I}X_{i}\big) \to \prod_{i\in I}\bCgtsmc(X_{i})\, , \quad \psi\colon \prod_{i\in I}\bCgtsmc(X_{i})\to\bCgtsmc\big(\bigsqcup^{\free}_{i\in I}X_{i}\big)$$
and then show that they are  inverse to each other up to unitary isomorphism.

We start with the construction of $\phi$.
\begin{enumerate}
\item objects: Let $(C,\rho,\mu)$ be an object of $\bCgtsmc(\bigsqcup^{\free}_{i\in I}X_{i})$. By Lemma \ref{wtigowgfrefrewf}  for every $i$ in $I$ we choose an image $(C_{i},u_{i})$ of the projection $\mu(X_{i})$. We further define $\rho_{i} = ( \rho_{g,X_{i}})_{g\in G}$ as in \eqref{refrewhiu3hfiuherifwerfrfw}, where $\rho_{g,X_{i}} \colon C_{i}\to gC_{i}$ (since $X_{i}$ is $G$-invariant), and  $\mu_{i} \coloneqq \mu_{X_{i}}$ by \eqref{fewveoivhroivjoieverwcecwec}. In this way we get an object
$$\phi(C,\rho,\mu) \coloneqq ((C_{i},\rho_{i},\mu_{i}))_{i\in I}$$ of $ \prod_{i\in I}\bCgtsmc(X_{i})$.
\item morphisms: Let $A \colon (C,\rho,\mu)\to (C',\rho',\mu')$  be a morphism in $\bCgtsmc(\bigsqcup^{\free}_{i\in I}X_{i})$. Note that 
the subsets $X_{i}$ of the free union are mutually coarsely disjoint. Therefore $A$ does not propagate between different components. 
For every $i$ in $I$ we consider the restriction $A_{i} \coloneqq u_{i}^{*}Au_{i}$ of $A$ to $C_{i}$. Then we have $\sup_{i\in I}\|A_{i}\|=\|A\|$. We define the morphism $\phi(A)$ in $\prod_{i\in I}\bCgtsmc(X_{i})$ by $$\phi(A) \coloneqq (A_{i})_{i\in I}\colon ((C_{i},\rho_{i},\mu_{i}))_{i\in I}\to ((C'_{i},\rho'_{i},\mu'_{i}))_{i\in I}\, .$$   
\end{enumerate}
One checks in a straightforward manner that this construction defines a  functor $\phi$  in $\Ccat$.  
 Note that it depends on the choices made for the images of the projection $\mu$.


We now define the functor $\psi$.
\begin{enumerate}
\item objects: Let  $((C_{i},\rho_{i},\mu_{i}))_{i\in I}$ be an object of  $\prod_{i\in I}\bCgtsm(X_{i})$. Since we assume that  $\bC$ admits all very small orthogonal AV-sums  we can choose an orthogonal AV-sum
$(\bigoplus_{i\in I}C_{i}, (e_{i})_{i\in I})$. Using 
  \cite[Lem. 7.8]{cank}  we  obtaine the last two entries of the right-hand side in
$$\psi(((C_{i},\rho_{i},\mu_{i}))_{i\in I}) \coloneqq (\bigoplus_{i\in I}C_{i},\oplus_{i\in I}\rho_{i} ,\oplus_{i\in I} \mu_{i})\, .$$
One checks again in a straightforward manner that this object of $\bCgtsmc(\bigsqcup^{\free}_{i\in I}X_{i})$ is well-defined.
\item morphisms: Let  $$(A_{i})_{i\in I} \colon ((C_{i},\rho_{i},\mu_{i}))_{i\in I}\to ((C'_{i},\rho'_{i},\mu'_{i}))_{i\in I}$$ be a morphism in
$\prod_{i\in I}\bCgtsm(X_{i})$. Since morphisms in a 
 product of $C^{*}$-categories are precisely the uniformly bounded families of morphisms 
  we have $\sup_{i\in I}\|A_{i}\|<\infty$ and can use \cite[Lem. 7.8]{cank}  to define
a morphism $$\oplus_{i\in I}A_{i} \colon \bigoplus_{i\in I}C_{i}\to \bigoplus_{i\in I}C'_{i}$$ in $\bM\bC$.
One checks that it is compatible with the cocycles, i.e.,    $\oplus_{i\in I}A_{i}$ is a morphism in ${(\bM\bC)}^{G} (\bigsqcup^{\free}_{i\in I}X_{i})$.
It remains to show that it  is a morphism in the Roe category.  Fix $\varepsilon$ in $(0,\infty)$. Then for every $i$ in $I$ there exists an entourage $U_{i}$ of $X_{i}$ and a $U_{i}$-controlled operator $A_{i}' \colon (C_{i},\rho_{i},\mu_{i})\to (C'_{i},\rho_{i}',\mu_{i}')$ in $ \bCgtsmc(X_{i})$ such that $\|A_{i}-A_{i}'\|\le \varepsilon$. Then $\oplus_{i\in I}A_{i}'$ is in $ \bCgtsmc (\bigsqcup^{\free}_{i\in I}X_{i})$. It is   controlled by the entourage $\bigcup_{i\in I} U_{i}$ in the coarse structure of the free union, and we have $\|\oplus_{i\in I}A_{i}-\oplus_{i\in I}A_{i}'\|\le \varepsilon$. Since we can choose $\varepsilon$ arbitrary small we conclude that
$\oplus_{i\in I}A_{i}$ is a morphism in $\bCgtsmc(\bigsqcup^{\free}_{i\in I}X_{i})$.
\end{enumerate}
One checks in a straightforward manner that this construction defines a functor $\psi$  in $\Ccat$. It depends on the choices involved in the sums.

We now show that $\phi$ and $\psi$ are mutually inverse to each other.

\begin{enumerate}
\item {The unitary} $u \colon \phi\circ \psi \xrightarrow{\cong} \id$: Let $((C_{i},\rho_{i},\mu_{i}))_{i\in I}$ be an object of  $\prod_{i\in I}\bCgtsmc(X_{i})$.
For every $j$ in $I$  let   $((\bigoplus_{i\in I} C_{i})_{j},u'_{j})$  be the choice of an image of $(\oplus_{i\in I}\mu_{i})(X_{j})$  in $\bM\bC$ adopted in the construction of $\phi$. We define
  the multiplier morphism
$$a_{j} \coloneqq e_{j}^{*} u'_{j} \colon (\bigoplus_{i\in I} C_{i})_{j}\to C_{j}\, .$$  
We get a unitary multiplier 
$$u_{((C_{i},\rho_{i},\mu_{i}))_{i\in I}} \coloneqq (a_{j})_{j\in I}\colon ((\bigoplus_{i\in I} C_{i})_{j})_{j\in I} \to (C_{j})_{j\in I}\, .$$
The family $ (u_{((C_{i},\rho_{i},\mu_{i}))_{i\in I}})_{((C_{i},\rho_{i},\mu_{i}))_{i\in I}\in \Ob(\prod_{i\in I}\bCgtsmc(X_{i}))}$
is the desired unitary isomorphism of functors.
\item {The unitary} $v\colon\psi\circ \phi  \stackrel{\cong}{\to}  \id    $: Let   $(C,\rho,\mu)$ be an object of $\bCgtsmc(\bigsqcup^{\free}_{i\in I}X_{i})$. 
 We  use Lemma \ref{eriogjqwefqewfewfeqdewdq} in order to conclude that  
$$v_{(C,\rho,\mu)} \coloneqq \sum_{i\in I} u_{i}e_{i}^{*} \colon \bigoplus_{i\in I}C_{i} \to C\, .$$ 
 is a unitary multiplier.
Then the family $(v_{(C,\rho,\mu)})_{(C,\rho,\mu)\in \Ob(\bCgtsmc(\bigsqcup^{\free}_{i\in I}X_{i}))}$ is the desired unitary isomorphism of functors.
\end{enumerate}

We now show that $\phi$ has the asserted property \eqref{eq_additivity_functor}.
Let $i$ be in $I$. We construct the unitary isomorphism
$$w \colon \pr_{i} \circ \phi \circ f_{i,*}\xrightarrow{\cong} \id_{\bCgtsmc(X_{i})}$$
{as follows.} Let $(C,\rho,\mu)$ be an object in $\bCgtsmc(X_{i})$. Then  the multiplier 
$u_{i} \colon (f_{*}C)_{i}\to C$ is a unitary, where $((f_{*}C)_{i},u_{i})$ denotes the choice of the image 
of $(f_{i,*}\mu)(X_{i})$  in $\bM\bC$ involved in the construction of $\phi$.  
We set $w_{(C,\rho,\mu)} \coloneqq u_{i}
$.
The family $(w_{(C,\rho,\mu)})_{(C,\rho,\mu)\in \Ob(\bCgtsmc(X_{i}))}$ is the desired unitary isomorphism.

If $i,j$ are in $I$ and $i\not=j$, then  $ (f_{i,*}C)_{j}=0$. Consequently, $\pr_{j} \circ \phi \circ f_{i,*}\cong 0$.
\end{proof}

\begin{proof}[{Proof of Theorem \ref{rhrherthertgergertg}}]
Since $\Kcat \bC\cX^{G}_{c}$ is excisive, it is $\pi_{0}$-excisive (Definition~\ref{qerogfjqeropfewfqwffe}).
Let $(X_{i})_{i\in I}$ be a family in $G\BC$. For $j$ in $J$ we consider the  invariant coarsely disjoint partition
$(X_{j},\bigcup_{i\in I\setminus\{j\}} X_{i})$ of $\bigsqcup^{\free}_{i\in I}X_{i}$.
Since $\Kcat\bC\cX^{G}_{c}$ is $\pi_{0}$-excisive we get a decomposition
$$\Kcat\bC\cX^{G}_{c}\big(\bigsqcup^{\free}_{i\in I}X_{i}\big) \simeq \Kcat\bC\cX^{G}_{c}(X_{j})\oplus  \Kcat\bC\cX^{G}_{c}\big(\bigsqcup^{\free}_{i\in I\setminus\{j\}}X_{i}\big)\,.$$
We let
$$p_{j}\colon \Kcat\bC\cX^{G}_{c}\big(\bigsqcup^{\free}_{i\in I}X_{i}\big) \to  \Kcat\bC\cX^{G}_{c}(X_{j})$$ be the {corresponding} projection onto the first summand.
We must show that
$$(p_{j})_{j\in J}\colon \Kcat\bC\cX^{G}_{c}\big(\bigsqcup^{\free}_{i\in I}X_{i}\big)\to \prod_{j\in I} \Kcat\bC\cX^{G}_{c}(X_{j})$$
is an equivalence.
We consider the diagram
$$\xymatrix{\Kcat\bC\cX^{G}_{c}(\bigsqcup^{\free}_{i\in I}X_{i})\ar[rr]^{(p_{i})_{i\in J}}\ar@{=}[d]^{\text{defn.}}&& \prod_{i\in I} \Kcat\bC\cX^{G}_{c}(X_{i})\ar@{=}[d]^{\text{defn.}}\\
\Kcat(\bCgtsmc(\bigsqcup^{\free}_{i\in I}X_{i}))\ar[dr]^{\quad\simeq,\Kcat(\phi)}_{\text{Lemma}~\ref{rgqoijfoewfewfqweff}\qquad}&&\prod_{i\in I}\Kcat(\bCgtsmc (X_{i}))\ar[dd]^{\pr_{k}}\\
&\Kcat(\prod_{j\in I}\bCgtsmc( X_{j})) \ar[dr]^{\quad\Kcat(\pr_{k})} \ar[ur]^{\simeq}_{!}&\\
\Kcat(\bCgtsmc(X_{j}))\ar[uu]^{\Kcat(f_{j,*})}  \ar[rr]^{\delta_{jk}\id}&&\Kcat(\bCgtsmc(X_{k}))}$$
where the marked equivalence together with commutativity of the right triangle for every $k$ in $I$ are stated in \cite[Thm.~15.7]{cank}, and commutativity of the lower left square is asserted by \eqref{eq_additivity_functor}. The big outer square commutes for every $j,k$ in $J$ in view of the construction of the map $p_{k}$. This implies that the middle pentagon commutes as well.
This finally implies the assertion of the theorem.
\end{proof}

\section{CP-functors}
\label{sec_CPfunctors}

Let $G$ be a group and let $G\Orb$ denote its orbit category.
In this section we describe  a functor $G\Orb\to \Sp$   associated to  $ \bC$ in
$\Fun(BG, \nCcat)$ admitting all very small AV-sums.
The main theorem states that this functor is a CP-functor.  

 Since $\bC^{u}$ is in particular finitely additive we have  the functor
{$$ \bC^{u} [-]\rtimes_{r}G \colon G\Set\to \Ccat$$ from} \eqref{qwefktbvrtfavkoavd}.
We can consider the functor
\begin{equation}\label{qrwfnofdewffeqewfefeqwfqfwef}
\Kcat( \bC^{u} [-]\rtimes_{r}G) \colon G\Set \to \Sp\, .
\end{equation}
\begin{ddd}\label{qergijoqergwrefweqfwqefwedewdqewd}
We define $\Kcat \bC_{G,r} \colon G\Orb\to \Sp$ to be the restriction of \eqref{qrwfnofdewffeqewfefeqwfqfwef} to the orbit category.
\end{ddd}

Before we can state the main theorem about $\Kcat \bC_{G,r}$ we must recall the notion of a CP-functor.
We have the  functor
$$i \colon G\Orb\to G\BC\,, \quad S\mapsto S_{min,max}$$
(see Example \ref{qrgioqjrgoqrqfewfeqfqewfe}).
We consider a functor $$M \colon G\Orb\to \bM\, .$$
\begin{ddd}[{\cite[Defn.~1.8]{desc}}]\label{rgioqrgfwrefwefwerfwer}
$M$ is a CP-functor if
\begin{enumerate}
\item \label{werhijowegrefrefwfrf}$\bM$ is stable, admits all very small colimits and limits, and  is compactly generated.
\item There exists an equivariant coarse homology theory $E \colon G\BC\to \bM$ (see Definition \ref{wefuihqfwefwffqwefefwq})
such that:
\begin{enumerate}
\item\label{qwrijogfqreefvfv} $M$ is equivalent to $E_{G_{can,min}}\circ i$ (see Examples \ref{regiuehrifwefqfewffqewedqe} \& \ref{etwgokergpoergegregegwergrg}).
\item\label{qwrijogfqreefvfv1}
\begin{enumerate}
\item\label{wegiowjegerfwefwerfwerfwrf} $E$ is continuous (Definition \ref{weifhqewiefjeefqefwefqwef}).
\item $E$ is strongly additive (Definition \ref{ergoijerwogergerfwfref9}).
\item $E$ admits transfers (Definition \ref{poqjkpowckwdcqwcqwcwc}).
\end{enumerate}
\end{enumerate}
\end{enumerate}
\end{ddd}

\begin{theorem}\label{wergiojerfqerwfqwefwefewfqef}
If $  \bC$ in $\Fun(BG,\nCcat)$ admits all very small AV-sums, then 
 the functor $\Kcat \bC_{G,r}$ is a CP-functor.
\end{theorem}
\begin{proof}
The target category of the functor $\Kcat \bC_{G,r}$ is the large category $\Sp$ of small spectra. This category is stable, compactly generated and admits all   small colimits and limits. This verifies Condition \ref{rgioqrgfwrefwefwerfwer}.\ref{werhijowegrefrefwfrf}.

For $E$ in Definition  \ref{rgioqrgfwrefwefwerfwer}  we take the equivariant coarse homology theory
$\Kcat \bC\cX^{G}_{c}$ obtained by specializing Definition \ref{qwregojwoergwergerwffwfwerfwref}.\ref{rgoijweorgwerferwfwerfwerfw} to the case $\Homol=\Kcat$.

By applying $\Kcat$ to the equivalence stated in Proposition \ref{qergioqefweqwecqcasdc}.\ref{egijogrgergewrgwegwrewg2} we  obtain an equivalence 
$$\Kcat \bC\cX^{G}_{c,G_{can,{min}}}\circ i\simeq  \Kcat \bC_{G,r}\, .$$
This verifies Condition \ref{rgioqrgfwrefwefwerfwer}.\ref{qwrijogfqreefvfv}.

The coarse homology theory   $\Kcat \bC\cX^{G}_{c}$ has the  properties required in \ref{rgioqrgfwrefwefwerfwer}.\ref{qwrijogfqreefvfv1}:
\begin{enumerate}
\item It is continuous by Theorem \ref{ergoiegwergewrgwergergwer}.
\item It is strongly additive by Theorem \ref{rhrherthertgergertg}.
\item  It admits transfers by Theorem \ref{qwrgiherogwefewfqewfqwefqwfewf} which is applicable for $\Homol=\Kcat$ since $\Kcat$, being a homological functor, sends unitary equivalences to equivalences.\qedhere
\end{enumerate}
\end{proof}

\bibliographystyle{alpha}
\bibliography{forschung}

\end{document}